\documentclass[3p,longnamesfirst,sort&compress]{elsarticle}

\usepackage{amsmath,amsthm,amssymb,mathrsfs}
\usepackage{thmtools,thm-restate}
\usepackage{enumerate}
\usepackage{slashed}
\usepackage{txfonts,dsfont}
\usepackage[breaklinks,colorlinks]{hyperref}
\usepackage{nameref}
\usepackage{cleveref}

\makeatletter
\newcommand{\autorefcheckize}[1]{%
  \expandafter\let\csname @@\string#1\endcsname#1%
  \expandafter\DeclareRobustCommand\csname relax\string#1\endcsname[1]{%
    \csname @@\string#1\endcsname{##1}\wrtusdrf{##1}}%
  \expandafter\let\expandafter#1\csname relax\string#1\endcsname
}
\makeatother

\declaretheorem[numberwithin=section]{theorem}
\declaretheorem[sibling=theorem, name=Lemma]{lem}

\declaretheorem[sibling=theorem, name=Remark]{rem}

\numberwithin{equation}{section}

\newcommand{\norm}[1]{\left\lVert#1\right\rVert}
\newcommand{\abs}[1]{\left\lvert#1\right\rvert}
\newcommand{\set}[1]{\left\{#1\right\}}
\newcommand{\hin}[2]{\left\langle#1,#2\right\rangle}

\newcommand*{\To}{\longrightarrow}
\newcommand*{\Rmn}[1]{\uppercase\expandafter{\romannumeral#1}}
\newcommand*{\dif}{\mathop{}\!\mathrm{d}}

\journal{}

\allowdisplaybreaks

\begin{document}

\begin{frontmatter}

\title{A flow approach to the Toda system\tnoteref{LSY}}
    
\author[cqut]{Yong Luo}
\ead{yongluo-math@cqut.edu.cn}
    \address[cqut]{Mathematical Science Research Center of Mathematics, Chongqing University of Technology, Chongqing, 400054, China}

\author[xtu]{Linlin Sun
}
    \ead{sunll@xtu.edu.cn}
    \address[xtu]{School of Mathematics and Computational Science \& Hunan Research Center of the Basic Discipline Fundamental Algorithmic Theory and Novel Computational Methods, Xiangtan University, Xiangtan, 411105, China}

\author[uni-freiburg]{Guofang Wang}
\ead{guofang.wang@math.uni-freiburg.de}
    \address[uni-freiburg]{Inst of Math., University of Freiburg, Ernst-Zermelo-Str.1, Freiburg, 79104, Germany}



\tnotetext[LSY]{Y. Luo acknowledges support from the National Natural Science Foundation of China (Grant Nos. 12271069, 12571055), L. Sun acknowledges support from the National Natural Science Foundation of China (Grant No. 12571055), partially by the 111 Project (No. D23017), as well as the Program for Science and Technology Innovative Research Team in Higher Educational Institutions of Hunan Province, China.}

\begin{abstract}

In this paper we introduce a flow  to study the Toda system, which we call {\it Toda flow.} 
More generally, we introduce a flow of the Liouville systems, formulated as a coupled parabolic system with nonlocal interactions. 
Finite-time singularities are characterized and both necessary and sufficient conditions for convergence are provided in this general setting,
even when the prescribed functions are allowed to change sign.

As an application, we prove a global existence for  the  Toda flow in the critical case without restricting the sign of the prescribed functions. 
We provide a detailed description of blow-up behavior at infinity and obtain a sharp lower bound for the functional in cases where global convergence fails. By constructing appropriate test functions, we further establish a sufficient condition for the global convergence of the flow. These results are not affected by the sign-changing nature of the prescribed functions, and  extend the theorem of Jost, Lin and Wang (Comm. Pure Appl. Math. 59, 526–558, 2006) to systems of multiple equations under this more general and physically relevant condition.
\end{abstract}
\begin{keyword}
Toda flow \sep global existence \sep global convergence \sep blowup analysis

 \MSC[2020] 35B33 \sep 58J35

\end{keyword}

\end{frontmatter}

\section{Introduction}


The $\mathrm{SU}(N+1)$ Toda system occupies a special place in geometric analysis as the most algebraically structured generalization of the Liouville equation. Let $\Sigma$ be a closed Riemannian surface, with Laplace operator $\Delta$ and volume form $\dif\mu_{\Sigma}$.  Let $\rho = \left(\rho_1,\dots,\rho_N\right) \in \mathbb{R}^N$,  
$h = \left(h_1,\dots,h_N\right) \in C^\infty(\Sigma, \mathbb{R}^N)$ and $Q = (Q_1,\dots,Q_N) \in C^\infty(\Sigma, \mathbb{R}^N)$, where 
\begin{align*}
    \rho_i=\int_{\Sigma}Q_i\dif\mu_{\Sigma},\quad i\in I\coloneqq\set{1,\dots,N}.
\end{align*}
 The $\mathrm{SU}(N+1)$ Toda system takes the form
\begin{align}\label{eq:toda_system}
-\Delta u_i=\sum_{j=1}^N a_{Nij}\left(\rho_j\dfrac{h_je^{u_j}}{\int_\Sigma h_je^{u_j}d\mu}-Q_j\right),\quad i\in I,
\end{align}
where $A_N=\left(a_{Nij}\right)_{N\times N}$ is the Cartan matrix of $\mathrm{SU}(N+1)$ given by
\begin{align*}
A_N=\begin{pmatrix}2&-1&0&\hdots&\hdots&\hdots&0\\
-1&2&-1&0&\hdots&\hdots&0\\
0&-1&2&-1&0&\ddots&0\\
\vdots&\ddots&\ddots&\ddots&\ddots&\ddots&\vdots\\
0&\ddots&0&-1&2&-1&0\\
0&\hdots&\hdots&0&-1&2&-1\\
0&\hdots&\hdots&\hdots&0&-1&2\\
\end{pmatrix}.
\end{align*}

\vspace{2ex}
If necessary, one can set $Q_j = \rho_j$ by appropriately choosing $h_j$. This is achieved by solving $\Delta v_j =\frac{\rho_j}{\abs{\Sigma}}-Q_j,\int_{\Sigma}v_j=0,$ (where $\abs{\Sigma}$ denotes the area of $\Sigma$), defining $\tilde v_j = \sum_i a_{N}^{ij}v_j$, and then setting $\tilde u_i = u_i - \tilde v_i$. The new functions $\tilde u_i$ thus satisfy a system where $Q_j = \rho_j$ and $\tilde h_j = h_j e^{\tilde v_j}$.
The solutions to the Toda system \eqref{eq:toda_system} are critical points of the following functional
\begin{align}\label{eq:Toda-Functional}
    J(u)\coloneqq\dfrac12\sum_{i,j=1}^Na_N^{ij}\int_{\Sigma}\hin{\nabla u_i}{\nabla u_j}\dif\mu_{\Sigma}+\sum_{i=1}^N\left(\int_{\Sigma}Q_iu_i\dif\mu_{\Sigma}-\rho_i\ln\int_{\Sigma}h_ie^{u_i}\dif\mu_{\Sigma}\right),    
\end{align} for $$ u=(u_1,\cdots, u_N)\in H^1\left(\Sigma,\mathbb{R}^N\right).$$
Here $A_N^{-1}=\left(a_N^{ij}\right)$ is the inverse matrix of $A_N$. 
When $N=1$, this reduces to the classical mean field equation, but for $N\geq2$ the off-diagonal entries introduce genuine coupling that fundamentally changes analysis.
\vspace{2ex}

In geometry, the system is closely related 
to the theory of holomorphic curves. In fact, it arises as a system satisfied by 
the Frenet frame associated to a holomorphic curve in $\mathbb{CP}^N$ (see e.g. \cite{BolWood97some, Cal53isometric, CheWol87harmonic}). This link provides a direct correspondence between solutions of the PDE and geometric objects, making the system a powerful analytical tool for studying problems in projective geometry, such as the existence and properties of curves with prescribed ramification. The classification of entire solutions for the $U(N+1)$ Toda system crucially uses this link 
\cite{JosWan02classification}.

\vspace{2ex}

Beyond geometry, the system arises in several distinct contexts, each providing motivation for its study. In theoretical physics, it appears as the self-duality equation for vortex configurations in non-abelian Chern-Simons gauge theories \cite{Tar08selfdual, Yan01solitons}. Here the functions $u_i$ relate to the magnitudes of the Higgs fields, and the parameters $\rho_i$ correspond to coupling constants. The physical interpretation helps to explain certain analytical features: the possible presence of singular sources represents vortex points where the wave function vanishes, while the coupling matrix reflects interactions between different types of particles.
\vspace{2ex}

Mathematically, what makes the Toda system particularly interesting is how the algebraic structure of the Cartan matrix recognizes  the analytical behavior. When studying sequences of solutions that may blow up, one finds that the local masses $\sigma_i$ at a concentration point must satisfy the following (local) Pohozaev identity (see \cite{JosWan01analytic,LinWeiZhang})
\begin{align*}
    \sum_{j=1}^N a_{Nij}\sigma_i\sigma_j = 2\sum_{i=1}^N \sigma_i.
\end{align*}
This forces the local masses $\sigma_i$ to satisfy integrality conditions, leading to quantization of energy concentration. 

\vspace{2ex}
This interplay between algebra and analysis distinguishes the Toda system from more general coupled equations. The specific form of the coupling, dictated by the Cartan matrix, imposes constraints that make the analysis tractable while still exhibiting rich phenomena. The system thus serves as a natural testing ground for developing methods that combine Lie theory with nonlinear PDE techniques, with insights that may apply to other geometrically motivated systems.

\subsection{Motivations} 

The Nirenberg problem of prescribing Gaussian curvature on the sphere is one of the central problems in geometric analysis. Its study over the past half century has shaped the development of nonlinear PDE theory in two dimensions (see e.g. \cite{Kou72gaussian,Mos73on,ChaYan88conformal,ChaYan87prescribing,Han90prescribing,And21nirenberg,Ji04on,ChenDin87scalar} ). The problem is simply stated: given a smooth function $K$ on the standard sphere $\left(\mathbb{S}^2, g_0\right)$, does there exist a function $u$ such that the conformal metric $g = e^{2u}g_0$ has Gaussian curvature $K$? This leads to the equation
\begin{align*}
   -\Delta_{g_0} u + 1 = K e^{2u}. 
\end{align*}
 Kazdan and Warner \cite{KazWar74curvature} observed that any solution must satisfy the condition
\[
\int_{\mathbb{S}^2} \langle \nabla K, \nabla x_i \rangle e^{2u}\dif\mu_{g_0} = 0, \quad i = 1,2,3,
\]
where $x_i$ are the coordinate functions of $\mathbb{S}^2 \subset \mathbb{R}^3$. This shows that the symmetry of the sphere imposes genuine constraints: not every smooth positive function can be a prescribed curvature.

\vspace{2ex}

The mean field equation represents a natural extension of the ideas developed for the Nirenberg problem to a broader setting.  
 The mean field equation also appears  in the abelian Chern–Simons–Higgs models (see e.g. \cite{CafYan95vortex,DinJosLiPenWan01self}).  On a closed Riemannian surface $\Sigma$, one considers the equation
\begin{align}\label{eq:mean_field}
    -\Delta u = \rho\left( \dfrac{h e^{u}}{\int_\Sigma h e^{u}\dif\mu_g} - \dfrac{1}{\abs{\Sigma}} \right),
\end{align}
where $\rho$ is a positive parameter and $h$ is a smooth positive function. This equation arises as the Euler-Lagrange equation for the variational functional
 \begin{align*}
     H^1\left(\Sigma\right)\ni u\mapsto J(u)\coloneqq\dfrac12\int_{\Sigma}\abs{\nabla u}^2\dif\mu_{\Sigma}+\rho\left(\bar u-\ln\int_{\Sigma}he^{u}\dif\mu_{\Sigma}\right),
 \end{align*}
 where $\bar{u} = \frac{1}{\abs{\Sigma}} \int_{\Sigma} u \dif\mu_{\Sigma}$ is the average of $u$ over $\Sigma$. Critical points of $J$ in the Sobolev space $H^{1}\left(\Sigma\right)$ correspond precisely to solutions of \eqref{eq:mean_field}.
 \vspace{2ex}
 
This generalization reveals new structure. The parameter $\rho$ introduces a continuum of problems, with different behavior appearing at certain critical values. The threshold $\rho = 8\pi$ is particularly important: for $\rho < 8\pi$, the existence theory follows from standard variational methods using the Moser-Trudinger inequality; at $\rho = 8\pi$ and its multiples, the analysis becomes more delicate, with solutions developing concentration phenomena similar to those observed in the Nirenberg problem.
\vspace{2ex}

The equation arises in several contexts. In statistical mechanics, it describes the equilibrium states of a system of point vortices in the mean field limit. Here $\rho$ plays the role of inverse temperature. In gauge theory, it governs the abelian Chern-Simons vortex equations, with $\rho$ corresponding to a coupling constant. These connections give the equation physical significance beyond its mathematical interest.

\vspace{2ex}

The existence of solutions of \eqref{eq:mean_field} has been widely studied in recent decades when the prescribing function $h$ is positive everywhere. Many partial existence results have been obtained for noncritical cases according to the Euler characteristic of $\Sigma$ (see for example \cite{BreMer91uniform,Li94blow-up,CheLin03topological,DinJosLiWan99existence,SunZhu24existence,LiSunYan23boundary,Mal08morse,Zhu24remark}). In particular, Chen and Lin \cite{CheLin02sharp} obtained an existence result for general critical cases (i.e., $\rho\in 8\pi\mathbb{N}^*$) which generalized Ding-Jost-Li-Wang's result \cite{DinJosLiWan97differential}. Moreover, Djadli \cite{Dja08existence} established the existence of solutions for all surfaces when $\rho\notin 8\pi\mathbb{N}^*$ by studying the topology of sublevels to achieve a min-max scheme which already introduced by Ding, Jost, Li and Wang \cite{DinJosLiWan99existence}, Djadli and Malchiodi in \cite{DjaMal08existence}. For sign-changing potential $h$, we refer the reader to \cite{CheLi08priori,MarLopRui18compactness,Han90prescribing} and references therein.

\vspace{2ex}

While the elliptic theory of the mean field equation has been extensively studied, a different perspective emerges when considering its time-dependent counterpart. If $h$ is a positive function, then the mean field equation \eqref{eq:mean_field} is in fact equivalent to the following PDE
 \begin{align}\label{eq:mean_field'}
     -\Delta u=\rho\dfrac{e^{u}}{\int_{\Sigma}e^u\dif\mu_{\Sigma}}-Q
 \end{align}
 where $\int_{\Sigma}Q\dif\mu_{\Sigma}=\rho$. 
The following evolution problem associated to \eqref{eq:mean_field'}
\begin{align}\label{eq:mean_field_flow}
    \dfrac{\partial e^{u}}{\partial t} = \Delta u + \rho\dfrac{e^{u}}{\int_\Sigma  e^{u}\dif\mu_{\Sigma}} - Q,
\end{align}
 was also well studied by Castéras \cite{Cas15mean} for noncritical cases, i.e., $\rho <8\pi$.
This flow possesses a structure that is very similar to the Calabi and Ricci-Hamilton flows. When $Q$ is a constant equal to the scalar curvature with respect to the metric $g$,  flow \eqref{eq:mean_field_flow} has been studied by Struwe \cite{Struwe02curvature}. A flow approach to Nirenberg's problem was studied by Struwe in \cite{Struwe05flow}. Castéras \cite{Cas15mean2} obtained the global convergence of the mean field flow \eqref{eq:mean_field_flow} when $\rho\notin 8\pi\mathbb{N}^*$. Later, Li and Zhu \cite{LiZhu19convergence} studied the critical case $\rho=8\pi$ and reprove the existence result of Ding, Jost, Li and Wang \cite{DinJosLiWan97differential} (see also \cite{FeiJM} for a critical sinh-Gordon flow). Zhu and the second named author \cite{SunZhu21global} extended to the case that the prescribed function is nonnegative.  For some other generalizations, we refer the reader to \cite{WanYan22mean,CheLiLiXu21gaussian,LiXu22flow}.

\vspace{2ex}

This parabolic equation describes how a configuration evolves toward equilibrium. Its study provides insights not readily available from elliptic methods. For instance, the flow can reveal how different solutions are connected through dynamical processes, which solutions are stable, and how singularities develop in finite time.  The flow also serves as the basis for numerical methods, offering practical algorithms for finding solutions.


\vspace{2ex}

To study the Toda system, \eqref{eq:toda_system}, Jost and Wang \cite[Theorem 1.3]{JosWan01analytic} obtained the following sharp Moser-Trudinger inequality  that for every $u\in H^1\left(\Sigma,\mathbb{R}^N\right)$
 \begin{align}\label{eq:TM}
     \dfrac12\sum_{i,j=1}^Na_N^{ij}\int_{\Sigma}\hin{\nabla u_i}{\nabla u_j}\dif\mu_{\Sigma}+4\pi\sum_{i=1}^N\left(\fint_{\Sigma}u_i\dif\mu_{\Sigma}-\ln\fint_{\Sigma}e^{u_i}\dif\mu_{\Sigma}\right)\geq-c,
 \end{align}
 where $c$ is a constant depending only on the geometric of $\Sigma$. In particular, when $N=1$, this inequality is just the classical  Trudinger-Moser inequality (cf. \cite[Theorem 1.7]{Fon93sharp})
\begin{align}\label{eq:classical-TM}
\dfrac14\int_{\Sigma}\abs{\nabla u}^2\dif\mu_{\Sigma}+4\pi\left(\fint_{\Sigma}u\dif\mu_{\Sigma}-\ln\fint_{\Sigma}e^{u}\dif\mu_{\Sigma}\right)\geq-c,\quad\forall u\in H^1\left(\Sigma\right).
\end{align}
 Consequently, if all $\rho_i<4\pi$, then the functional $J$ mentioned in \eqref{eq:Toda-Functional}  is coercive, and therefore it has a smooth minimizer in this case, which solves the Toda system \eqref{eq:toda_system}.
\vspace{2ex}

If $\max_{i}\rho_i\geq 4\pi$, the situation becomes much more complicated and much efforts has been made to solve the Toda system even in this case for $N=2$.  When $\max\set{\rho_1,\rho_2}>4\pi$, the functional $J$ is unbounded from below, and solutions to the Toda system have been found in cases of: $\Sigma$ with positive genus and $\rho_1\in (4\pi,8\pi), \rho_2<4\pi$ (see \cite{JosLinWan06analytic}); $\rho_1\in (4m\pi,4(m+1)\pi), \rho_2<4\pi$ (see \cite{MalNdi07some}); $\Sigma$ with positive genus and $\rho_i\neq 4m\pi$ for both $i=1, 2$ (see \cite{BatJevMalRui15general});   $\rho_1\in (4m\pi, 4(m+1)\pi), \rho_2\in (4\pi, 8\pi)$ (see \cite{JevKalMal15topological,MalRui13variational}), where $m\in \mathbb{N}$. Most of these existence results for the Toda system were obtained by variational methods, based on compactness theorem through blowup analysis (see the excellent survey papers \cite{Mal17variational, Mal17minmax} for more information).
\vspace{2ex}

When $\max\set{\rho_1,\rho_2}=4\pi$, the functional $J$ is bounded from below but non-coercive, and the Toda system is solvable under some proper assumptions on the prescribed functions $h_i$ and the Gaussian  curvature $\kappa$ of the Riemannian surface $\Sigma$ (see \cite{JosLinWan06analytic, LiLi05solutions}). It should be mentioned that Jost, Lin and Wang \cite[Theorem 1.1]{JosLinWan06analytic} proposed the following sufficient conditions: if $\abs{\Sigma}=1$,  $h_1, h_2$ are positive smooth functions, $\rho_1=4\pi, \rho_2\in(0,4\pi)$ and
\begin{align*}
    \Delta\ln h_1+8\pi-\rho_2-2\kappa>0,\quad\text{in}\ \Sigma,
\end{align*}
then the following Toda system
\begin{align*}
    \begin{cases}
    -\Delta u_1=8\pi\left(h_1e^{u_1}-1\right)-\rho_2\left(h_2e^{u_2}-1\right),\\
    -\Delta u_2=2\rho_2\left(h_2e^{u_2}-1\right)-4\pi\left(h_1e^{u_1}-1\right)
    \end{cases}
\end{align*}
has a minimizing solution. Very recently, the second named author and Zhu extended Jost-Lin-Wang's existence result to sign changed prescibed functions (see \cite{SunZhu25existence1, SunZhu25existence2}).


\subsection{Main results}
In this paper, we focus on the $\mathrm{SU}(N+1)$ Toda system \eqref{eq:toda_system} on a closed Riemannian surface. 
The main aim of this paper is  to introduce the following   flow to study the Toda system 

\begin{align}\label{eq:toda}
\dfrac{\partial e^{u_i}}{\partial t}=\sum_{j=1}^Na_N^{ij}\Delta u_j+\dfrac{\rho_ih_ie^{u_i}}{\int_{\Sigma}h_ie^{u_i}\dif\mu_{\Sigma}}-Q_i, \quad i\in I,
\end{align}
which is a Yamabe type flow.
Actually we will  define a flow  for a more general Liouville system as follows. 
\vspace{2ex}

Let $A = (a_{ij})$ be a real symmetric, positive definite matrix of order $N$, $\rho = \left(\rho_1,\dots,\rho_N\right) \in \mathbb{R}^N$,  
$h = \left(h_1,\dots,h_N\right) \in C^\infty\left(\Sigma, \mathbb{R}^N\right)$, and $Q = \left(Q_1,\dots,Q_N\right) \in C^\infty\left(\Sigma, \mathbb{R}^N\right)$, where 
\begin{align*}
    \rho_i=\int_{\Sigma}Q_i\dif\mu_{\Sigma},\quad \max_{\Sigma}h_i>0,\quad i\in I\coloneqq\set{1,\dots,N}.
\end{align*}
Given a initial data $u_0 = \left(u_{1,0},\dots,u_{N,0}\right)$ with 
\begin{align}\label{eq:initial-regularity}
    \int_{\Sigma}h_ie^{u_{i,0}}\dif\mu_{\Sigma}>0,\quad i\in I,
\end{align}
we want to evolve $u_0$ towards a map  $u_{\infty}$ of solution to the Liouville system
\begin{align}\label{eq:liouville-system}
    -\sum_{j=1}^Na^{ij}\Delta u_j=\dfrac{\rho_ih_ie^{u_i}}{\int_{\Sigma}h_ie^{u_i}\dif\mu_{\Sigma}}-Q_i, \quad i\in I,
\end{align}
through a family of maps $u(t)=u(t,\cdot), t\geq0$, by solving the following   \emph{flow of Liouville system}:
\begin{align}\label{eq:liouville}
\dfrac{\partial e^{u_i}}{\partial t}=\sum_{j=1}^Na^{ij}\Delta u_j+\dfrac{\rho_ih_ie^{u_i}}{\int_{\Sigma}h_ie^{u_i}\dif\mu_{\Sigma}}-Q_i, \quad i\in I.
\end{align}
\vspace{2ex}

When $A$ is the Cartan Matrix $A_N$ for $\mathrm{SU}(N+1)$, it is the Toda flow defined by \eqref{eq:toda}.
Notice that in this paper, we  allow the prescribed functions change sign and assume that the initial data satisfies the regularity condition. Therefore, for  flow \eqref{eq:liouville} to be well-defined, the following regularity conditions are required for all the existing time (and this will be one of the subtle parts in  proof of global existence of the flow of  Liouville system \eqref{eq:liouville}):
\begin{align*}
    \int_{\Sigma}h_ie^{u_i(t)}\dif\mu_{\Sigma}>0,\quad i\in I.
\end{align*}
\vspace{2ex}

Since  flow \eqref{eq:liouville} is parabolic, the short time existence is clear. 
We establish that when the maximum existence time $T$ is finite, the following finite time blowup condition is satisfied:
\begin{align}
\liminf_{t\nearrow T}\min_{1\leq i\leq N}\int_{\Sigma}h_ie^{u_i(t)}\dif\mu_{\Sigma}=0. \label{problem}
\end{align}
In particular, if the  the prescribed functions $h_i$ are positive everywhere, then the mass preservation (cf. \autoref{lem:conservation}) implies that \eqref{problem} would not happen for $T<\infty$ and hence the longtime existence  holds in this case.
When the prescribed functions $h_i$ can change sign, the global existence  of this flow becomes subtle, since we do not know if \eqref{problem} occurs for $T<\infty$.   However, thanks to the Trudinger-Moser inequality \eqref{eq:TM} of Jost and Wang \cite{JosWan01analytic}, we can also obtain the global existence for the Toda flow if all $\rho_i\in(0,4\pi]$, which is the case we are interested in. 
\vspace{2ex}

Precisely, we state our first main theorem about the  global existence for the flow of Liouville system as follows.
 \begin{theorem}\label{main:thm1}
    Assume that $A\leq A_N$ and  $\rho_i\in(0,4\pi]$  for each $i$. Then for every smooth initial data $u_0=\left(u_{1,0},\dots,u_{N,0}\right)$ satisfying the regularity condition \eqref{eq:initial-regularity},  the flow of Liouville system \eqref{eq:liouville} admits a unique smooth solution defined for all $[0,\infty)$. 
\end{theorem}

\begin{rem}
It remains to be an interesting question if the global existence is always true. In other words, one may ask if \eqref{problem} really occurs.
\end{rem}

By a standard argument using the Łojasiewicz-Simon gradient inequality, we conclude that this flow converges globally to  a smooth solution of the Liouville system \eqref{eq:liouville-system}  if the trajectory $\set{u(t):t\geq0}$ is bounded in $H^1\left(\Sigma,\mathbb{R}^N\right)$ (see \autoref{thm:convergence-general}). As an immediate consequence, by using the Trudinger-Moser inequality \eqref{eq:TM} again,  we conclude that this flow converges at infinite time if $A\leq A_N$ and $\rho_i\in(0,4\pi)$ for every $i\in I$ (see \autoref{thm:convergence-special}). 
\vspace{2ex}

Next we focus on the Toda flow in \eqref{eq:toda}. 
The main result of this paper is the global convergence of the Toda flow \eqref{eq:toda} which is stated as follows.

\begin{theorem}\label{main:thm2} 
For a given integer $1\le k \le N$. Assume $\rho _k =4\pi$ and $0<\rho_i<4\pi$ for any $i\neq k$. Assume the following conditions holds on the set $\set{p\in\Sigma:h_k(p)>0}$: 
 \begin{align}\label{eq:JLW}
     \Delta\ln h_k+\sum_{j=1}^Na_{Nkj}Q_j-2\kappa>0,
 \end{align}
where $\kappa$ is the Gaussian curvature of $\Sigma$. Then there exists a smooth initial date such that the Toda flow \eqref{eq:toda} exists for all the time and converges smoothly at infinite time to a smooth solution of the Toda system \eqref{eq:toda_system}.  
\end{theorem}
 
\begin{rem}
When all prescribed functions $h_i$ are positive and  $\rho_i<4\pi$, the global existence and converges of the Toda flow was proved in the preminilary version of this paper (see \cite{LuoWang}). Recently, another type of flow of the Liouville system, a heat flow,  was studied by Park and Zhang in \cite{ParkZhang} and they proved its global existence and convergence under proper small parameters assumption. 
\end{rem}
\begin{rem}
The condition \eqref{eq:JLW} was first found by Ding-Jost-Li-Wang when $N=1$, in their paper \cite{DinJosLiWan97differential}  on the existence of solutions to the mean field equation in the critical case.
\end{rem}

\subsection{Strategy of proof of \autoref{main:thm2}.}
For the flow of Liouville system  \eqref{eq:liouville}, the monotonicity formula (see \autoref{lem:monotonicity}) gives
\begin{align*}
    \frac{dJ(u(t))}{dt}=-\sum_{i=1}^N\int_{\Sigma}e^{u_i}\abs{\dfrac{\partial u_i}{\partial t}}^2\dif\mu_{\Sigma}.
\end{align*}
Hence there is a sequence $t_n\to\infty$
\begin{align*}
    \lim_{n\to\infty}\int_{\Sigma}e^{u_i(t_n)}\abs{\dfrac{\partial u_i(t_n)}{\partial t}}^2\dif\mu_{\Sigma}=0,\quad\forall i\in I.
\end{align*}
Moreover, for the Toda flow \eqref{eq:toda}, after more refined analysis, we obtain the following decay estimates (see \autoref{thm:decay})
\begin{align*}
    \lim_{t\to\infty}\int_{\Sigma}e^{u_i}\abs{\dfrac{\partial u_i}{\partial t}}^2\dif\mu_{\Sigma}=0,\quad\forall i\in I.
\end{align*}
\vspace{2ex}

If this flow does not converge at infinity, then we conclude that $\set{u(t)}$ can not be bounded in $H^1\left(\Sigma\right)$ (see \autoref{thm:convergence-general}). By Jensen's inequality, we know that
 \begin{align*}
     \bar u_i(t)\leq C,\quad\forall i\in I.
 \end{align*}
 Moreover, according to the Trudinger-Moser inequality \eqref{eq:TM}, we obtain
 \begin{align*}
     \bar u_i(t)\geq-C,\quad i\neq k.
 \end{align*}
 Therefore, we conclude that $\bar u_k(t)$ subconverges to $-\infty$ (see \autoref{lem:blowup}). Denote by
 \begin{align*}
     U_i=\sum_{j=1}^Na_N^{ij}u_j,\quad i\in I.
 \end{align*}
 A key observation of this paper is the following (see \autoref{thm:partial-H^1})
 \begin{align*}
     \norm{\nabla U_i(t)}_{L^2\left(\Sigma\right)}\leq C,\quad i\neq k.
 \end{align*}
 \vspace{2ex}
 Now  using blowup analysis, we conclude that the total singular set  consists of only one point $\set{p_0}$ (see \autoref{lem:singular-set}). Moreover, $h_k(p_0)>0$. Then we can show that (see \autoref{lem:upper})
 \begin{align*}
     u_i(t)\leq C,\quad\forall i\neq k.
 \end{align*}
 Since $e^{u_i(t)}$ is bounded in $L^1\left(\Sigma\right)$, we conclude that $u_i(t)-\bar u_i(t)$ subconverges weakly to $G_i(\cdot,p_0)$ in $W^{1}_p\left(\Sigma\right)$ for $1<p<2$. Notice that the functions $G_i(\cdot,p_0)$ satisfies the following elliptic system
\begin{align*}
    \begin{cases}
    -\sum_{j=1}^Na_N^{ij}\Delta G_j(\cdot,p_0)=\rho_i\frac{h_ie^{G_i(\cdot,p_0)}}{\int_{\Sigma}h_ie^{G_i(\cdot,p_0)}\dif\mu_{\Sigma}}-Q_i,\quad\bar G_i(\cdot, p_0)=0,\quad i\neq k,\\
        -\sum_{j=1}^Na_N^{kj}\Delta G_j(\cdot,p_0)=4\pi\delta_{p_0}-Q_k,\quad \bar G_k(\cdot,p_0)=0.
    \end{cases}
\end{align*}
\vspace{2ex}

Consider the functional
\begin{align*}
    J_k(u)\coloneqq J(u)-\left(\dfrac14\int_{\Sigma}\abs{\nabla u_k}^2\dif\mu_{\Sigma}+\int_{\Sigma}Q_ku_k\dif\mu_{\Sigma}-\ln\int_{\Sigma}h_ke^{u_k}\dif\mu_{\Sigma}\right).
\end{align*}
We show that (see \autoref{thm:lower})
\begin{align*}
    J(u(t))\geq F(p_0)\coloneqq&J_k\left(G(\cdot,p_0)\right)+\dfrac14\int_{\Sigma}\abs{\nabla\left(H_{k-1}(\cdot,p_0)+H_{k+1}(\cdot,p_0)\right)}^2\dif\mu_{\Sigma}+\int_{\Sigma}Q_kH_k(\cdot,p_0)\dif\mu_{\Sigma}\\
    &-4\pi\left(\lim_{p\to p_0}\left(H_k(p,p_0)+2\ln \mathrm{dist}(p,p_0)\right)+\ln h_k(p_0)\right)-4\pi-4\pi\ln\pi.
\end{align*}
Here
\begin{align*}
    H_i=\sum_{j=1}^Na_N^{ij}G_j,\quad i\in I
\end{align*}
and for convenience, $H_{0}=H_{N+1}=0$. Finally, under Jost-Lin-Wang's condition  \eqref{eq:JLW}, we can find a family of smooth maps $u^{(\epsilon)}$ such that
\begin{align*}
    J(u^{(\epsilon)})\uparrow F(p_0),\quad \text{as}\ \epsilon\to 0^+.
\end{align*}
In particular, there is an initial smooth data $u_0$ such that
\begin{align*}
    J(u_0)<F(p_0).
\end{align*}
Therefore, if we choose $p_0$ as the minimal point of $F$ in $\set{p\in\Sigma:h_k(p)>0}$, then we conclude that the Toda flow \eqref{eq:toda} converges globally. 

\subsection{Notations}
In this paper we will use notations which for readers' convenience we list below:
$$W^k_p\left(\Sigma\right)\coloneqq\set{u:\Sigma\to\mathbb{R}\vert\nabla^lu\in L^p\left(\Sigma\right), l=0,1,\cdots,k}.$$
In particular, $W^k_2\left(\Sigma\right)$ is denoted by $H^k\left(\Sigma\right)$.
$$C^{k,\alpha}\left(\Sigma\right) (0<\alpha<1)\coloneqq\{u:\Sigma\to\mathbb{R}|\nabla^l\partial^r_tu\in C^\alpha\left(\Sigma\right), l=0,1,\cdots,k\}.$$ 
$$ W^{k,2k}_p\left([0, T]\times\Sigma\right)\coloneqq\set{u:[0,T]\times\Sigma\to \mathbb{R}|\nabla^l \partial^r_tu\in L^p\left([0,T]\times\Sigma\right), \forall l+2r\leq 2k }.$$ In particular, $W^{k,2k}_2\left([0, T]\times\Sigma\right)\coloneqq H^{k,2k}\left([0, T]\times\Sigma\right).$
\vspace{2ex}

For $p=(x,t_x), q=(y,t_y)\in\Sigma\times [0,T],$ the parabolic distance is defined by $\delta(p, q)\coloneqq(d(x,y)^2+|t_x-t_y|)^\frac{1}{2}$. 
$$[u]_{\alpha,[0,T]\times\Sigma}=\sup_{p,q\in\Sigma\times[0,T],p\neq q}\frac{|u(p)-u(q)|}{\delta^\alpha(p,q)},$$
$$[u]^t_{\alpha, [0, T]\times\Sigma}=\sup_{p\in\sigma}\sup_{t,s\in [0, T], t\neq s}\frac{|u(x, t)-u(x,s)|}{|t-s|^\alpha},$$
$$[u]_{k+\alpha,[0,T]\times\Sigma}=\sum_{l+2r= k}[\nabla^l\partial^r_tu]_{\alpha,[0,T]\times\Sigma}+\sum_{l+2r= k-1}[\nabla^l\partial^r_tu]^t_{\frac{\alpha}{2},[0,T]\times\Sigma},$$ and
$$|u|_{k+\alpha, [0, T]\times\Sigma}\coloneqq\sum_{l+2r=k}[\nabla^l\partial^r_tu]_{\alpha, [0, T]\times\Sigma}+\sum_{l+2r=k-1}[\nabla^l\partial^r_tu]^t_{\alpha, [0, T]\times\Sigma}.$$ Then
$$C^{2k+\alpha, k+\frac{\alpha}{2}}\left([0, T]\times\Sigma\right)\coloneqq\set{u:[0, T]\times\Sigma\to \mathbb{R}| |u|_{k+\alpha, [0, T]\times\Sigma}<\infty}.$$ 
Moreover, $W^k_p\left(\Sigma, \mathbb{R^N}\right)$ is the space of vector valued functions whose components belong to $W^k_p (\Sigma)$, and $C^{k,\alpha}\left(\Sigma, \mathbb{R}^N\right)$, $H^k\left(\Sigma, \mathbb{R}^N\right)$, $W^{k,2k}_p\left([0, T]\times\Sigma, \mathbb{R}^N\right)$, $H^{k,2k}\left([0, T]\times\Sigma, \mathbb{R}^N\right)$, $C^{k+\alpha/2, 2k+\alpha}\left([0, T]\times\Sigma, \mathbb{R}^N\right)$ should be interpreted in the same way. 

\section{A priori estimates}
In this section, we assume that the solution to the flow of  Liouville system  \eqref{eq:liouville} exists on the time interval \([0, T]\) and provide several a priori estimates for it.

Since the matrix \(A\) is positive definite, we have
\begin{align*}
    0<c_A^{-1}\leq A\leq c_A,
\end{align*}
for some constant \(c_A>0\). Moreover, we impose the following uniform lower bound condition:
\begin{align*}
   \int_{\Sigma}h_ie^{u_i}\dif\mu_{\Sigma}\geq c_{\text{reg}}^{-1}>0,\quad i\in I,
\end{align*}
where \(c_{\text{reg}}>0\) is a constant.

Recall that the flow of Liouville system is given by
\begin{align*} 
\dfrac{\partial e^{u_i}}{\partial t}=\sum_{j=1}^Na^{ij}\Delta u_j+\dfrac{\rho_ih_ie^{u_i}}{\int_{\Sigma}h_ie^{u_i}\dif\mu_{\Sigma}}-Q_i,\quad i=1, 2\dots, N.
\end{align*}
and the funcational $J$ is defined by 
\begin{align*}
J(u)\coloneqq\dfrac12\sum_{i,j=1}^N\int_{\Sigma}a^{ij}\langle \nabla u_i,\nabla u_j\rangle\dif\mu_{\Sigma}+\sum_{i=1}^N\left(\int_{\Sigma}Q_iu_i\dif\mu_{\Sigma}-\rho_i\ln\int_{\Sigma}h_ie^{u_i}\dif\mu_{\Sigma}\right),\quad u\in H^1(\Sigma, \mathbb{R}^N).
\end{align*}
WE have  the following monotonicity property for the functional along the flow.

 \begin{lem}[Monotonicity formula]\label{lem:monotonicity}
     Along the flow of Liouville system \eqref{eq:liouville}, the following monotonicity holds:
     \begin{align}\label{eq:monotonicity}
    \dfrac{\dif}{\dif t}J(u)=-\sum_{i=1}^N\int_{\Sigma}e^{u_i}\abs{\dfrac{\partial u_i}{\partial t}}^2\dif\mu_{\Sigma}.
\end{align}
 \end{lem}
\begin{proof}
By direct computation, we obtain
\begin{align*}
    \dfrac{\dif}{\dif t}J(u)=&\sum_{i,j=1}^N\int_{\Sigma}a^{ij}\hin{\nabla\left(\dfrac{\partial u_i}{\partial t}\right)}{\nabla u_j}\dif\mu_{\Sigma}+\sum_{i=1}^N\left(\int_{\Sigma}Q_i\dfrac{\partial u_i}{\partial t}\dif\mu_{\Sigma}-\rho_i\dfrac{\int_{\Sigma}h_ie^{u_i}\frac{\partial u_i}{\partial t}\dif\mu_{\Sigma}}{\int_{\Sigma}h_ie^{u_i}\dif\mu_{\Sigma}}\right)\\
    =&-\int_{\Sigma}\sum_{i=1}^N\left(\sum_{j=1}^Na^{ij}\Delta u_j-Q_i+\rho_i\dfrac{h_ie^{u_i}}{\int_{\Sigma}h_ie^{u_i}\dif\mu_{\Sigma}}\right)\dfrac{\partial u_i}{\partial t}\dif\mu_{\Sigma}\\
    =&-\sum_{i=1}^N\int_{\Sigma}e^{u_i}\abs{\dfrac{\partial u_i}{\partial t}}^2\dif\mu_{\Sigma}.
\end{align*}
    
\end{proof}

Another useful property is the preservation of masses along the flow.
\begin{lem}[Mass preservation]\label{lem:conservation}
 For the flow of Liouville system \eqref{eq:liouville}, the following  holds:
\begin{align}\label{eq:conservation}
    \dfrac{\dif}{\dif t}\int_{\Sigma}e^{u_i}\dif\mu_{\Sigma}=0,\quad i\in I.
\end{align}
\end{lem}
\begin{proof}
    Integrating both sides of \eqref{eq:liouville} over $\Sigma$ yields
    \begin{align*}
        \dfrac{\dif}{\dif t}\int_{\Sigma}e^{u_i}\dif\mu_{\Sigma}=&\int_{\Sigma}\dfrac{\partial e^{u_i}}{\partial t}\dif\mu_{\Sigma}\\
        =&\int_{\Sigma}\left(\sum_{j=1}^Na^{ij}\Delta u_j+\dfrac{\rho_ih_ie^{u_i}}{\int_{\Sigma}h_ie^{u_i}\dif\mu_{\Sigma}}-Q_i\right)\dif\mu_{\Sigma}\\
        =&\rho_i-\int_{\Sigma}Q_i\dif\mu_{\Sigma}
        =0.
    \end{align*}
\end{proof}

According to the classical Trudinger–Moser inequality \eqref{eq:classical-TM}, we first establish the following a priori estimate for the flow  of Liouville system \eqref{eq:liouville}.

\begin{lem}[Upper bound estimate]
There exists a constant $C>0$, depending only on the upper bounds of $c_{A}, c_{\text{reg}}, \norm{u_0}_{H^1\left(\Sigma\right)}, \abs{\rho}, \norm{h}_{L^{\infty}\left(\Sigma\right)}$ and $\norm{Q}_{L^2\left(\Sigma\right)}$,  such that
\begin{align}\label{eq:e^{2u}}
    \sum_{i=1}^N\int_{\Sigma}e^{2u_i(t)}\dif\mu_{\Sigma}\leq Ce^{Ct}.
\end{align}
\end{lem}
\begin{proof}
Using \eqref{eq:liouville}, we compute
    \begin{align*}
        \dfrac12\dfrac{\dif}{\dif t}\sum_{i,j=1}^Na_{ij}\int_{\Sigma}e^{u_i+u_j}\dif\mu_{\Sigma}=&\sum_{i,j=1}^Na_{ij}\int_{\Sigma}e^{u_i}\dfrac{\partial e^{u_j}}{\partial t}\dif\mu_{\Sigma}\\
        =&\sum_{i=1}^N\int_{\Sigma}e^{u_i}\left(\Delta u_i+\sum_{j=1}^Na_{ij}\left(\rho_j\dfrac{h_je^{u_j}}{\int_{\Sigma}h_je^{u_j}\dif\mu_{\Sigma}}-Q_j\right)\right)\dif\mu_{\Sigma}\\
        =&-\sum_{i=1}^N\int_{\Sigma}e^{u_i}\abs{\nabla u_i}^2\dif\mu_{\Sigma}+\sum_{i=1}^N\int_{\Sigma}e^{u_i}\sum_{j=1}^Na_{ij}\left(\rho_j\dfrac{h_je^{u_j}}{\int_{\Sigma}h_je^{u_j}\dif\mu_{\Sigma}}-Q_j\right)\dif\mu_{\Sigma}.
    \end{align*}
    Since $A$ is positive definite, for any matrix $D$ and vector $x$ we have
    \begin{align*}
        \abs{x^TADx}\leq\abs{A^{1/2}x}\abs{A^{1/2}Dx}\leq c_A^{1/2}\abs{A^{1/2}x}\abs{Dx}\leq c_A^{1/2}\lambda_D\abs{A^{1/2}x}\abs{x}\leq c_A\lambda_D\abs{A^{1/2}x}^2,
    \end{align*}
    where $\lambda_D$ denotes the largest singular value of $D$. 
    
    Applying H\"older's inequality, we obtain
    \begin{align*}
        &\sum_{i=1}^N\int_{\Sigma}e^{u_i}\sum_{j=1}^Na_{ij}\left(\rho_j\dfrac{h_je^{u_j}}{\int_{\Sigma}h_je^{u_j}\dif\mu_{\Sigma}}-Q_j\right)\dif\mu_{\Sigma}\\
    \leq&c_Ac_{\text{reg}}\abs{\rho}\norm{h}_{L^{\infty}\left(\Sigma\right)}\sum_{i,j=1}^Na_{ij}\int_{\Sigma}e^{u_i+u_j}\dif\mu_{\Sigma}+\left(\sum_{i,j=1}^Na_{ij}\int_{\Sigma}Q_iQ_j\dif\mu_{\Sigma}\right)^{1/2}\left(\sum_{i,j=1}^Na_{ij}\int_{\Sigma}e^{u_i+u_j}\dif\mu_{\Sigma}\right)^{1/2}.
    \end{align*}
Hence we derive the differential inequality
    \begin{align*}
        \dfrac{\dif}{\dif t}\sum_{i,j=1}^Na_{ij}\int_{\Sigma}e^{u_i+u_j}\dif\mu_{\Sigma}\leq \left(2c_Ac_{\text{reg}}\abs{\rho}\norm{h}_{L^{\infty}\left(\Sigma\right)}+1\right)\sum_{i,j=1}^Na_{ij}\int_{\Sigma}e^{u_i+u_j}\dif\mu_{\Sigma}+c_A\norm{Q}^2_{L^2\left(\Sigma\right)}.
    \end{align*}
Applying Gronwall's inequality yields
\begin{align*}
\sum_{i,j=1}^Na_{ij}\int_{\Sigma}e^{u_{i}(t)+u_{j}(t)}\dif\mu_{\Sigma}\leq e^{Ct}\left(1+\sum_{i,j=1}^Na_{ij}\int_{\Sigma}e^{u_{i,0}+u_{j,0}}\dif\mu_{\Sigma}\right),
\end{align*}
where $C$ depends only on the quantities listed in the lemma. Combining this bound with the classical Trudinger-Moser inequality \eqref{eq:classical-TM} gives \eqref{eq:e^{2u}}.

\end{proof}

Based on the upper bound estimate \eqref{eq:e^{2u}} together with the monotonicity property \eqref{eq:monotonicity} and the mass preservation \eqref{eq:conservation}, we obtain the following a priori estimate.

\begin{lem}[$H^1$-estimate]
There exists a constant $C>0$, depending only on $c_{A}, c_{\text{reg}}, \norm{u_0}_{H^1\left(\Sigma\right)}, \abs{\rho}, \norm{h}_{L^{\infty}\left(\Sigma\right)}$ and $\norm{Q}_{L^2\left(\Sigma\right)}$, such that
\begin{align}\label{eq:H^1}
    \norm{u(t)}_{H^1\left(\Sigma\right)}\leq Ce^{Ct}.
\end{align}
\end{lem}
\begin{proof}
For each $i\in I$ and $t\in[0,T]$ define
\begin{align*}
A_i(t)=\set{p\in\Sigma: e^{u_i(t,p)}\geq\dfrac12\fint_{\Sigma}e^{u_{i,0}}\dif\mu_{\Sigma}}.
\end{align*}
By the mass preservation \eqref{eq:conservation} and the estimate \eqref{eq:e^{2u}}, we have
\begin{align*}
    \int_{\Sigma}e^{u_{i,0}}\dif\mu_{\Sigma}=\int_{\Sigma}e^{u_i(t)}\dif\mu_{\Sigma}=\int_{\Sigma\setminus A_i(t)}e^{u_i(t)}\dif\mu_{\Sigma}+\int_{A_i(t)}e^{u_i(t)}\dif\mu_{\Sigma}\leq \dfrac12\int_{\Sigma}e^{u_{i,0}}\dif\mu_{\Sigma}+Ce^{Ct}\abs{A_i(t)}^{1/2},
\end{align*}
where $\abs{A_i(t)}$ denotes the area of $A_i(t)$. Hence
\begin{align}\label{eq:area-bound}
    \abs{A_i(t)}\geq e^{-C-Ct},\quad\abs{\int_{A_i(t)}u_i(t)
    \dif\mu_{\Sigma}}\leq C+Ct,\quad\forall t\in[0,T].
\end{align}

 Using \eqref{eq:area-bound},
\begin{align*}
 \abs{\int_{\Sigma}u_i(t)
    \dif\mu_{\Sigma}}\leq&\abs{\int_{\Sigma\setminus A_i(t)}u_i(t)
    \dif\mu_{\Sigma}}+\abs{\int_{A_i(t)}u_i(t)
    \dif\mu_{\Sigma}}\\
    \leq&\abs{\Sigma\setminus A_i(t)}^{1/2}\left(\int_{\Sigma\setminus A_i(t)}\abs{u_i(t)}^2
    \dif\mu_{\Sigma}\right)^{1/2}+C+Ct\\
    \leq&\sqrt{\abs{\Sigma}-e^{-C-Ct}}\norm{u_i(t)}_{L^2\left(\Sigma\right)}+C+Ct.
\end{align*}
Applying Poincar\'e's inequality,
\begin{align*}
    \norm{u_i(t)}_{L^2\left(\Sigma\right)}\leq C\norm{\nabla u_i(t)}_{L^2\left(\Sigma\right)}+\sqrt{\abs{\Sigma}}\abs{\bar u_i(t)}\leq C\norm{\nabla u_i(t)}_{L^2\left(\Sigma\right)}+\sqrt{1-e^{-C-Ct}}\norm{u_i(t)}_{L^2\left(\Sigma\right)}+C+Ct
\end{align*}
which yields
\begin{align}\label{eq:L^2}
    \norm{u(t)}_{L^2\left(\Sigma\right)}\leq Ce^{Ct}\left(1+\norm{\nabla u(t)}_{L^2\left(\Sigma\right)}\right).
\end{align}

From the classical Trudinger-Moser inequality \eqref{eq:classical-TM} we have
\begin{align*}
    J(u_0)\leq C,\quad \sum_{i=1}^N\int_{\Sigma}e^{u_{i,0}}\dif\mu_{\Sigma}\leq C.
\end{align*}
Using  \eqref{eq:monotonicity} and \eqref{eq:conservation},
we have \begin{align*} 
\dfrac12\sum_{i,j=1}^na^{ij}\int_{\Sigma}\hin{\nabla u_i(t)}{\nabla u_j(t)}\dif\mu_{\Sigma}+\sum_{i=1}^N\int_{\Sigma}Q_iu_i(t)\dif\mu_{\Sigma}=J(u(t))+\sum_{i=1}^N\rho_i\ln\int_{\Sigma}h_ie^{u_i(t)}\dif\mu_{\Sigma}\leq C.
\end{align*}
By Young's inequality, for any $\varepsilon>0$ we obtain
\begin{align}\label{eq:H_0^1}
    C\geq \int_{\Sigma}\abs{\nabla u(t)}^2\dif\mu_{\Sigma}-\varepsilon\int_{\Sigma}\abs{u(t)}^2\dif\mu_{\Sigma}-\dfrac{C}{\varepsilon}.
\end{align}
Choosing $\varepsilon$ sufficiently small and substituting \eqref{eq:L^2} into \eqref{eq:H_0^1} finally gives \eqref{eq:H^1}.

\end{proof}

Using the $H^1$-estimate obtained above and differentiating the flow of Liouville system \eqref{eq:liouville}, we can derive the following refined a priori estimate.

\begin{lem}[$H^2$-estimate]\label{lem:H2-estimate}
There exists a constant $C_T>0$, depending only on $T, c_{A}, c_{\text{reg}}, \norm{u_0}_{H^2\left(\Sigma\right)}, \abs{\rho}, \norm{h}_{C^{1}\left(\Sigma\right)}$ and $\norm{Q}_{H^1\left(\Sigma\right)}$,  such that 
\begin{align}\label{eq:H^2}
\norm{u(t)}_{H^2\left(\Sigma\right)}\leq C_T,\quad t\in[0,T].
\end{align}
\end{lem}

\begin{proof}
First, we differentiate the quadratic form involving the Laplacians:
\begin{align*}
\dfrac{\dif}{\dif t}\sum_{i,j=1}^Na^{ij}\int_{\Sigma}\Delta u_i\Delta u_j\dif\mu_{\Sigma}=&2\sum_{i=1}^N\int_{\Sigma}\left(\dfrac{\partial e^{u_i}}{\partial t}+\rho_i\dfrac{h_ie^{u_i}}{\int_{\Sigma}h_ie^{u_i}\dif\mu_{\Sigma}}-Q_i\right)\Delta\dfrac{\partial u_i}{\partial t}\dif\mu_{\Sigma}.
\end{align*}
Set $w_i=e^{u_i/2}\frac{\partial u_i}{\partial t}$. Then
\begin{align*}
\dfrac{\dif}{\dif t}\sum_{i,j=1}^Na^{ij}\int_{\Sigma}\Delta u_i\Delta u_j\dif\mu_{\Sigma}=&2\sum_{i=1}^N\int_{\Sigma}\left(e^{u_i/2}w_i+\rho_i\dfrac{h_ie^{u_i}}{\int_{\Sigma}h_ie^{u_i}\dif\mu_{\Sigma}}-Q_i\right)\Delta\left(e^{-u_i/2}w_i\right)\dif\mu_{\Sigma}\\
=&-2\int_{\Sigma}\abs{\nabla w}^2\dif\mu_{\Sigma}+\dfrac12\sum_{i=1}^N\int_{\Sigma}\abs{\nabla u_i}^2w_{i}^2\dif\mu_{\Sigma}\\
&-2\sum_{i=1}^N\int_{\Sigma}\hin{\rho_i\dfrac{1}{\int_{\Sigma}h_ie^{u_i}\dif\mu_{\Sigma}}e^{u_i/2}\left(h_i\nabla u_i+\nabla h_i\right)-e^{-u_i/2}\nabla Q_i}{\nabla w_i-\dfrac12w_i\nabla u_i}\dif\mu_{\Sigma}.
\end{align*}
From the Trudinger-Moser inequality \eqref{eq:classical-TM} and the $H^1$-estimate \eqref{eq:H^1} we obtain, for every $p\in\mathbb{R}$ and $i\in I$,
\begin{equation*}
\int_{\Sigma} e^{p u_i(t)}\dif\mu_{\Sigma}\le C_{p,T},\qquad t\in[0,T].
\end{equation*}
Consequently, using Young's inequality we have for $t\in[0,T]$,
\begin{align*}
\dfrac{\dif}{\dif t}\sum_{i,j=1}^Na^{ij}\int_{\Sigma}\Delta u_i(t)\Delta u_j(t)\dif\mu_{\Sigma}\leq&-\int_{\Sigma}\abs{\nabla w(t)}^2\dif\mu_{\Sigma}+\norm{\nabla u(t)}_{L^4\left(\Sigma\right)}^2\norm{w(t)}_{L^4\left(\Sigma\right)}^2+C_T\norm{\nabla u(t)}_{L^4\left(\Sigma\right)}^2+C_T.
\end{align*}

Applying the Gagliardo–Nirenberg interpolation inequality
\begin{align}\label{eq:interpolation}
\norm{f}_{L^4\left(\Sigma\right)}^2\leq C\norm{f}_{L^2\left(\Sigma\right)}\norm{f}_{H^1\left(\Sigma\right)},\quad\forall f\in H^1\left(\Sigma\right),
\end{align}
together with \eqref{eq:H^1}, we estimate further:
\begin{align*}
&\dfrac{\dif}{\dif t}\sum_{i,j=1}^Na^{ij}\int_{\Sigma}\Delta u_i(t)\Delta u_j(t)\dif\mu_{\Sigma}\\
\leq&-\int_{\Sigma}\abs{\nabla w(t)}^2\dif\mu_{\Sigma}+C_T\norm{u(t)}_{H^2\left(\Sigma\right)}\norm{w(t)}_{L^2\left(\Sigma\right)}\norm{w(t)}_{H^1\left(\Sigma\right)}+C_T\norm{u(t)}_{H^2\left(\Sigma\right)}+C_T\\
\leq&-\dfrac12\int_{\Sigma}\abs{\nabla w(t)}^2\dif\mu_{\Sigma}+C_T\norm{u(t)}_{H^2\left(\Sigma\right)}^2\norm{w(t)}_{L^2\left(\Sigma\right)}^2+C_T\norm{u(t)}^2_{H^2\left(\Sigma\right)}+C_T\norm{w(t)}_{L^2\left(\Sigma\right)}^2+C_T\\
\leq&-\dfrac12\int_{\Sigma}\abs{\nabla w(t)}^2\dif\mu_{\Sigma}+C_T\norm{\Delta u(t)}_{L^2\left(\Sigma\right)}^2\norm{w(t)}_{L^2\left(\Sigma\right)}^2+C_T\norm{\Delta u(t)}^2_{L^2\left(\Sigma\right)}+C_T\norm{w(t)}_{L^2\left(\Sigma\right)}^2+C_T\\
\leq&-\dfrac12\int_{\Sigma}\abs{\nabla w(t)}^2\dif\mu_{\Sigma}+C_T\left(1+\sum_{i,j=1}^Na^{ij}\int_{\Sigma}\Delta u_i(t)\Delta u_j(t)\dif\mu_{\Sigma}\right)\left(1+\norm{w(t)}_{L^2\left(\Sigma\right)}^2\right).
\end{align*}
By Gronwall's inequality we obtain
\begin{align*}
\sum_{i,j=1}^Na^{ij}\int_{\Sigma}\Delta u_i(t)\Delta u_j(t)\dif\mu_{\Sigma}+\int_0^T\int_{\Sigma}\abs{\nabla w}^2\dif\mu_{\Sigma}\dif t\leq &C_T+C_T\sum_{i=1}^N\int_0^T\int_{\Sigma}e^{u_i(t)}\abs{\dfrac{\partial u_i}{\partial t}}^2\dif\mu_{\Sigma}\dif t\leq C_T,
\end{align*}
where we used \eqref{eq:monotonicity}. In particular, this yields the $H^2$-estimate \eqref{eq:H^2}.
\end{proof}

Building upon the $H^2$-estimate, we apply the Sobolev embedding theorem to obtain the following H\"older continuity result.
 
\begin{lem}[H\"older estimate]
    For every $\alpha\in(0,1)$ there exists a constant $C_{T,\alpha}>0$,
depending only on $T$, $\alpha$ and upper bounds of $c_{A}, c_{\text{reg}}, \norm{u_0}_{H^2\left(\Sigma\right)}, \abs{\rho}, \norm{h}_{C^{1}\left(\Sigma\right)}$ and $\norm{Q}_{H^1\left(\Sigma\right)}$,  such that 
\begin{align}\label{eq:holder}
\norm{u}_{C^{\alpha/2, \alpha}\left([0, T]\times\Sigma\right)}\leq C_{T,\alpha}.
\end{align}
\end{lem}
\begin{proof}
 By the Sobolev embedding
\begin{align*}
    \norm{f}_{C^{\alpha}\left(\Sigma\right)}\leq C_{\alpha}\norm{f}_{H^2\left(\Sigma\right)},\quad\forall f\in H^2\left(\Sigma\right),\quad\forall\alpha\in(0,1)
\end{align*}
 together with the estimate \eqref{eq:H^2}, we obtain
\begin{align}\label{eq:spatial-holder}
    \norm{u(t)}_{C^{\alpha}\left(\Sigma\right)}\leq C_{T,\alpha},\quad\forall \alpha\in(0,1).
\end{align}
Moreover, equation \eqref{eq:liouville} implies the pointwise bound
\begin{align*}
    \abs{\dfrac{\partial u}{\partial t}}\leq C_T\abs{\Delta u}+C_T,
\end{align*}
where the constant $C_T$ again relies on the a priori estimate \eqref{eq:H^2}. In particular,
\begin{align*}
    \norm{\dfrac{\partial u}{\partial t}}_{L^2\left(\Sigma\right)}\leq C_T. 
\end{align*}

To obtain the joint H\"older continuity, it suffices to prove the
temporal H\"older estimate
\begin{align}\label{eq:temp-holder}
    \abs{u(t,p)-u(s,p)}\leq C_{T,\alpha}\abs{t-s}^{\alpha/2}.
\end{align}
Indeed, combining \eqref{eq:spatial-holder} with \eqref{eq:temp-holder} yields
\begin{align*}
    \abs{u(t,p)-u(s,q)}
\leq\abs{u(t,p)-u(s,p)}+\abs{u(s,p)-u(s,q)}
\leq C_{T,\alpha}(\abs{t-s}^{\alpha/2}+\mathrm{dist}(p,q)^\alpha),
\end{align*}
which is exactly \eqref{eq:holder}.

  We now prove \eqref{eq:temp-holder}. Without loss of generality assume $0<t-s<1$.
    Denote by $B^{\Sigma}_r(p)$ the geodesic ball of radius $r$ centered at $p$. Then
\begin{align*}
    \abs{u(t,p)-u(s,p)}\leq&\dfrac{1}{\abs{B^{\Sigma}_{\sqrt{t-s}}(p)}}\int_{B^{\Sigma}_{\sqrt{t-s}}(q)}\abs{u(t,p)-u(s,p)}\dif\mu_{\Sigma}(q)\\
    \leq&\dfrac{1}{\abs{B^{\Sigma}_{\sqrt{t-s}}(p)}}\int_{B^{\Sigma}_{\sqrt{t-s}}(q)}\abs{u(t,p)-u(t,q)}\dif\mu_{\Sigma}(q)\\
    &+\dfrac{1}{\abs{B^{\Sigma}_{\sqrt{t-s}}(p)}}\int_{B^{\Sigma}_{\sqrt{t-s}}(q)}\abs{u(t,q)-u(s,q)}\dif\mu_{\Sigma}(q)\\
    &+\dfrac{1}{\abs{B^{\Sigma}_{\sqrt{t-s}}(p)}}\int_{B^{\Sigma}_{\sqrt{t-s}}(q)}\abs{u(s,q)-u(s,p)}\dif\mu_{\Sigma}(q)\\
    \leq&C_{T,\alpha}\abs{t-s}^{\alpha/2}+C\int_{B^{\Sigma}_{\sqrt{t-s}}(q)}\sup_{s\leq\tau\leq t}\abs{\dfrac{\partial u(\tau,q)}{\partial\tau}}\dif\mu_{\Sigma}(q)\\
    \leq&C_{T,\alpha}\abs{t-s}^{\alpha/2}+\sup_{0\leq \tau\leq T}\norm{\dfrac{\partial u(\tau)}{\partial\tau}}_{L^2\left(\Sigma\right)}\abs{B^{\Sigma}_{\sqrt{t-s}}(q)}^{1/2}\\
    \leq&C_{T,\alpha}\abs{t-s}^{\alpha/2}.
\end{align*}
 This completes the proof.

\end{proof}

We now state the uniqueness result for the flow of Liouville system \eqref{eq:liouville}.

\begin{lem}[$H^2$-uniqueness]\label{lem:uniqueness}
If the initial datum satisfies $u_0\in H^2\left(\Sigma,\mathbb{R}^N\right)$, then the flow of Liouville system \eqref{eq:liouville} admits at most one solution.
\end{lem}
\begin{proof}
Let $u$ and $v$ be two solutions of \eqref{eq:liouville} with the same initial datum $u_0$, and define $\eta_i\coloneqq u_i-v_i$. From \eqref{eq:liouville} we obtain
\begin{align*}
    \dfrac{\partial\left(e^{u_i}-e^{v_i}\right)}{\partial t}=\sum_{j=1}^Na^{ij}\Delta (u_j-v_j)+\dfrac{\rho_ih_ie^{u_i}}{\int_{\Sigma}h_ie^{u_i}\dif\mu_{\Sigma}}-\dfrac{\rho_ih_ie^{v_i}}{\int_{\Sigma}h_ie^{v_i}\dif\mu_{\Sigma}},\quad i\in I.
\end{align*}
A direct computation yields
\begin{align*}
    \dfrac{\dif}{\dif t}\sum_{i=1}^N\int_{\Sigma}\left(e^{u_i}-e^{v_i}\right)\eta_i\dif\mu_{\Sigma}=&-2\sum_{i,j=1}^Na^{ij}\int_{\Sigma}\hin{\nabla\eta_i}{\nabla\eta_j}\dif\mu_{\Sigma}\\
    &+2\sum_{i=1}^N\int_{\Sigma}\left(\dfrac{\rho_ih_ie^{u_i}}{\int_{\Sigma}h_ie^{u_i}\dif\mu_{\Sigma}}-\dfrac{\rho_ih_ie^{v_i}}{\int_{\Sigma}h_ie^{v_i}\dif\mu_{\Sigma}}\right)\eta_i-\sum_{i=1}^N\int_{\Sigma}\dfrac{\partial}{\partial t}\dfrac{e^{u_i}-e^{v_i}}{u_i-v_i}\eta_i^2\dif\mu_{\Sigma}.
\end{align*}

Because $u_0\in H^2\left(\Sigma,\mathbb{R}^N\right)$, the H\"older estimate \eqref{eq:holder} shows that
both $u$ and $v$ lie in $C^{\alpha, \alpha/2}([0, T]\times\Sigma,\mathbb{R}^N)$. Consequently there exists a constant $C_T>0$ such that
\begin{align}\label{eq:equiv-metric}
C_T^{-1}|\eta_i|^2\le\bigl(e^{u_i}-e^{v_i}\bigr)\eta_i\le C_T|\eta_i|^2,
\qquad i\in I.
\end{align}
We first estimate the non‑local terms. Observe that
\begin{align*}
    \abs{\dfrac{\rho_ih_ie^{u_i}}{\int_{\Sigma}h_ie^{u_i}\dif\mu_{\Sigma}}-\dfrac{\rho_ih_ie^{v_i}}{\int_{\Sigma}h_ie^{v_i}\dif\mu_{\Sigma}}}=&\abs{\rho_ih_i\dfrac{\left(e^{u_i}-e^{v_i}\right)\int_{\Sigma}h_ie^{v_i}\dif\mu_{\Sigma}+e^{v_i}\int_{\Sigma}h_i\left(e^{u_i}-e^{v_i}\right)\dif\mu_{\Sigma}}{\left(\int_{\Sigma}h_ie^{u_i}\dif\mu_{\Sigma}\right)\left(\int_{\Sigma}h_ie^{v_i}\dif\mu_{\Sigma}\right)}}\\
    \leq&C_T\abs{\eta_i}+C_T\norm{\eta_i}_{L^2\left(\Sigma\right)}.
\end{align*}
For the time‑derivative term we use the identity
\begin{align*}
    \dfrac{e^{u_i}-e^{v_i}}{u_i-v_i}
=\int_0^1 e^{(1-\theta)u_i+\theta v_i}\dif\theta.
\end{align*}
Differentiating under the integral and employing the uniform bounds provided by the $H^2$-estimate \eqref{eq:H^2}, we obtain
\begin{align*}
   \abs{\dfrac{\partial}{\partial t}\dfrac{e^{u_i}-e^{v_i}}{u_i-v_i}}=&\abs{\dfrac{\left(e^{u_i}\frac{\partial u_i}{\partial t}-e^{v_i}\frac{\partial v_i}{\partial t}\right)\left(u_i-v_i\right)-\left(e^{u_i}-e^{v_i}\right)(\frac{\partial u_i}{\partial t}-\frac{\partial v_i}{\partial t})}{\left(u_i-v_i\right)^2}}\\
   =&\abs{\dfrac{e^{u_i}\left(e^{-\eta_i}+\eta_i-1\right)\frac{\partial u_i}{\partial t}-e^{v_i}\left(e^{\eta_i}-\eta_i-1\right)\frac{\partial v_i}{\partial t}}{\eta_i^2}}\\
   \leq&C_T.
\end{align*}
Collecting these estimates and using \eqref{eq:equiv-metric} we arrive at
\begin{align*}
    \dfrac{\dif}{\dif t}\sum_{i=1}^N\int_{\Sigma}\left(e^{u_i}-e^{v_i}\right)\eta_i\dif\mu_{\Sigma}\leq&-\sum_{i,j=1}^Na^{ij}\int_{\Sigma}\hin{\nabla\eta_i}{\nabla\eta_j}\dif\mu_{\Sigma}+C_T\sum_{i=1}^N\int_{\Sigma}\left(e^{u_i}-e^{v_i}\right)\eta_i\dif\mu_{\Sigma}.
\end{align*}
Applying Gronwall's inequality gives
\begin{align*}
    \int_{\Sigma}\left(e^{u_i}-e^{v_i}\right)\left(u_i-v_i\right)\dif\mu_{\Sigma}=0,
\end{align*}
which forces $u_i\equiv v_i$ for every $i$. Hence the solution is unique.

\end{proof}

Additionally, we derive a priori estimates for higher regularity through an inductive argument.

\begin{lem}[Higher regularities]Let $k$ be a positive integer. There exists a positive constant $C_{T,k}$, depending only on $T$, $k$ and upper bounds of   $c_{A}, c_{\text{reg}}, \norm{u_0}_{H^{2k}\left(\Sigma\right)}, \abs{\rho}, \norm{h}_{C^{2k-1}\left(\Sigma\right)}$ and $\norm{Q}_{H^{2k-1}\left(\Sigma\right)}$, such that
\begin{align}\label{eq:H^{2k}}
\norm{u(t)}_{H^{2k}\left(\Sigma\right)}\leq C_{T,k},\quad\forall t\in[0,T].
\end{align}
\end{lem}
\begin{proof}
The case $k=1$ has been established in \autoref{lem:H2-estimate}. We therefore assume $k\ge 2$.

Rewrite the flow of Liouville system \eqref{eq:liouville} as
\begin{align*}
    \dfrac{\partial u_i}{\partial t}=e^{-u_i}\sum_{j=1}^Na^{ij}\Delta u_j+\tilde h_i-e^{-u_i}Q_i,
\end{align*}
where
\begin{align*}
    \tilde h_i=\dfrac{\rho_ih_i}{\int_{\Sigma}h_ie^{u_i}\dif\mu_{\Sigma}}.
\end{align*}
Applying the operator $\Delta^{k-1}$ to both sides gives
\begin{align*}
    \dfrac{\partial(\Delta^{k-1}u_i)}{\partial t}=e^{-u_i}\sum_{j=1}^Na^{ij}\Delta (\Delta^{k-1}u_j)+\sum_{j=1}^{N}a^{ij}\left[\Delta^{k-1},e^{-u_i}\Delta \right]u_j+\Delta^{k-1}\tilde h_i-\Delta^{k-1}\left(e^{-u_i}Q_i\right).
\end{align*}
Define $\phi_{i}=e^{u_i/2}\frac{\partial(\Delta^{k-1}u_i)}{\partial t}$.  Differentiating the quadratic form of higher-order Laplacians yields
\begin{align*}
    &\dfrac{\dif}{\dif t}\sum_{i,j=1}^Na^{ij}\int_{\Sigma}\Delta^{k}u_i\Delta^{k}u_j\dif\mu_{\Sigma}\\
    =&2\sum_{i=1}^N\int_{\Sigma}\Delta\left(e^{-u_i/2}\phi_i\right)\left(e^{u_i/2}\phi_i-e^{u_i}\sum_{j=1}^{N}a^{ij}\left[\Delta^{k-1},e^{-u_i}\Delta \right]u_j+e^{u_i}\Delta^{k-1}\left(e^{-u_i}Q_i-\tilde h_i\right)\right)\dif\mu_{\Sigma}\\
    =&-2\int_{\Sigma}\abs{\nabla\phi}^2\dif\mu_{\Sigma}+\dfrac12\int_{\Sigma}\abs{\nabla u_i}^2\phi_i^2\dif\mu_{\Sigma}\\
    &+2\sum_{i=1}^N\int_{\Sigma}\hin{\nabla\phi_i-\dfrac12\phi_i\nabla u_i}{e^{-u_i/2}\nabla\left(e^{u_i}\sum_{j=1}^{N}a^{ij}\left[\Delta^{k-1},e^{-u_i}\Delta \right]u_j-e^{u_i}\Delta^{k-1}\left(e^{-u_i}Q_i-\tilde h_i\right)\right)}\dif\mu_{\Sigma}\\
    \leq&-\norm{\nabla\phi}_{L^2\left(\Sigma\right)}^2+\norm{\nabla u}_{L^4\left(\Sigma\right)}^2\norm{\phi}_{L^4\left(\Sigma\right)}^2\\
    &+c\sum_{i=1}^N\norm{e^{-u_i/2}\nabla\left(e^{u_i}\sum_{j=1}^{N}a^{ij}\left[\Delta^{k-1},e^{-u_i}\Delta \right]u_j-e^{u_i}\Delta^{k-1}\left(e^{-u_i}Q_i-\tilde h_i\right)\right)}_{L^2\left(\Sigma\right)}^2.
\end{align*}
Using the H\"older estimate \eqref{eq:holder} together with the already obtained bounds in $H^{2\ell}\left(\Sigma\right)$ for $\ell<k$, we can control the commutator and the nonlinear terms. Indeed, by standard commutator estimates and the Sobolev inequalities,
\begin{align*}
    \sum_{i=1}^N\norm{e^{-u_i/2}\nabla\left(e^{u_i}\sum_{j=1}^{N}a^{ij}\left[\Delta^{k-1},e^{-u_i}\Delta \right]u_j-e^{u_i}\Delta^{k-1}\left(e^{-u_i}Q_i\right)\right)}_{L^2\left(\Sigma\right)}^2\leq C_{T,k}\norm{u}_{H^{2k}\left(\Sigma\right)}^2+C_{T,k}
\end{align*}
Applying the Gagliardo-Nirenberg interpolation inequality \eqref{eq:interpolation} to $\phi$ and employing the $H^{2k-2}$ bound (induction hypothesis) together with the $H^2$ estimate \eqref{eq:H^2}, we obtain
\begin{align*}
    \dfrac{\dif}{\dif t}\sum_{i,j=1}^Na^{ij}\int_{\Sigma}\Delta^{k}u_i\Delta^{k}u_j\dif\mu_{\Sigma}\leq-\dfrac12\norm{\nabla\phi}_{L^2\left(\Sigma\right)}^2+C_{T,k}\norm{u}_{H^{2k}\left(\Sigma\right)}^2+C_{T,k}.
\end{align*}
Gronwall's inequality then yields
\begin{align*}
    \sup_{t\in[0,T]}
\sum_{i,j=1}^N a^{ij}\int_{\Sigma}\Delta^k u_i\,\Delta^k u_j\dif\mu_{\Sigma}
\le C_{T,k}.
\end{align*}
By elliptic regularity this implies $\norm{u(t)}_{H^{2k}\left(\Sigma\right)}\le C_{T,k}$ for all $t\in[0,T]$, which completes the induction step.

\end{proof}

As a consequence of the higher regularity estimates, we obtain the smoothness of the flow of Liouville system with smooth initial data.

\begin{lem}\label{lem:smooth}
     If the initial datum $u_0$ is smooth, i.e. $u_0\in C^\infty\left(\Sigma,\mathbb{R}^N\right)$,
    then the solution $u$ of the flow of Liouville system \eqref{eq:liouville} is smooth on $[0, T]\times\Sigma$.
\end{lem}

\begin{proof}
By \eqref{eq:holder}, for every $\alpha\in(0,1)$ we have
\begin{align*}
    \norm{u}_{ C^{\alpha/2,\alpha}\left([0, T]\times\Sigma\right)}\leq C_{T,\alpha}.
\end{align*}
  Suppose that for some positive integer $k$ the estimate
\begin{align*}
    \norm{u}_{ C^{k-1+\alpha/2, 2(k-1)+\alpha}\left([0,T]\times\Sigma\right)}\leq C_{T,k-1
    \,\alpha}
\end{align*}
holds.
According to \eqref{eq:H^{2k}} and the Sobolev embedding theorem, we obtain
\begin{align*}
    \norm{u(t)}_{C^{2k+\alpha}\left(\Sigma\right)}\leq C_{T,k,\alpha}.
\end{align*}
      Observe that
   \begin{align*}
    \dfrac{\partial^{k}u_i}{\partial t^{k}}=e^{-u_i}\sum_{j=1}^Na^{ij}\Delta\dfrac{\partial^{k-1}u_j}{\partial t^{k-1}}+\left[e^{-u_i}\sum_{j=1}^Na^{ij}\Delta,\dfrac{\partial^{k-1}}{\partial t^{k-1}}\right]u_j+\dfrac{\partial^{k-1}\tilde h_i}{\partial t^{k-1}}-\dfrac{\partial^{k-1}}{\partial t^{k-1}}\left(e^{-u_i}Q_i\right).
   \end{align*}
   Consequently,
   \begin{align*}
       \norm{\dfrac{\partial^{k}u}{\partial t^{k}}}_{C^{\alpha/2}\left(\Sigma\right)}\leq C_{T,k,\alpha}.
   \end{align*}
     Hence
   \begin{align*}
       \norm{u}_{C^{k+\alpha/2, 2k+\alpha}\left([0, T]\times\Sigma\right)}\leq C_{T,k,\alpha}.
   \end{align*}
   Iterating this argument yields smoothness of $u$ along the flow \eqref{eq:liouville}.
\end{proof}

\section{Global existence}

In this section, we first establish the short time existence and uniqueness of the flow of Liouville system  under the following regularity condition:
\begin{align*}
\min_{1\leq i\leq N}\int_{\Sigma}h_ie^{u_i}\dif\mu_{\Sigma}>0.
\end{align*}

\begin{theorem}[Short time existence]\label{thm:local}
Let $u_0=\left(u_{1,0},\dots,u_{N,0}\right)\in C^{\infty}\left(\Sigma,\mathbb{R}^N\right)$ satisfy
\begin{align*}
\int_{\Sigma}h_ie^{u_{i,0}}\dif\mu_{\Sigma}>0, \quad i\in I.
\end{align*}
Then the flow of Liouville system \eqref{eq:liouville} admits a unique smooth solution on some time interval $[0,\epsilon)$. Moreover, if the maximal existence time $T$ is finite, then
\begin{align*}
\liminf_{t\nearrow T}\min_{1\leq i\leq N}\int_{\Sigma}h_ie^{u_i(t)}\dif\mu_{\Sigma}=0.
\end{align*}
\end{theorem}

\begin{proof}
The proof follows standard techniques and is analogous to the approach used for the harmonic map heat flow as presented by Chang \cite{Cha89heat}. For the convenience of the reader we provide a complete proof.

For $p>4$, given $u_i\in W^{1,2}_p\left([0,T]\times\Sigma\right), f_i\in L^p\left([0, T]\times\Sigma\right)$ and $u_{i,0}, g_i\in W^{2}_p\left(\Sigma\right)$ with
\begin{align*}
    u_i(0)=u_{i,0},\quad \int_{\Sigma}h_ie^{u_{i,0}}\dif\mu_{\Sigma}>0,
\end{align*}
we consider the linearized system obtained from the flow of Liouville system \eqref{eq:liouville}:
\begin{align}\label{eq:linearization}
\dfrac{\partial}{\partial t}(e^{u_i}\eta_i)=\sum_{j=1}^Na^{ij}\Delta \eta_j+\dfrac{\rho_ih_ie^{u_i}\eta_i}{\int_{\Sigma}h_ie^{u_i}\dif\mu_{\Sigma}}-\dfrac{\rho_ih_ie^{u_i}}{\left(\int_{\Sigma}h_ie^{u_i}\dif\mu_{\Sigma}\right)^2}\int_{\Sigma}h_ie^{u_i}\eta_i\dif\mu_{\Sigma}+f_i,\quad \eta_i(0)=g_i,\quad i\in I.
\end{align}
For every $p\in(2,4)\cup(4,\infty)$, the embedding (see \cite[Theorem 3.4]{Che23second})
\begin{align*}
W^{1,2}_p\left([0, T]\times\Sigma\right)\subset C^{1-2/p, 2-4/p}\left([0, T]\times\Sigma\right)
\end{align*}
holds; in particular $u$ is continuous.
Thus we may assume that for some $\delta>0$,
\begin{align*}
    \int_{\Sigma}h_ie^{u_{i}(t)}\dif\mu_{\Sigma}\geq c_{\text{reg}}^{-1}>0,\quad \forall t\in[0,\delta].
\end{align*} 
Consequently, system \eqref{eq:linearization} is uniformly parabolic and admits a unique solution $\eta\in W^{1,2}_p([0,\delta]\times\Sigma,\mathbb{R}^N)$.

By the inverse function theorem, we obtain a local solution of the flow of Liouville system  \eqref{eq:liouville}. Define
\begin{align*}
    \psi_i\coloneqq-\sum_{j=1}^Na^{ij}\Delta u_{j,0}-\dfrac{\rho_ih_ie^{u_{i,0}}}{\int_{\Sigma}h_ie^{u_{i,0}}\dif\mu_{\Sigma}}+Q_i,\quad i\in I,
\end{align*}
and for each $\epsilon\in(0,\delta)$ set
\begin{align*}
    \psi_{i,\epsilon}(t)=\begin{cases}
        0,&t\in[0,\epsilon],\\
        \psi_i(t),&t\in(\epsilon,\delta].
    \end{cases}
\end{align*}
Then $\psi=(\psi_{1,\epsilon},\dots,\psi_{N,\epsilon})$ satisfies 
\begin{align*}
    \norm{\psi_{\epsilon}-\psi}_{L^{p}([0,\delta]\times\Sigma)}\leq C\epsilon^{1/p}.
\end{align*}
Thus, for sufficiently small $\epsilon>0$, the inverse function theorem yields a solution $\left(u_{1,\epsilon},\dots,u_{N,\epsilon}\right)\in W^{1,2}_p\left([0,\delta]\times\Sigma,\mathbb{R}^N\right)$ of 
\begin{align*}
    \dfrac{\partial e^{u_{i,\epsilon}}}{\partial t}=\sum_{j=1}^Na^{ij}\Delta u_{j,\epsilon}+\dfrac{\rho_ih_ie^{u_{i,\epsilon}}}{\int_{\Sigma}h_ie^{u_{i,\epsilon}}\dif\mu_{\Sigma}}-Q_i+\psi_{i,\epsilon},\quad u_{i,\epsilon}(0)=u_{i,0},\quad i\in I.
\end{align*}
In particular, the flow of Liouville system has a solution in $W^{1,2}_p([0,\epsilon]\times\Sigma)$. By the $H^2$-uniqueness result \autoref{lem:uniqueness}, this solution is unique; moreover, since the initial data are smooth, the solution itself is smooth according to the regularity result \autoref{lem:smooth}.
Furthermore, if the maximal existence time $T$ is finite, then we must have  
 \begin{align*}
     \liminf_{t\nearrow T}\min_{1\leq i\leq N}\int_{\Sigma}h_ie^{u_i(t)}\dif\mu_{\Sigma}=0.
 \end{align*}
Indeed, otherwise we would obtain
 \begin{align*}
     \int_{\Sigma}h_ie^{u_i(t)}\dif\mu_{\Sigma}\geq c^{-1}>0,\quad\forall t\in[0,T),
 \end{align*}
 and standard parabolic estimates would imply that for some $\alpha\in(0,1)$ 
 \begin{align*}
     \norm{u}_{C^{1+\alpha/2, 2+\alpha}\left([0,T]\times\Sigma\right)}\leq C_T.
 \end{align*}
 Hence the flow of Liouville system \eqref{eq:liouville} could be extended continuously to $t=T$, and the argument above would then yield an extension to $[0,T+\epsilon]$ for some $\epsilon>0$, contradicting the maximality of $T$.

 Finally, the uniqueness of the flow follows directly from \autoref{lem:uniqueness}.
\end{proof}

Now we establish the existence of a positive constant $C$ such that the following uniform regularity condition holds along the flow of Liouville system \eqref{eq:liouville}:
\begin{align}\label{eq:uniformly_regularity}
    \int_{\Sigma}h_ie^{u_i}\dif\mu_{\Sigma}\geq C^{-1},\quad i\in I,    
\end{align}
provided that
\begin{align*}
 A\leq A_N\quad \hbox{ and } \quad 0<\rho_i\leq 4\pi.
\end{align*}
 As an immediate consequence, we obtain the global existence result stated in \autoref{main:thm1}.

Recall the Trudinger-Moser inequality \eqref{eq:TM} for system: for every $u\in H^1\left(\Sigma,\mathbb{R}^N\right)$
\begin{align*}
    \dfrac12\sum_{i,j=1}^N\int_{\Sigma}a_N^{ij}\hin{\nabla u_i}{\nabla u_j}\dif\mu_{\Sigma}+4\pi\sum_{i=1}^N\left(\bar u_i-\ln\fint_{\Sigma}e^{u_i}\dif\mu_{\Sigma}\right)\geq-c.
\end{align*}

We can now state the following global existence theorem for the flow of Liouville system \eqref{eq:liouville}.

\begin{theorem}\label{thm:global}
    Assume that 
    $A\leq A_N$ and  $\rho_i\in(0,4\pi]$ for each $i$. Then for every initial datum $u_0=\left(u_{1,0},\dots,u_{N,0}\right)\in H^{2}(\Sigma,\mathbb{R}^N )$ satisfying
 \begin{align*}
     \int_{\Sigma}h_ie^{u_{i,0}}\dif\mu_{\Sigma}>0,\quad i\in I,
 \end{align*}
 the flow of Liouville system \eqref{eq:liouville} admits a unique solution defined for all $[0,\infty)$. 
\end{theorem}

\begin{proof}

 Uniqueness follows directly from   \autoref{lem:uniqueness}.

Firstly, assume $u_0\in C^{\infty}\left(\Sigma,\mathbb{R}^N\right)$. A direct computation yields
\begin{align}\label{eq:TM'}
    \begin{split}
J(u)=&\dfrac12\sum_{i,j=1}^N\int_{\Sigma}a^{ij}\hin{\nabla\left(U_i-\Phi_i\right)}{\nabla \left(u_j-\phi_j\right)}\dif\mu_{\Sigma}+4\pi\sum_{i=1}^N\left(\overline{u_i-\phi_i}-\ln\fint_{\Sigma}e^{u_i-\phi_i}\dif\mu_{\Sigma}\right)\\
    &+\sum_{i=1}^N\left(\rho_i-4\pi\right)\left(\overline{u_i-\phi_i}-\ln\fint_{\Sigma}e^{u_i-\phi_i}\dif\mu_{\Sigma}\right)+\sum_{i=1}^N\rho_i\ln\fint_{\Sigma}e^{u_i-\phi_i}\dif\mu_{\Sigma}-\sum_{i=1}^N\rho_i\ln\int_{\Sigma}h_ie^{u_i}\dif\mu_{\Sigma}\\
    &-\dfrac12\sum_{i,j=1}^Na^{ij}\int_{\Sigma}\hin{\nabla\phi_i}{\nabla\phi_j}\dif\mu_{\Sigma}.
    \end{split}
\end{align}
Here $\phi=(\phi_1,\dots,\phi_N)$ is the unique solution of the elliptic system
\begin{align*}
    -\sum_{j=1}^Na^{ij}\Delta\phi_j=\dfrac{\rho_i}{\abs{\Sigma}}-Q_i,\quad \bar\phi_i=0,\quad 1\leq i\leq N.
\end{align*}
Using  \eqref{eq:monotonicity} and  \eqref{eq:conservation}, we obtain
\begin{align*}
    C\geq&\dfrac12\sum_{i,j=1}^N\int_{\Sigma}a_N^{ij}\hin{\nabla\left(U_i-\Phi_i\right)}{\nabla \left(u_j-\phi_j\right)}\dif\mu_{\Sigma}+4\pi\sum_{i=1}^N\left(\overline{u_i-\phi_i}-\ln\fint_{\Sigma}e^{u_i-\phi_i}\dif\mu_{\Sigma}\right)\\
    &+\sum_{i=1}^N\left(\rho_i-4\pi\right)\left(\overline{u_i-\phi_i}-\ln\fint_{\Sigma}e^{u_i-\phi_i}\dif\mu_{\Sigma}\right)-\sum_{i=1}^N\rho_i\ln\int_{\Sigma}h_ie^{u_i}\dif\mu_{\Sigma}.
\end{align*}
Applying the Trudinger-Moser inequality \eqref{eq:TM}, we deduce
\begin{align*}
    \sum_{i=1}^N\rho_i\ln\int_{\Sigma}h_ie^{u_i}\dif\mu_{\Sigma}\geq&-C,
\end{align*}
 which implies the uniform regularity condition \eqref{eq:uniformly_regularity}. By the local existence result \autoref{thm:local}, we conclude that the flow of Liouville system \eqref{eq:liouville} exists globally.

Finally, a standard argument by approximation yields that it is also true for $u_0\in H^2\left(\Sigma,\mathbb{R}^N\right)$.

\end{proof}

\section{Global convergence}

In this section, we apply the Łojasiewicz–Simon gradient inequality to study the global convergence of the flow \eqref{eq:liouville}. Since the flow is of negative gradient type with respect to the real‑analytic functional $J$, the Łojasiewicz–Simon inequality yields the following general convergence result. In particular, it implies an existence result for solutions of the  Liouville system
\begin{align*}
    -\sum_{j=1}^Na^{ij}\Delta u_j=\dfrac{\rho_ih_ie^{u_i}}{\int_{\Sigma}h_ie^{u_i}\dif\mu_{\Sigma}}-Q_i,\quad i\in I.
\end{align*}

\begin{theorem}\label{thm:convergence-general}
Let $u$ be the global solution of the flow of Liouville system  \eqref{eq:liouville}. If the trajectory $\set{u(t):t\geq0}$ is bounded in $H^1(\Sigma, \mathbb{R}^N)$, then $u(t)$ converges in $C^{\infty}(\Sigma, \mathbb{R}^N)$ as $t\to\infty$ to a smooth solution of the Liouville system \eqref{eq:liouville-system}.
\end{theorem}

\begin{proof}
The argument follows standard lines and we provide the details for completeness.

The functional $J:H^2(\Sigma, \mathbb{R}^N)\To\mathbb{R}$ is real-analytic. Its gradient map $\mathcal{M}:H^2(\Sigma, \mathbb{R}^N)\To L^2\left(\Sigma,\mathbb{R}^N\right)$  is given by
\begin{align*}
\mathcal{M}(u)_{i}\coloneqq-\sum_{j=1}^Na^{ij}\Delta u_j-\dfrac{\rho_ih_ie^{u_i}}{\int_{\Sigma}h_ie^{u_i}\dif\mu_{\Sigma}}+Q_i,\quad i\in I,
\end{align*}
with $\mathcal{M}(u)=(\mathcal{M}(u)_1,\dotsc,\mathcal{M}(u)_N)$. At a critical point $u_{\infty}$ the Jacobi operator of $\mathcal{L}_{ u_{\infty}}: H^2\left(\Sigma, \mathbb{R}^N\right)\To L^2\left(\Sigma,\mathbb{R}^N\right)$ acts as
\begin{align*}
\mathcal{L}_{u_{\infty}}(\xi)_{i}\coloneqq-\sum_{j=1}^Na^{ij}\Delta\xi_j-\dfrac{\rho_ih_ie^{u_{\infty,i}}}{\int_{\Sigma}h_ie^{u_{\infty,i}}\dif\mu_{\Sigma}}\xi_i+\dfrac{\rho_ih_ie^{u_{\infty,i}}}{\left(\int_{\Sigma}h_ie^{u_{\infty,i}}\dif\mu_{\Sigma}\right)^2}\int_{\Sigma}h_ie^{u_{\infty,i}}\xi_i\dif\mu_{\Sigma},\quad i\in I,
\end{align*}
where $u_{\infty}=(u_{\infty,1},\dotsc,u_{\infty,N})$ and $\xi=(\xi_1,\dotsc,\xi_N)$. 
The operator $\mathcal{L}_{u_{\infty}}: H^2(\Sigma, \mathbb{R}^N)\To L^2\left(\Sigma,\mathbb{R}^N\right)$ is Fredholm of index zero.   Consequently the \L ojasiewicz-Simon gradient inequality (see \cite[Theorem 2]{FeeMar20lojasiewicz} or \cite[Proposition 1.3]{Jen98simple}) holds:  
\begin{quotation}
There exist positive constants $\sigma$ and $\theta\in[1/2,1)$ such that
\begin{align*}
\forall u\in H^2(\Sigma, \mathbb{R}^N),\ \norm{u-u_{\infty}}_{H^2\left(\Sigma\right)}<\sigma \quad \Longrightarrow\quad\abs{J(u)-J(u_{\infty})}^{\theta}\leq\norm{\mathcal{M}(u)}_{L^{2}\left(\Sigma\right)}.
\end{align*}
\end{quotation}

Because the trajectory $\set{u(t):t\geq0}$ is bounded in $H^1\left(\Sigma, \mathbb{R}^N\right)$, we may extract a sequence $t_n\to \infty$ such that $u(t_n)$ converges weakly in $H^1\left(\Sigma, \mathbb{R}^N\right)$ to some limit $u_{\infty}$. By compactness we may assume the convergence is strong in every $L^p\left(\Sigma,\mathbb{R}^N\right)$ for $1<p<+\infty$. Since $\set{u(t):t\geq0}$ is bounded in $H^1(\Sigma, \mathbb{R}^N)$, the Truingder-Moser inequality \eqref{eq:classical-TM} yields that $J(u(t))$ is bounded. Then the monotonicity formula \eqref{eq:monotonicity} implies that $u_{\infty}$ is a weak solution of the system \eqref{eq:liouville-system}. Brezis-Merle's estimate \autoref{lem:BM} together with standard elliptic regularity then show that $u_{\infty}$ is smooth. 

Now we compare $u(t_n)$ with $u_{\infty}$.  From 
\begin{align*}
-\sum_{j=1}^Na^{ij}\Delta(u_j(t_n)-u_{j,\infty})=&\rho_i\left(\dfrac{h_ie^{u_i(t_n)}}{\int_{\Sigma}h_ie^{u_i(t_n)}\dif\mu_{\Sigma}}-\dfrac{h_ie^{u_{i,\infty}}}{\int_{\Sigma}h_ie^{u_{i,\infty}}\dif\mu_{\Sigma}}\right)-\dfrac{\partial e^{u_i(t_n)}}{\partial t},
\end{align*}
the classical Trudinger–Moser inequality \eqref{eq:classical-TM} and the $H^{2}$-elliptic estimate for the Laplacian yield
\begin{align}\label{eq:LS-inequality}
\norm{u(t_n)-u_{\infty}}_{H^2\left(\Sigma\right)}\leq& C\left(\norm{u(t_n)-u_{\infty}}_{L^2\left(\Sigma\right)}+\sum_{i=1}^N\norm{e^{u_i(t_n)/2}\dfrac{\partial u_i(t_n)}{\partial t}}_{L^2\left(\Sigma\right)}\right).
\end{align}
Consequently, $u(t_n)$ converges strongly in $H^2(\Sigma, \mathbb{R}^N)$ to $u_{\infty}$, and a bootstrap argument gives smooth convergence.

Define
\begin{align*}
s_n=\inf\set{t\geq t_n: \norm{u(t)-u(t_n)}_{L^2\left(\Sigma\right)}\geq\sigma},
\end{align*}
and assume $t_n<s_n\leq+\infty$. 
Since the trajectory is uniformly bounded in $H^1\left(\Sigma\right)$, we have the estimate
\begin{align}\label{eq:converges-H^2}
\norm{u(t)-u_{\infty}}_{H^2\left(\Sigma\right)}\leq C\left(\norm{u(t)-u_{\infty}}_{L^2\left(\Sigma\right)}+\sum_{i=1}^N\norm{e^{u_i(t)/2}\dfrac{\partial u_i(t)}{\partial t}}_{L^2\left(\Sigma\right)}\right),\quad\forall t>0.
\end{align}
Choosing $\sigma>0$ sufficiently small, the Łojasiewicz-Simon inequality \eqref{eq:LS-inequality} implies
\begin{align*}
\forall t_n\leq t<s_n,\ \norm{u(t)-u_{\infty}}_{L^2\left(\Sigma\right)}<\sigma \quad \Longrightarrow\quad\abs{J(u(t))-J(u_{\infty})}^{\theta}\leq\norm{\mathcal{M}(u(t))}_{L^{2}\left(\Sigma\right)}.
\end{align*}
Moreover,
\begin{align*}
\norm{u(t)}_{H^2\left(\Sigma\right)}\leq C,\quad\forall t_n\leq t\leq s_n.
\end{align*}
Recall
\begin{align*}
\dfrac{\dif}{\dif t}J(u(t))=-\sum_{i=1}^N\int_{\Sigma}e^{u_i(t)}\abs{\dfrac{\partial u_i(t)}{\partial t}}^2\dif\mu_{\Sigma}=-\sum_{i=1}^N\int_{\Sigma}e^{-u_i(t)}\abs{\mathcal{M}(u(t))_i}^2\dif\mu_{\Sigma}.
\end{align*}
Hence for $t_n\leq t\leq s_n$,
\begin{align*}
\dfrac{\dif}{\dif t}J(u(t))\leq&-C^{-1}\norm{\dfrac{\partial u(t)}{\partial t}}_{L^2\left(\Sigma\right)}\norm{\mathcal{M}(u(t))}_{L^2\left(\Sigma\right)}\\
\leq&-C^{-1}\norm{\dfrac{\partial u(t)}{\partial t}}_{L^2\left(\Sigma\right)}\abs{J(u(t))-J(u_{\infty})}^{\theta}.
\end{align*}

We claim that for all sufficiently large $n$ one has $s_n=+\infty$. If not, after passing to a subsequence we may assume $s_n<+\infty$ for every $n$. By \eqref{eq:monotonicity} we may suppose $J(u(t))>J(u_{\infty})$ for all $t$. Then
\begin{align*}
\dfrac{\dif}{\dif t}\left(J(u(t))-J(u_{\infty})\right)^{1-\theta}\leq-\dfrac{1-\theta}{C}\norm{\dfrac{\partial u(t)}{\partial t}}_{L^2\left(\Sigma\right)},\quad\forall t_n\leq t<s_n.
\end{align*}
Integrating from $t_n$ to $s_n$ gives
\begin{align*}
\int_{t_n}^{s_n}\norm{\dfrac{\partial u(t)}{\partial t}}_{L^2\left(\Sigma\right)}\dif t\leq \dfrac{1-\theta}{C}\left(J(u(t_n))-J(u_{\infty})\right)^{1-\theta}.
\end{align*}
Consequently
\begin{align*}
\sigma=&\norm{u(s_n)-u_{\infty}}_{L^2\left(\Sigma\right)}\\
\leq&\norm{u(t_n)-u_{\infty}}_{L^2\left(\Sigma\right)}+\norm{u(s_n)-u(t_n)}_{L^2\left(\Sigma\right)}\\
\leq& \norm{u(t_n)-u_{\infty}}_{L^2\left(\Sigma\right)}+\dfrac{1-\theta}{C}\left(J(u(t_n))-J(u_{\infty})\right)^{1-\theta}.
\end{align*}
Letting $n\to\infty$ yields a contradiction, because the right‑hand side tends to zero while the left‑hand side equals $\sigma>0$.
 
Thus $s_n=\infty$ for large $n$, and the previous estimate actually shows
\begin{align*}
\int_0^{+\infty}\norm{\dfrac{\partial u(t)}{\partial t}}_{L^2\left(\Sigma\right)}\dif t<+\infty.
\end{align*}
In particular,
\begin{align*}
\lim_{t\to+\infty}\norm{u(t)-u_{\infty}}_{L^2\left(\Sigma\right)}=0,
\end{align*}
which, together with \eqref{eq:converges-H^2} and standard elliptic estimates, implies that $u(t)$ converges strongly to $u_{\infty}$ in $C^{\infty}(\Sigma, \mathbb{R}^N)$.
\end{proof}

As an immediate consequence, we obtain the following convergence result.

\begin{theorem}\label{thm:convergence-special}
Assume that $A\leq A_N$ and $\rho_i\in(0,4\pi)$ for every $i$. Then for every initial datum $u_0=\left(u_{1,0},\dots,u_{N,0}\right)\in H^{2}\left(\Sigma,\mathbb{R}^N\right)$ satisfying
\begin{align*}
    \int_{\Sigma}h_ie^{u_{i,0}}\dif\mu_{\Sigma}>0,\quad i\in I,
\end{align*}
the flow of Liouville system \eqref{eq:liouville} admits a unique global solution, and this solution converges smoothly as $t\to\infty$ to a solution of the Liouville system \eqref{eq:liouville-system}.
\end{theorem}
\begin{rem}
From proofs of  \autoref{thm:convergence-general} and \autoref{thm:convergence-special} we see that if we assume all $h_i$ are positive functions on $\Sigma$, then we can weaken the assumption on $\rho_i$ to $\rho_i\in\mathbb{R}$ for all $i$ and $\rho_i<4\pi$ for all $i$, respectively. 
\end{rem}

\section{Decay estimates}

In this section, we establish an evolving Trudinger–Moser type inequality along the flow of Liouville system \eqref{eq:liouville}. Specifically, under the conditions
\begin{align*}
     A\leq A_N,\quad\rho_i\leq 4\pi,\quad i\in I,
 \end{align*}
 we prove that for every $w\in H^1(\Sigma, \mathbb{R}^N)$
\begin{align}\label{eq:evolve-TM}
    \dfrac{1}{2}\sum_{i,j=1}^Na^{ij}\int_{\Sigma}\hin{\nabla w_i}{\nabla w_j}\dif\mu_{\Sigma}+\sum_{i=1}^N\int_{\Sigma}\left(\rho_i\dfrac{h_i}{\int_{\Sigma}h_ie^{u_i}\dif\mu_{\Sigma}}-\dfrac{\partial u_i}{\partial t}\right)w_ie^{u_i}\dif\mu_{\Sigma}-\sum_{i=1}^N\rho_i\ln\int_{\Sigma}e^{w_i}e^{u_i}\dif\mu_{\Sigma}\geq-C,
\end{align}
holds along the flow of Liouville system \eqref{eq:liouville}. As a consequence, we obtain the following decay estimates along the flow.

\begin{theorem}\label{thm:decay}
Assume that $A\leq A_N$ and $\rho_i\in(0,4\pi]$ for each $1\leq i\leq N$. Then along the flow of Liouville system \eqref{eq:liouville} we have, for every $p>0$,
\begin{align}\label{eq:t-lp}
    \lim_{t\to\infty}\sum_{i=1}^N\int_{\Sigma}\abs{\dfrac{\partial u_i(t)}{\partial t}}^pe^{u_i(t)}\dif\mu_{\Sigma}=0
\end{align}
and
\begin{align}\label{eq:t-gradient}
    \lim_{t\to\infty}\int_{\Sigma}\abs{\nabla\dfrac{\partial u(t)}{\partial t}}^2\dif\mu_{\Sigma}=0.
\end{align}
\end{theorem}

\begin{proof}
A direct calculation shows that for every $w=\left(w_1,\dotsc,w_N\right)\in H^1(\Sigma, \mathbb{R}^N)$,
\begin{align*}
    &\dfrac{1}{2}\sum_{i,j=1}^Na^{ij}\int_{\Sigma}\hin{\nabla w_i}{\nabla w_j}\dif\mu_{\Sigma}+\sum_{i=1}^N\int_{\Sigma}\left(\rho_i\dfrac{h_i}{\int_{\Sigma}h_ie^{u_i}\dif\mu_{\Sigma}}-\dfrac{\partial u_i}{\partial t}\right)w_ie^{u_i}\dif\mu_{\Sigma}-\sum_{i=1}^N\rho_i\ln\int_{\Sigma}e^{w_i}e^{u_i}\dif\mu_{\Sigma}\\
    =&\dfrac{1}{2}\sum_{i,j=1}^Na^{ij}\int_{\Sigma}\hin{\nabla(w_i+u_i)}{\nabla (w_j+u_j)}\dif\mu_{\Sigma}+\sum_{i=1}^N\int_{\Sigma}Q_i(w_i+u_i)\dif\mu_{\Sigma}-\sum_{i=1}^N\rho_i\ln\int_{\Sigma}e^{w_i+u_i}\dif\mu_{\Sigma}\\
    &-J(u)-\sum_{i=1}^N\rho_i\ln\int_{\Sigma}h_ie^{u_i}.
\end{align*}
Using  \eqref{eq:conservation} and \eqref{eq:monotonicity}, we obtain
\begin{align*}
    &\dfrac{1}{2}\sum_{i,j=1}^Na^{ij}\int_{\Sigma}\hin{\nabla w_i}{\nabla w_j}\dif\mu_{\Sigma}+\sum_{i=1}^N\int_{\Sigma}\left(\rho_i\dfrac{h_i}{\int_{\Sigma}h_ie^{u_i}\dif\mu_{\Sigma}}-\dfrac{\partial u_i}{\partial t}\right)w_ie^{u_i}\dif\mu_{\Sigma}-\sum_{i=1}^N\rho_i\ln\int_{\Sigma}e^{w_i}e^{u_i}\dif\mu_{\Sigma}\\
    \geq&\dfrac{1}{2}\sum_{i,j=1}^Na_N^{ij}\int_{\Sigma}\hin{\nabla(w_i+u_i)}{\nabla (w_j+u_j)}\dif\mu_{\Sigma}+\sum_{i=1}^N\int_{\Sigma}Q_i(w_i+u_i)\dif\mu_{\Sigma}-\sum_{i=1}^N\rho_i\ln\int_{\Sigma}e^{w_i+u_i}\dif\mu_{\Sigma}-C.
\end{align*}
In particular, if $\max_{1\leq i\leq N}\rho_i\leq4\pi$, the Trudinger-Moser inequality \eqref{eq:TM} yields the evolving Trudinger-Moser inequality \eqref{eq:evolve-TM}, i.e.,
\begin{align*}
    \dfrac{1}{2}\sum_{i,j=1}^Na^{ij}\int_{\Sigma}\hin{\nabla w_i}{\nabla w_j}\dif\mu_{\Sigma}+\sum_{i=1}^N\int_{\Sigma}\left(\rho_i\dfrac{h_i}{\int_{\Sigma}h_ie^{u_i}\dif\mu_{\Sigma}}-\dfrac{\partial u_i}{\partial t}\right)w_ie^{u_i}\dif\mu_{\Sigma}-\sum_{i=1}^N\rho_i\ln\int_{\Sigma}e^{w_i}e^{u_i}\dif\mu_{\Sigma}\geq-C.
\end{align*}

Assume first that
\begin{align*}
    \norm{\nabla w_i}_{L^2\left(\Sigma\right)}\leq 1,\quad\int_{\Sigma}\left(\rho_i\dfrac{h_i}{\int_{\Sigma}h_ie^{u_i}\dif\mu_{\Sigma}}-\dfrac{\partial u_i}{\partial t}\right)w_ie^{u_i}\dif\mu_{\Sigma}=0,\quad i\in I.
\end{align*}
Then we conclude that
\begin{align*}
    \sum_{i=1}^N\rho_i\ln\int_{\Sigma}e^{w_i}e^{u_i}\dif\mu_{\Sigma}\leq C.
\end{align*}
Hence
\begin{align*}
    \rho_1\ln\int_{\Sigma}e^{w_1}e^{u_1}\dif\mu_{\Sigma}\leq C-\sum_{i=2}^N\rho_i\ln\int_{\Sigma}e^{w_i}e^{u_i}\dif\mu_{\Sigma}.
\end{align*}
Since $\rho_1>0$, we conclude that
\begin{align*}
    \int_{\Sigma}e^{w_1}e^{u_1}\dif\mu_{\Sigma}\leq \exp\left[\dfrac{1}{\rho_1}\left(C-\sum_{i=2}^N\rho_i\ln\int_{\Sigma}e^{w_i}e^{u_i}\dif\mu_{\Sigma}\right)\right].
\end{align*}
Replacing $w_1$ by $-w_1$ gives
\begin{align*}
    \int_{\Sigma}e^{-w_1}e^{u_1}\dif\mu_{\Sigma}\leq \exp\left[\dfrac{1}{\rho_1}\left(C-\sum_{i=2}^N\rho_i\ln\int_{\Sigma}e^{w_i}e^{u_i}\dif\mu_{\Sigma}\right)\right].
\end{align*}
We obtain
\begin{align*}
    \rho_1\ln\int_{\Sigma}\cosh\left(w_1\right)e^{u_1}\dif\mu_{\Sigma}\leq C-\sum_{i=2}^N\rho_i\ln\int_{\Sigma}e^{w_i}e^{u_i}\dif\mu_{\Sigma}.
\end{align*}
Since $\rho_i>0$ for each $i\in I$, we may repeating the argument above yields
\begin{align*}
    \sum_{i=1}^N\rho_i\ln\int_{\Sigma}\cosh\left(w_i\right)e^{u_i}\dif\mu_{\Sigma}\leq C,
\end{align*}
which gives for each $i\in I$
\begin{align*}
    \int_{\Sigma}\cosh\left(w_i\right)e^{u_i}\dif\mu_{\Sigma}\leq C. 
\end{align*}
Thus for every $p>0$
\begin{align*}
    \int_{\Sigma}\abs{w_i}^pe^{u_i}\dif\mu_{\Sigma}\leq C_p,\quad i\in I.
\end{align*}

Since for every positive number $p$ and nonnegative numbers $a, b$, we have
 \begin{align*}
     \left(a+b\right)^p\leq \max\set{1,2^{p-1}}\left(a^p+b^p\right).
 \end{align*}
For a general $w\in H^1(\Sigma, \mathbb{R}^N)$, we have for every $p>0$ 
\begin{align*}
\left(\sum_{i=1}^N\int_{\Sigma}\abs{w_i}^{p}e^{u_i}\dif\mu_{\Sigma}\right)^{1/p}\leq C_p\left(\norm{\nabla w}_{L^2\left(\Sigma\right)}+\sum_{i=1}^N\abs{\int_{\Sigma}\left(\dfrac{\rho_ih_i}{\int_{\Sigma}h_ie^{u_i}\dif\mu_{\Sigma}}-\dfrac{\partial u_i}{\partial t}\right)w_ie^{u_i}\dif\mu_{\Sigma}}\right).
\end{align*}
In particular
\begin{align}\label{eq:evolve-lp}
    \left(\sum_{i=1}^N\int_{\Sigma}\abs{\dfrac{\partial u_i}{\partial t}}^{p}e^{u_i}\dif\mu_{\Sigma}\right)^{1/p}\leq C_p\left(\norm{\nabla \dfrac{\partial u}{\partial t}}_{L^2\left(\Sigma\right)}+\sum_{i=1}^N\int_{\Sigma}\abs{\dfrac{\partial u_i}{\partial t}}^2e^{u_i}\dif\mu_{\Sigma}+\sum_{i=1}^N\left(\int_{\Sigma}\abs{\dfrac{\partial u_i}{\partial t}}^2e^{u_i}\dif\mu_{\Sigma}\right)^{1/2}\right).
\end{align}

Now we compute
\begin{align*}
    \dfrac{\dif}{\dif t}\sum_{i=1}^N\int_{\Sigma}\abs{\dfrac{\partial u_i}{\partial t}}^2e^{u_i}\dif\mu_{\Sigma}=&2\sum_{i=1}^N\int_{\Sigma}\dfrac{\partial^2 u_i}{\partial t^2}\dfrac{\partial u_i}{\partial t}e^{u_i}\dif\mu_{\Sigma}+\sum_{i=1}^N\int_{\Sigma}\left(\dfrac{\partial u_i}{\partial t}\right)^3e^{u_i}\dif\mu_{\Sigma}\\
    =&2\sum_{i=1}^N\int_{\Sigma}\dfrac{\partial^2 e^{u_i}}{\partial t^2}\dfrac{\partial u_i}{\partial t}\dif\mu_{\Sigma}-\sum_{i=1}^N\int_{\Sigma}\left(\dfrac{\partial u_i}{\partial t}\right)^3e^{u_i}\dif\mu_{\Sigma}\\
    =&2\sum_{i=1}^N\int_{\Sigma}\dfrac{\partial }{\partial t}\left(\sum_{j=1}^Na^{ij}\Delta u_j+\rho_i\dfrac{h_ie^{u_i}}{\int_{\Sigma}h_ie^{u_i}\dif\mu_{\Sigma}}-Q_i\right)\dfrac{\partial u_i}{\partial t}\dif\mu_{\Sigma}-\sum_{i=1}^N\int_{\Sigma}\left(\dfrac{\partial u_i}{\partial t}\right)^3e^{u_i}\dif\mu_{\Sigma}\\
    =&-2\sum_{i,j=1}^Na^{ij}\int_{\Sigma}\hin{\nabla\dfrac{\partial u_i}{\partial t}}{\nabla\dfrac{\partial u_j}{\partial t}}\dif\mu_{\Sigma}+2\sum_{i=1}^N\rho_i\dfrac{\int_{\Sigma}\abs{\frac{\partial u_i}{\partial t}}^2h_ie^{u_i}}{\int_{\Sigma}h_ie^{u_i}\dif\mu_{\Sigma}}-2\sum_{i=1}^N\rho_i\abs{\dfrac{\dif}{\dif t}\ln\int_{\Sigma}h_ie^{u_i}\dif\mu_{\Sigma}}^2\\
    &-\sum_{i=1}^N\int_{\Sigma}\left(\dfrac{\partial u_i}{\partial t}\right)^3e^{u_i}\dif\mu_{\Sigma}.
\end{align*}
By Young's inequality, for any $\varepsilon>0$,
 \begin{align*}
     -\sum_{i=1}^N\int_{\Sigma}\left(\dfrac{\partial u_i}{\partial t}\right)^3e^{u_i}\dif\mu_{\Sigma}\leq&\left(\sum_{i=1}^N\int_{\Sigma}\abs{\dfrac{\partial u_i}{\partial t}}^6e^{u_i}\dif\mu_{\Sigma}\right)^{1/4}\left(\sum_{i=1}^N\int_{\Sigma}\abs{\dfrac{\partial u_i}{\partial t}}^2e^{u_i}\dif\mu_{\Sigma}\right)^{3/4}\\
     \leq&\varepsilon\left(\sum_{i=1}^N\int_{\Sigma}\abs{\dfrac{\partial u_i}{\partial t}}^6e^{u_i}\dif\mu_{\Sigma}\right)^{1/3}+\dfrac{27}{256\epsilon^3}\left(\sum_{i=1}^N\int_{\Sigma}\abs{\dfrac{\partial u_i}{\partial t}}^2e^{u_i}\dif\mu_{\Sigma}\right)^{3}.
 \end{align*}
 Applying \eqref{eq:evolve-lp}, we obtain
\begin{align}\label{eq:evolve-lp'}
    \dfrac{\dif}{\dif t}\sum_{i=1}^N\int_{\Sigma}\abs{\dfrac{\partial u_i}{\partial t}}^2e^{u_i}\dif\mu_{\Sigma}\leq&-\sum_{i,j=1}^Na^{ij}\int_{\Sigma}\hin{\nabla\dfrac{\partial u_i}{\partial t}}{\nabla\dfrac{\partial u_j}{\partial t}}\dif\mu_{\Sigma}+C\sum_{i=1}^N\int_{\Sigma}\abs{\dfrac{\partial u_i}{\partial t}}^2e^{u_i}\dif\mu_{\Sigma}+C\left(\sum_{i=1}^N\int_{\Sigma}\abs{\dfrac{\partial u_i}{\partial t}}^2e^{u_i}\dif\mu_{\Sigma}\right)^{3}.
\end{align}
Consequently, for all $T\geq t_0\geq0$,
\begin{align*}
    \arctan\left(\sum_{i=1}^N\int_{\Sigma}\abs{\dfrac{\partial u_i(T)}{\partial t}}^2e^{u_i(T)}\dif\mu_{\Sigma}\right)\leq \arctan\left(\sum_{i=1}^N\int_{\Sigma}\abs{\dfrac{\partial u_i(t_0)}{\partial t}}^2e^{u_i(t_0)}\dif\mu_{\Sigma}\right)+C\sum_{i=1}^N\int_{t_0}^T\int_{\Sigma}\abs{\dfrac{\partial u_i(t)}{\partial t}}^2e^{u_i(t)}\dif\mu_{\Sigma}\dif t.
\end{align*}
The monotonicity formula \eqref{eq:monotonicity} then gives
\begin{align}\label{eq:t-l2}
    \lim_{t\to\infty}\sum_{i=1}^N\int_{\Sigma}\abs{\dfrac{\partial u_i(t)}{\partial t}}^2e^{u_i(t)}\dif\mu_{\Sigma}=0.
\end{align}
Thus the decay estimate \eqref{eq:evolve-lp} holds for $p=2$. 

Inserting \eqref{eq:t-l2} into \eqref{eq:evolve-lp'}, we obtain
\begin{align*}
   \dfrac{\dif}{\dif t}\sum_{i=1}^N\int_{\Sigma}\abs{\dfrac{\partial u_i}{\partial t}}^2e^{u_i}\dif\mu_{\Sigma}\leq&-\sum_{i,j=1}^Na^{ij}\int_{\Sigma}\hin{\nabla\dfrac{\partial u_i}{\partial t}}{\nabla\dfrac{\partial u_j}{\partial t}}\dif\mu_{\Sigma}+C\sum_{i=1}^N\int_{\Sigma}\abs{\dfrac{\partial u_i}{\partial t}}^2e^{u_i}\dif\mu_{\Sigma}.
\end{align*}
Hence
\begin{align*}
    \sum_{i,j=1}^Na^{ij}\int_0^{\infty}\int_{\Sigma}\hin{\nabla\dfrac{\partial u_i}{\partial t}}{\nabla\dfrac{\partial u_j}{\partial t}}\dif\mu_{\Sigma}\dif t\leq C.
\end{align*}

Now we compute
\begin{align*}
    \dfrac{\dif}{\dif t}\sum_{i,j=1}^Na^{ij}\int_{\Sigma}\hin{\nabla\dfrac{\partial u_i}{\partial t}}{\nabla\dfrac{\partial u_j}{\partial t}}\dif\mu_{\Sigma}=&2\sum_{i,j=1}^Na^{ij}\int_{\Sigma}\hin{\nabla\dfrac{\partial^2 u_i}{\partial t^2}}{\nabla\dfrac{\partial u_j}{\partial t}}\dif\mu_{\Sigma}\\
    =&-2\sum_{i,j=1}^Na^{ij}\int_{\Sigma}\dfrac{\partial^2 u_i}{\partial t^2}\dfrac{\partial \Delta u_j}{\partial t}\dif\mu_{\Sigma}\\
    =&-2\sum_{i=1}^N\int_{\Sigma}\dfrac{\partial^2u_i}{\partial t^2}\dfrac{\partial}{\partial t}\left(\dfrac{\partial e^{u_i}}{\partial t}-\rho_i\dfrac{h_ie^{u_i}}{\int_{\Sigma}h_ie^{u_i}\dif\mu_{\Sigma}}+Q_i\right)\dif\mu_{\Sigma}\\
    =&-2\sum_{i=1}^N\int_{\Sigma}\abs{\dfrac{\partial^2u_i}{\partial t^2}}^2e^{u_i}\dif\mu_{\Sigma}-2\sum_{i=1}^N\int_{\Sigma}\dfrac{\partial^2u_i}{\partial t^2}\abs{\dfrac{\partial u_i}{\partial t}}^2e^{u_i}\dif\mu_{\Sigma}\\
    &+2\sum_{i=1}^N\dfrac{\rho_i}{\int_{\Sigma}h_ie^{u_i}\dif\mu_{\Sigma}}\int_{\Sigma}\dfrac{\partial^2u_i}{\partial t^2}\dfrac{\partial u_i}{\partial t}h_ie^{u_i}\dif\mu_{\Sigma}-2\sum_{i=1}^N\dfrac{\rho_i\int_{\Sigma}h_ie^{u_i}\frac{\partial u_i}{\partial t}\dif\mu_{\Sigma}}{\left(\int_{\Sigma}h_ie^{u_i}\dif\mu_{\Sigma}\right)^2}\int_{\Sigma}\dfrac{\partial^2u_i}{\partial t^2}h_ie^{u_i}\dif\mu_{\Sigma}\\
    \leq&-\sum_{i=1}^N\int_{\Sigma}\abs{\dfrac{\partial^2u_i}{\partial t^2}}^2e^{u_i}\dif\mu_{\Sigma}+C\sum_{i=1}^N\int_{\Sigma}\abs{\dfrac{\partial u_i}{\partial t}}^2e^{u_i}\dif\mu_{\Sigma}+C\sum_{i=1}^N\int_{\Sigma}\abs{\dfrac{\partial u_i}{\partial t}}^4e^{u_i}\dif\mu_{\Sigma}.
\end{align*}
Thus
\begin{align*}
    & \dfrac{\dif}{\dif t}\left(\sum_{i=1}^N\int_{\Sigma}\abs{\dfrac{\partial u_i}{\partial t}}^2e^{u_i}\dif\mu_{\Sigma}+\sum_{i,j=1}^Na^{ij}\int_{\Sigma}\hin{\nabla\dfrac{\partial u_i}{\partial t}}{\nabla\dfrac{\partial u_j}{\partial t}}\dif\mu_{\Sigma}\right)\\
    \leq&-\sum_{i=1}^N\int_{\Sigma}\abs{\dfrac{\partial^2u_i}{\partial t^2}}^2e^{u_i}\dif\mu_{\Sigma}+C\left(\sum_{i=1}^N\int_{\Sigma}\abs{\dfrac{\partial u_i}{\partial t}}^2e^{u_i}\dif\mu_{\Sigma}+\sum_{i,j=1}^Na^{ij}\int_{\Sigma}\hin{\nabla\dfrac{\partial u_i}{\partial t}}{\nabla\dfrac{\partial u_j}{\partial t}}\dif\mu_{\Sigma}\right)\\
    &+C\left(\sum_{i=1}^N\int_{\Sigma}\abs{\dfrac{\partial u_i}{\partial t}}^2e^{u_i}\dif\mu_{\Sigma}+\sum_{i,j=1}^Na^{ij}\int_{\Sigma}\hin{\nabla\dfrac{\partial u_i}{\partial t}}{\nabla\dfrac{\partial u_j}{\partial t}}\dif\mu_{\Sigma}\right)^2.
\end{align*}
By an argument analogous to the one above, we obtain the estimate \eqref{eq:t-gradient}.

Finally, combining \eqref{eq:evolve-lp}, \eqref{eq:t-l2} and \eqref{eq:t-gradient} yields the desired estimate \eqref{eq:t-lp}.
\end{proof}

\section{Partial \texorpdfstring{$H^1$-estimates}{H1-estimates}}
Define
\begin{align*}
    U_i=\sum_{j=1}^Na^{ij}u_j.
\end{align*}
Then we obtain the following partial $H^1$-estimate.

\begin{theorem}[Partial $H^1$-estimate]\label{thm:partial-H^1}
     Assume that $A\leq A_N$ and $\max_{1\leq i\leq N}\rho_i\leq 4\pi$. 
    Then along the flow of Liouville system \eqref{eq:liouville} the following bounds hold:
    \begin{align*}
        \bar u_i \geq -C, \qquad \norm{\nabla U_i}_{L^2\left(\Sigma\right)} \leq C,
    \end{align*}
    provided $\rho_i < 4\pi$.
\end{theorem}

\begin{proof}
    Let $\phi=(\phi_1,\dots,\phi_N)$ be the unique solution of the system
\begin{align*}
    -\sum_{j=1}^Na^{ij}\Delta\phi_j=\dfrac{\rho_i}{\abs{\Sigma}}-Q_i,\quad \bar\phi_i=0,\quad 1\leq i\leq N.
\end{align*} 
 From \eqref{eq:TM'}, which appeared in the proof of \autoref{thm:global}, we have
\begin{align}\label{eq:TM''}
    \dfrac12\sum_{i,j=1}^N\int_{\Sigma}a^{ij}\hin{\nabla\left(U_i-\Phi_i\right)}{\nabla \left(u_j-\phi_j\right)}\dif\mu_{\Sigma}+\sum_{i=1}^N\rho_i\left(\overline{u_i-\phi_i}-\ln\fint_{\Sigma}e^{u_i-\phi_i}\dif\mu_{\Sigma}\right)\leq C.
\end{align}
Hence the Trudinger-Moser inequality \eqref{eq:TM} implies
\begin{align*}
    \sum_{i=1}^N(4\pi-\rho_i)\bar u_i\geq-C.
\end{align*}
    By Jensen's inequality and \eqref{eq:conservation}, 
\begin{align*}
    \bar u_i \leq C,\quad\forall i\in I.
\end{align*}
Consequently,
\begin{align*}
    \bar u_i \geq -C,\quad\text{whenever}\quad  \rho_i<4\pi.
\end{align*}

 We now prove the second part of the statement.

 Consider two cases.
    
 \textbf{Case 1.} Assume $\rho_i<4\pi$ for all $k+1\leq i\leq N$.

  Write $A^{-1}$ in block form as
\begin{align*}
    A^{-1}=\begin{pmatrix}
        A_{11}&A_{12}\\
        A_{21}&A_{22}
    \end{pmatrix},
\end{align*}
 where $A_{11}$ is a $k\times k$ matrix. For a vector $x=(x_1^T,x_2^T)^T$ we have
\begin{align*}
    x^TA^{-1}x=&\begin{pmatrix}
        x_1^T&x_2^T
    \end{pmatrix}\begin{pmatrix}
        A_{11}&A_{12}\\
        A_{21}&A_{22}
    \end{pmatrix}\begin{pmatrix}
        x_1\\
        x_2
    \end{pmatrix}\\
    =&x_{1}^TA_{11}x_1+2x_1^TA_{12}x_2+x_2^TA_{22}x_2\\
    =&\abs{A_{22}^{1/2}x_2}^2+2\hin{A_{22}^{1/2}x_2}{A_{22}^{-1/2}A_{21}x_1}+x_{1}^TA_{11}x_1\\
    =&\abs{A_{22}^{1/2}x_2+A_{22}^{-1/2}A_{21}x_1}^2-\abs{A_{22}^{-1/2}A_{21}x_1}^2+x_{1}^TA_{11}x_1\\
    =&x_1^T\left(A_{11}-A_{12}A_{22}^{-1}A_{21}\right)x_1+\begin{pmatrix}
        x_1^T&x_2^T
    \end{pmatrix}\begin{pmatrix}
        A_{12}\\
        A_{22}
    \end{pmatrix}A_{22}^{-1}\begin{pmatrix}
        A_{21}&A_{22}
    \end{pmatrix}\begin{pmatrix}
        x_1\\
        x_2
    \end{pmatrix}.
\end{align*}
That is, the following identity holds
\begin{align*}
    x^TA^{-1}x=&x_1^T\left(A_{11}-A_{12}A_{22}^{-1}A_{21}\right)x_1+x^T\begin{pmatrix}
        A_{12}\\
        A_{22}
    \end{pmatrix}A_{22}^{-1}\begin{pmatrix}
        A_{21}&A_{22}
    \end{pmatrix}x.
\end{align*}
Hence
\begin{align*}
    x^TA^{-1}x\geq x_1^T\left(A_{11}-A_{12}A_{22}^{-1}A_{21}\right)x_1+C^{-1}\abs{\begin{pmatrix}
        A_{21}&A_{22}
    \end{pmatrix}x}^2.
\end{align*}

Notice that
\begin{align*}
    A=\begin{pmatrix}
        \left(A_{11}-A_{12}A_{22}^{-1}A_{21}\right)^{-1}&-A_{11}^{-1}A_{12}(A_{22}-A_{21}A_{11}^{-1}A_{21})^{-1}\\
        -(A_{22}-A_{21}A_{11}^{-1}A_{12})^{-1}A_{21}A_{11}^{-1}&(A_{22}-A_{21}A_{11}^{-1}A_{12})^{-1}
    \end{pmatrix}.
\end{align*}
Since $A-A_N$ is non‑positive definite, the matrix
    $\left(A_{11}-A_{12}A_{22}^{-1}A_{21}\right)^{-1}-A_{k}$ is also non‑positive definite,
    where $A_{k}$ denotes the Cartan matrix of $\mathrm{SU}(k+1)$.
    Equivalently, $A_{11}-A_{12}A_{22}^{-1}A_{21}-A_k^{-1}$ is nonnegative definite. Therefore,
\begin{align*}
    x^TA^{-1}x\geq x_1^TA_k^{-1}x_1+C^{-1}\abs{\begin{pmatrix}
        A_{21}&A_{22}
    \end{pmatrix}x}^2.
\end{align*}

Applying this to $x_i=\nabla\left(U_i-\Phi_i\right)$ and using \eqref{eq:TM''}, we obtain
\begin{align*}
    C\geq&\dfrac12\sum_{i,j=1}^{k}a^{ij}_{k}\int_{\Sigma}\hin{\nabla\left(U_i-\Phi_i\right)}{\nabla\left(u_j-\phi_j\right)}\dif\mu_{\Sigma}+\sum_{i=1}^{k}\rho_i\left(\overline{u_i-\phi_i}-\ln\fint_{\Sigma}e^{u_i-\phi_i}\dif\mu_{\Sigma}\right)+C^{-1}\sum_{i=k+1}^N\int_{\Sigma}\abs{\nabla\left(U_i-\Phi_i\right)}^2\dif\mu_{\Sigma},
\end{align*}
where
\begin{align*}
    a_{k}^{ij}=\min\set{i,j}-\dfrac{ij}{k+1},\quad 1\leq i, j\leq k,
\end{align*}
and we used the fact $\bar u_i\geq -C$ for $k+1\leq i\leq N$. The Trudinger-Moser inequality \eqref{eq:TM} then yields
\begin{align*}
    \sum_{i=k+1}^N\int_{\Sigma}\abs{\nabla\left(U_i-\Phi_i\right)}^2\dif\mu_{\Sigma}\leq C
\end{align*}
which implies
\begin{align*}
    \norm{\nabla U_i}_{L^2\left(\Sigma\right)}\leq C,\quad\forall k+1\leq i\leq N.
\end{align*}
 A similar argument shows that if $\rho_i<4\pi$ for all $1\leq i\leq k$, then
\begin{align*}
    \norm{\nabla U_i}_{L^2\left(\Sigma\right)}\leq C,\quad 1\leq i\leq k.
\end{align*}

 \textbf{Case 2.} Assume $\rho_i<4\pi$ for all $k+1\leq i\leq N-m$.

  Write $A^{-1}$ in block form as
\begin{align*}
    A^{-1}=\begin{pmatrix}
        A_{11}&A_{12}&A_{13}\\
        A_{21}&A_{22}&A_{23}\\
        A_{31}&A_{32}&A_{33}
    \end{pmatrix},
\end{align*}
where $A_{11}$ and $A_{33}$ are square matrices of order $k$ and $m$ respectively. We have
\begin{align*}
   \begin{pmatrix}
       1&0&0\\
       0&0&1\\
       0&1&0
   \end{pmatrix} A^{-1}\begin{pmatrix}
       1&0&0\\
       0&0&1\\
       0&1&0
   \end{pmatrix}=\begin{pmatrix}
   B_{11}&B_{12}\\
      B_{21}&A_{22}
   \end{pmatrix},
\end{align*}
which implies
\begin{align*}
  \begin{pmatrix}
       1&0&0\\
       0&0&1\\
       0&1&0
   \end{pmatrix} A\begin{pmatrix}
       1&0&0\\
       0&0&1\\
       0&1&0
   \end{pmatrix}  =\begin{pmatrix}
       \left(B_{11}-B_{12}A_{22}^{-1}B_{21}\right)^{-1}&-B_{11}B_{12}\left(A_{22}-B_{21}B_{11}^{-1}B_{12}\right)^{-1}\\
      -\left(A_{22}-B_{21}B_{11}^{-1}B_{12}\right)^{-1}B_{21}B_{11}^{-1} &\left(A_{22}-B_{21}B_{11}^{-1}B_{12}\right)^{-1}
   \end{pmatrix}.
\end{align*}
Notice that
\begin{align*}
    B_{11}-B_{12}A_{22}^{-1}B_{21}=&\begin{pmatrix}
        A_{11}-A_{12}A_{22}^{-1}A_{21}&A_{13}-A_{12}A_{22}^{-1}A_{23}\\
        A_{31}-A_{32}A_{22}^{-1}A_{21}&A_{33}-A_{32}A_{22}^{-1}A_{23}
    \end{pmatrix}.
\end{align*}
Since
\begin{align*}
    \begin{pmatrix}
       1&0&0\\
       0&0&1\\
       0&1&0
   \end{pmatrix} A_N\begin{pmatrix}
       1&0&0\\
       0&0&1\\
       0&1&0
   \end{pmatrix}=A_N,
\end{align*}
a similar argument mentioned above as in the first case, we conclude that the matrix $B_{11}-B_{12}A_{22}^{-1}B_{21}-A_{k+m}^{-1}$ is nonnegative definite. That is, the matrix 
\begin{align*}
    \begin{pmatrix}
            A_{11}&A_{13}\\
            A_{31}&A_{33}
        \end{pmatrix}-\begin{pmatrix}
            A_{12}\\
            A_{32}
        \end{pmatrix}A_{22}^{-1}\begin{pmatrix}
            A_{21}&A_{23}
        \end{pmatrix}-A_{k+m}^{-1}
\end{align*}
 is nonnegative definite. 

    If we write $x=(x_1^T,x_2^T,x_3^T)^T$, then 
    \begin{align*}
        x^TA^{-1}x=&x_1^TA_{11}x_1+x_2^TA_{22}x_2+x_3^TA_{33}^Tx_3+2x_1^TA_{12}x_2+2x_1^TA_{13}x_3+2x_2^TA_{23}x_3\\
        =&\begin{pmatrix}
            x_1^T&x_3^T
        \end{pmatrix}\begin{pmatrix}
            A_{11}&A_{13}\\
            A_{31}&A_{33}
        \end{pmatrix}\begin{pmatrix}
            x_1\\
            x_3
        \end{pmatrix}+x_2^TA_{22}x_2+2\begin{pmatrix}
            x_1^T&x_3^T
        \end{pmatrix}\begin{pmatrix}
            A_{12}\\
            A_{32}
        \end{pmatrix}x_2\\
    =&\abs{A_{22}^{1/2}x_2+A_{22}^{-1/2}\begin{pmatrix}
        A_{21}&A_{23}
    \end{pmatrix}\begin{pmatrix}
            x_1\\
            x_3
        \end{pmatrix}}^2-\abs{A_{22}^{-1/2}\begin{pmatrix}
        A_{21}&A_{23}
    \end{pmatrix}\begin{pmatrix}
            x_1\\
            x_3
        \end{pmatrix}}^2+\abs{\begin{pmatrix}
            A_{11}&A_{13}\\
            A_{31}&A_{33}
        \end{pmatrix}^{1/2}\begin{pmatrix}
            x_1\\
            x_3
        \end{pmatrix}}^2\\
        =&\begin{pmatrix}
            x_1^T&
            x_3^T
        \end{pmatrix}\left(\begin{pmatrix}
            A_{11}&A_{13}\\
            A_{31}&A_{33}
        \end{pmatrix}-\begin{pmatrix}
            A_{12}\\
            A_{32}
        \end{pmatrix}A_{22}^{-1}\begin{pmatrix}
            A_{21}&A_{23}
        \end{pmatrix}\right)\begin{pmatrix}
            x_1\\
            x_3
        \end{pmatrix}+x^T\begin{pmatrix}
            A_{12}\\
            A_{22}\\
            A_{23}
        \end{pmatrix}A_{22}^{-1}\begin{pmatrix}
            A_{21}&A_{22}&A_{23}
        \end{pmatrix}x.
    \end{align*}
Hence
\begin{align*}
    x^TA^{-1}x\geq&\begin{pmatrix}
        x_1^T&x_3^T
    \end{pmatrix}^TA_{k+m}^{-1}\begin{pmatrix}
        x_1\\
        x_3
    \end{pmatrix}+C^{-1}\abs{\begin{pmatrix}
        A_{21}&A_{22}&A_{23}
    \end{pmatrix}x}^2.
\end{align*}
Consequently,  applying this estimate to $x_i=\nabla\left(U_i-\Phi_i\right)$ and using \eqref{eq:TM''} again, we obtain
\begin{align*}
    C\geq&\dfrac12\sum_{i,j=1}^{k+m}a^{ij}_{k+m+1}\int_{\Sigma}\hin{\nabla\left(\hat u_{i}-\hat \phi_{i}\right)}{\nabla\left(\hat u_{j}-\hat\phi_{j}\right)}\dif\mu_{\Sigma}+4\pi\sum_{i=1}^{k+m}\left(\overline{\hat u_{i}-\hat \phi_{i}}-\ln\fint_{\Sigma}e^{\hat u_{i}-\hat\phi_{i}}\dif\mu_{\Sigma}\right)\\
    &+C^{-1}\sum_{i=k+1}^{N-m}\int_{\Sigma}\abs{\nabla\left(U_i-\Phi_i\right)}^2\dif\mu_{\Sigma}.
\end{align*}
Thus
\begin{align*}
    \norm{\nabla U_i}_{L^2\left(\Sigma\right)}\leq C,\quad k+1\leq i\leq N-m.
\end{align*}

Combining the two cases and proceeding by induction, we finally obtain
\begin{align*}
\rho_i<4\pi\quad\Longrightarrow\quad \norm{\nabla U_i}_{L^2\left(\Sigma\right)}\leq C,
\end{align*}
which completes the proof.

\end{proof}

\section{Blowup analysis}

From now on, we assume $A=A_N$ is the Cartan matrix associated with $\mathrm{SU}(N+1)$. In other words, we shall focus on the Toda flow \eqref{eq:toda}. Furthermore, we assume $0<\rho_i<4\pi$ for all $i\in I\setminus\set{k}$ for some $k\in I$, and $\rho_k=4\pi$. According to  \autoref{thm:global}, the Toda flow \eqref{eq:toda} exists globally.

Define
\begin{align*}
    M_i(t)=\max_{\Sigma}u_i(t),\quad i\in I.
\end{align*}
Using Jensen's inequality, we have
\begin{align*}
    \bar u_i(t)\leq\ln\fint_{\Sigma}e^{u_i(t)}\dif\mu_{\Sigma}\leq M_i(t),\quad i\in I.
\end{align*}
Together with \eqref{eq:conservation}, we obtain
\begin{align*}
    M_i(t)\geq-C,\quad \bar u_i(t)\leq C,\quad i\in I.
\end{align*}
\begin{lem}\label{lem:blowup}
    Along the Toda flow \eqref{eq:toda}, the following three statements are equivalent:
\begin{enumerate}[(1)]
    \item $\sum_{i=1}^NM_i(t)\leq C$,
    \item $\norm{\nabla u(t)}_{L^2\left(\Sigma\right)}\leq C$,
    \item $\sum_{i=1}^N\bar u_i(t)\geq-C$.
\end{enumerate}
\end{lem}
\begin{proof}
From \eqref{eq:TM''} we have 
\begin{align*}
    -C\leq \dfrac14\sum_{i,j=1}^N\int_{\Sigma}a^{ij}\hin{\nabla u_i}{\nabla u_j}\dif\mu_{\Sigma}+4\pi\sum_{i=1}^N\bar u_i\leq C.
\end{align*}
Hence the last two statements are equivalent. 

If the first statement holds, i.e.,
\begin{align*}
    \sum_{i=1}^NM_i(t)\leq C,
\end{align*}
then in particular 
\begin{align*}
    u_i(t)\leq C,\quad\forall i\in I.
\end{align*}
Standard elliptic estimates then give
\begin{align*}
    \norm{u(t)-\bar u(t)}_{H^2\left(\Sigma\right)}\leq C,
\end{align*}
and consequently
\begin{align*}
    \norm{\nabla u(t)}_{L^2\left(\Sigma\right)}\leq C,
\end{align*}
which is the second statement.

If the second statement holds, then
\begin{align*}
    \norm{u(t)}_{H^1\left(\Sigma\right)}\leq C.
\end{align*}
 Using elliptic regularity again we obtain
\begin{align*}
    \norm{u(t)}_{H^2\left(\Sigma\right)}\leq C.
\end{align*}
By the Sobolev embedding theorem,
\begin{align*}
    \norm{u(t)}_{L^{\infty}\left(\Sigma\right)}\leq C,
\end{align*}
which implies the first statement. 
\end{proof}

We say that \emph{$\set{u(t)}$ blows up} if either one of the statement does not hold in \autoref{lem:blowup}.

For the remainder of this section, we will consider that the Toda flow \eqref{eq:toda} does not converge as time approaches infinity. According to \autoref{thm:convergence-general}, we know that
\begin{align*}
    \limsup_{t\to\infty}\norm{u(t)}_{H^1\left(\Sigma\right)}=\infty.
\end{align*}
In other words, the $\set{u(t)}$ blows up. By the assumption and \autoref{lem:blowup}, we know that
\begin{align*}
    \bar u_i(t)\geq-C,\quad i\neq k,\\
    \liminf_{t\to\infty}\bar u_k(t)=-\infty.
\end{align*}

Fixed a positive number $\epsilon$, for each $i\in I$, we consider the following singular set 
\begin{align*}
    S_i(\epsilon)=\set{p\in\Sigma:\limsup_{\delta\to0}\limsup_{t\to\infty}\int_{B^{\Sigma}_{\delta}(p)}e^{u_i(t)}\dif\mu_{\Sigma}\geq\epsilon}.
\end{align*}
Denote $S(\epsilon)=\cup_{i=1}^NS_i(\epsilon)$ by the total singular set. One can check that $S(\epsilon)$ is a finite subset of $\Sigma$. To see this, consider a set $\set{p_1, \dots, p_m} \subset S_i(\epsilon)$. We select $\delta_0 > 0$ and $t_0 > 0$ such that the following condition holds:
\begin{align*}
    B^{\Sigma}_{\delta_0}(p_i) \cap B^{\Sigma}_{\delta_0}(p_j) = \emptyset, \quad 1 \leq i < j \leq m.
\end{align*}
Moreover, we ensure that
\begin{align*}
    \int_{B^{\Sigma}_{\delta_0}(p_i)} e^{u_i(t_0)} \dif\mu_{\Sigma} \geq \frac{\epsilon}{2}, \quad 1 \leq i \leq m.
\end{align*}
Consequently, we deduce
\begin{align*}
    \dfrac{m\epsilon}{2} \leq \sum_{k=1}^m \int_{B^{\Sigma}_{\delta_0}(p_k)} e^{u_i(t_0)} \dif\mu_{\Sigma} = \int_{\cup_{k=1}^m B^{\Sigma}_{\delta_0}(p_k)} e^{u_i(t_0)} \dif\mu_{\Sigma} \leq \int_{\Sigma} e^{u_i(t_0)} \dif\mu_{\Sigma} \leq C.
\end{align*}
As a result, it follows that
\begin{align*}
    \#S(\epsilon) \leq \dfrac{C}{\epsilon}.
\end{align*}

Recall Brezis-Merle's estimate \cite[Theorem 1]{BreMer91uniform}. 
\begin{lem}[cf. \cite{DinJosLiWan97differential}]\label{lem:BM}Let $\Omega\subset\Sigma$ be a smooth domain. Assume $u\in H^1(\Omega)$ is a weak solution to
\begin{align*}
\begin{cases}
    -\Delta u=f,&\text{in}\ \Omega,\\
    u=0,&\text{on}\ \partial\Omega,
\end{cases}
\end{align*}
where $f\in L^1(\Omega)$. For every $0<\delta<4\pi$, there is a constant $C$ depending only on $\delta$ and $\Omega$ such that
\begin{align*}
    \int_{\Omega}\exp\left(\dfrac{(4\pi-\delta)\abs{u}}{\norm{f}_{L^1(\Omega)}}\right)\dif\mu_{\Sigma} \leq C.
\end{align*}
\end{lem}

As a consequence of \autoref{lem:BM}, we have the following

\begin{lem}\label{lem:basic} There exists a positive number $\epsilon_0$ such that for every compact subset $K\subset\Sigma\setminus S(\epsilon_0)$, there exists a positive constant $C_K$ such that
\begin{align}\label{eq:basic}
    \norm{u(t)-\bar u(t)}_{L^{\infty}(K)}\leq C_K.
\end{align}

\end{lem}

\begin{proof}The proof is standard and we refer the reader to  \cite[Lemma 2.8]{DinJosLiWan97differential} for a single equation. In fact, for each $p_0\notin S$, by definition,
\begin{align*}
    \limsup_{\delta\to0}\limsup_{t\to\infty}\int_{B^{\Sigma}_{\delta}(p_0)}e^{u_i(t)}\dif\mu_{\Sigma}<\epsilon.
\end{align*}
According to \eqref{eq:conservation}, \eqref{eq:uniformly_regularity} and \eqref{eq:t-lp},  
we may assume for some positive numbers $\epsilon_0, \delta$ and $r$,  $B_{4r}^{\Sigma}(p_0)\subset\Sigma\setminus S(\epsilon_0)$
\begin{align}\label{eq:BM-1}
    \int_{B_{4r}^{\Sigma}(p_0)}\abs{\sum_{j=1}^Na_{ij}\left(\rho_j\dfrac{h_je^{u_j(t)}}{\int_{\Sigma}h_je^{u_j(t)}\dif\mu_{\Sigma}}-Q_j-\dfrac{\partial e^{u_j(t)}}{\partial t}\right)}\dif\mu_{\Sigma}\leq4\pi-2\delta,\quad \forall i\in I.
\end{align} 
Solve
\begin{align}\label{eq:w}
    \begin{cases}
        -\Delta w_i(t)=\sum_{j=1}^Na_{ij}\left(\rho_j\frac{h_je^{u_j(t)}}{\int_{\Sigma}h_je^{u_j(t)}\dif\mu_{\Sigma}}-Q_j-\frac{\partial e^{u_j(t)}}{\partial t}\right),&\text{in}\ B_{4r}^{\Sigma}(p_0),\\
        w_i(t)=0,&\text{on}\ \partial B_{4r}^{\Sigma}(p_0).
    \end{cases}
\end{align}
According to \autoref{lem:BM} and \eqref{eq:BM-1}, we have
\begin{align*}
    \norm{e^{\abs{w_i(t)}}}_{L^{p}(B_{4r}^{\Sigma}(p_0))}\leq C,\quad p=\dfrac{4\pi-\delta}{4\pi-2\delta},\quad i\in I.
\end{align*}
In particular, $w(t)$ is uniformly bounded in $L^1(B_{4r}^{\Sigma}(p_0))$. 
 Since $h(t)\coloneqq u(t)-\bar u(t)-w(t)$ is harmonic in $B_{4r}^{\Sigma}(p_0)$, we have by the mean value theorem for harmonic functions
\begin{align*}
    \norm{h(t)}_{L^{\infty}(B_{2r}^{\Sigma}(p_0))}\leq&C\norm{h(t)}_{L^{1}(B_{4r}^{\Sigma}(p_0))}\\
    \leq&C(\norm{u-\bar u(t)}_{L^{1}(B_{4r}^{\Sigma}(p_0))}+\norm{w(t)}_{L^{1}(B_{4r}^{\Sigma}(p_0))})\\
    \leq&C(\norm{u(t)-\bar u(t)}_{L^{1}\left(\Sigma\right)}+\norm{w(t)}_{L^{1}(B_{4r}^{\Sigma}(p_0))})\\
    \leq&C.
\end{align*}
Notice that $\bar u_i(t)\leq C$ for each $i$ by Jensen's inequality. We obtain
\begin{align*}
    \norm{e^{u_i(t)}}_{L^{p}(B_{2r}^{\Sigma}(p_0))}\leq C\norm{e^{w_i(t)}}_{L^{p}(B_{2r}^{\Sigma}(p_0))}\leq C,\quad\forall i\in I.
\end{align*}
Applying the standard elliptic estimates for \eqref{eq:w}, we get
\begin{align*}
    \norm{w(t)}_{L^{\infty}(B_{r}^{\Sigma}(p_0))}\leq C.
\end{align*}
Hence
\begin{align*}
    \norm{u(t)-\bar u(t)}_{L^{\infty}(B_{r}^{\Sigma}(p_0))}\leq&\norm{h(t)+w(t)}_{L^{\infty}(B_{r}^{\Sigma}(p_0))}\\
    \leq&\norm{h(t)}_{L^{\infty}(B_{r}^{\Sigma}(p_0))}+\norm{w(t)}_{L^{\infty}(B_{r}^{\Sigma}(p_0))}\\
    \leq&C.
\end{align*}
We obtain the desired estimate \eqref{eq:basic}. 

\end{proof}

\begin{lem}\label{lem:singular-set} The set $S(\epsilon_0)=\set{p_0}$ comprises a solitary element. Moreover, $h_{k}(p_0)>0$.
\end{lem}
\begin{proof}
    Since $\set{u(t)}$ blows up, according to \autoref{lem:basic}, we know that $S(\epsilon_0)$ is nonempty. 

Fixed $p_0\in S(\epsilon_0)$. Choose $\delta>0$ such that $B^{\Sigma}_{2\delta}(p_0)\cap S(\epsilon_0)=\set{p_0}$. 
We have
\begin{align*}
    \limsup_{t\to\infty}\max_{1\leq i\leq N}\max_{\overline{B^{\Sigma}_{\delta}(p_0)}}u_i(t)=\infty.
\end{align*}
Choose $t_n\to\infty$ and $p_n\in \overline{B^{\Sigma}_{\delta}(p_0)}$ and $j\in\set{1,2,\dots,N}$ such that
\begin{align*}
    M_n\coloneqq u_j(t_n,p_n)=\max_{\overline{B^{\Sigma}_{\delta}(p_0)}}u_j(t_n)=\max_{1\leq i\leq N}\max_{\overline{B^{\Sigma}_{\delta}(p_0)}}u_i(t_n)\to\infty.
\end{align*}
One can check that $p_n$ converges to $p_0$. For otherwise, we may assume, up to a subsequence, $p_n$ converges to $\tilde p_0\neq p_0$, then $\tilde p_0$ is a regular point in the sense that $\tilde p_0\in\Sigma\setminus S(\epsilon_0)$. Notice that $M_n$ is not bounded from above. However, the estimate stated in \autoref{lem:basic} claims that $u(t_n)$ is uniformly bounded from below in a neighborhood  of $\tilde p_0$ which is impossible. 

Without loss of generality,  we assume identify $B^{\Sigma}_{2\delta}(p_0)=B_{2\delta}\subset\mathbb{R}^2$ and the coordinate of $p_n$ is $x_n$.  
We scale $u_i(t_n)$ by
\begin{align*}
    \tilde u_i(t_n)(x)=u_i\left(t_n,x_n+e^{-M_n/2}x\right)-M_n,\quad \abs{x}<e^{M_n/2}\left(\delta-\abs{x_n}\right).
\end{align*}
We have
\begin{align*}
    e^{-M_n/2}e^{\tilde u_i(t_n)}\dfrac{\partial u_i(t_n)}{\partial t}=\sum_{j=1}^Na^{ij}\Delta_{\mathbb{R}^2}\tilde u_j(t_n)+\rho_i\dfrac{h_i\left(x_n+e^{-M_n/2}\cdot\right)e^{\tilde u_i(t_n)}}{\int_{\Sigma}h_ie^{u_i(t_n)}\dif\mu_{\Sigma}}-e^{-M_n}Q_i.
\end{align*}
Moreover
\begin{align*}
    \int_{B_{e^{M_n/2}\left(\delta-\abs{x_n}\right)}(0)}e^{\tilde u_i(t_n)}\dif\mu_{\mathbb{R}^2}\leq\int_{\Sigma}e^{u_i(t_n)}\dif\mu_{\Sigma}\leq C.
\end{align*}
Then a standard blowup analysis together with the classification of the Toda system \autoref{athm:classification} gives $j=k$ and $h_{k}(p_0)>0$.  Consequently, $S(\epsilon_0)=S_k(\epsilon_0)$.

According to an improved Trudinger-Moser inequality, we know that $\#S_k(\epsilon_0)\leq 1$. In fact, if $\#S_k(\epsilon_0)\geq 2$, then we can choose two different points $p_1, p_2\in S_k(\epsilon_0)$ and $\delta>0$ such that $\mathrm{dist}(p_1,p_2)\geq 3\delta$ and 
\begin{align*}
    \int_{B^{\Sigma}_{\delta}(p_1)}e^{u_{k}(t_n)-\phi}\dif\mu_{\Sigma}\geq\dfrac{\epsilon_0}{2},\quad \int_{B^{\Sigma}_{\delta}(p_2)}e^{u_{k}(t_n)-\phi}\dif\mu_{\Sigma}\geq\dfrac{\epsilon_0}{2}
\end{align*}
for some sequence $t_n\to\infty$. Then one can use an improved Trudinger-Moser inequality (see, for example, \cite[Lemma 2.2]{CheLi91prescribing}, \cite[Theorem 2.1]{DinJosLiWan99existence} or \cite[Lemma 2.2]{LiSunYan23boundary})
\begin{align*}
    \ln\fint_{\Sigma}e^{u_{k}(t_n)-\phi}\dif\mu_{\Sigma}\leq\dfrac{1}{24\pi}\int_{\Sigma}\abs{\nabla \left(u_k(t_n)-\phi\right)}^2\dif\mu+\fint_{\Sigma}(u_{k}(t_n)-\phi)\dif\mu_{\Sigma}+C. 
\end{align*}
However, according to the proof of \autoref{thm:partial-H^1}, we know that
\begin{align*}
    J(u(t_n))\geq&\dfrac14\int_{\Sigma}\abs{\nabla \left(u_k(t_n)-\phi\right)}^2\dif\mu+4\pi\left(\fint_{\Sigma}\left(u_k(t_n)-\phi\right)\dif\mu_{\Sigma}-\ln\fint_{\Sigma}e^{u_{k}(t_n)-\phi}\dif\mu_{\Sigma}\right)-C.
\end{align*}
Applying  \eqref{eq:monotonicity}, we conclude that
\begin{align*}
    \bar u_{k}(t_n)\geq-C
\end{align*}
which is a contradiction. Therefore, $S(\epsilon_0)=\set{p_0}$.

\end{proof}

\begin{lem}\label{lem:upper}
    The following estimate holds
    \begin{align*}
        \abs{u_i(t)}\leq C,\quad i<k-1\ or\  i>k+1,\\
        u_{k-1}(t)\leq C,\quad u_{k+1}(t)\leq C.
    \end{align*}
\end{lem}
\begin{proof}
    Firstly, we show that for every $p>0$
    \begin{align}\label{eq:uniformly-l2}
        \int_{\Sigma}e^{pu_i(t)}\dif\mu_{\Sigma}\leq C_p
    \end{align}
    provided $i\neq k$. It suffices to prove
    \begin{align*}
        \int_{B^{\Sigma}_{\delta}(p_0)}e^{pu_{k-1}(t)}\dif\mu_{\Sigma}\leq C_p,\quad \int_{B^{\Sigma}_{\delta}(p_0)}e^{pu_{k+1}(t)}\dif\mu_{\Sigma}\leq C_p
    \end{align*}
    for some $\delta>0$ and $p_0\in S(\epsilon_0)$. In fact, since
    \begin{align*}
        u_i=\sum_{j=1}^Na_{ij}U_j,
    \end{align*}
    according to \autoref{thm:partial-H^1}, we have
    \begin{align*}
        \norm{u_i(t)}_{H^1\left(\Sigma\right)}\leq C,\quad i<k-1\ \text{or}\ i>k+1.
    \end{align*}
    The classical Trudinger-Moser inequality \eqref{eq:classical-TM} implies
    \begin{align*}
        \int_{\Sigma}e^{pu_i(t)}\dif\mu_{\Sigma}\leq C_p,\quad i<k-1\ \text{or}\ i>k+1.
    \end{align*}
    Moreover, since
    \begin{align*}
        \abs{\bar u_{i}(t)}\leq C,\quad\forall i\neq k,
    \end{align*}
    according to \autoref{lem:basic}, we know that the functions $u_{k-1}(t)$ and $u_{k+1}(t)$ is uniformly bounded in $L^{\infty}_{loc}(\Sigma\setminus S(\epsilon_0))$. 

Since 
\begin{align*}
    -\Delta U_i=\rho_i\dfrac{h_ie^{u_i(t)}}{\int_{\Sigma}h_ie^{u_i(t)}\dif\mu_{\Sigma}}-Q_i-\dfrac{\partial e^{u_i(t)}}{\partial t}
\end{align*}
and for $i\neq k$
\begin{align*}
    \norm{\nabla U_i(t)}_{L^2\left(\Sigma\right)}\leq C,
\end{align*}
we know that
\begin{align*}
    \norm{\rho_i\dfrac{h_ie^{u_i(t)}}{\int_{\Sigma}h_ie^{u_i(t)}\dif\mu_{\Sigma}}-\dfrac{\partial e^{u_i(t)}}{\partial t}}_{H^{-1}\left(\Sigma\right)}\leq C.
\end{align*}
In particular, for every Lipschitz function
 \begin{align*}
     \eta(p)=\begin{cases}
         1,&r(p)\coloneqq\mathrm{dist}(p,p_0)\leq\delta,\\
         1+\frac{\ln r(p)-\ln\delta }{\ln\sqrt{\delta}},&\delta<r(p)<\sqrt{\delta},\\
         0,&r(p)\geq\sqrt{\delta},
     \end{cases}
 \end{align*}
 where $p_0\in \Sigma$ and $0<\delta<1$, we have
\begin{align*}
\abs{\int_{\Sigma}\rho_i\dfrac{h_ie^{u_i(t)}}{\int_{\Sigma}h_ie^{u_i(t)}\dif\mu_{\Sigma}}\eta\dif\mu_{\Sigma}}\leq C\norm{\eta}_{H^1\left(\Sigma\right)}+\norm{\dfrac{\partial e^{u_i(t)}}{\partial t}}_{L^1\left(\Sigma\right)}\leq C\left(\sqrt{\delta}+\dfrac{1}{\abs{\ln\delta}^{1/2
}}\right)+\norm{\dfrac{\partial e^{u_i(t)}}{\partial t}}_{L^1\left(\Sigma\right)}.
\end{align*}
Thus
\begin{align*}
    \abs{\int_{B^{\Sigma}_{\delta}(p_0)}\rho_i\dfrac{h_ie^{u_i(t)}}{\int_{\Sigma}h_ie^{u_i(t)}\dif\mu_{\Sigma}}\dif\mu_{\Sigma}}\leq&\abs{\int_{B^{\Sigma}_{\sqrt{\delta}}(p_0)\setminus B^{\Sigma}_{\delta}(p_0)}\rho_i\dfrac{h_ie^{u_i(t)}}{\int_{\Sigma}h_ie^{u_i(t)}\dif\mu_{\Sigma}}\eta\dif\mu_{\Sigma}}+C\left(\sqrt{\delta}+\dfrac{1}{\abs{\ln\delta}^{1/2
}}\right)+\norm{\dfrac{\partial e^{u_i(t)}}{\partial t}}_{L^1\left(\Sigma\right)}.
\end{align*}
According to \autoref{lem:basic}, we know that for every $p_0\in\Sigma$
\begin{align}\label{eq:no-S_i}
    \lim_{\delta\to0}\lim_{t\to\infty}\abs{\int_{B^{\Sigma}_{\delta}(p_0)}\rho_i\dfrac{h_ie^{u_i(t)}}{\int_{\Sigma}h_ie^{u_i(t)}\dif\mu_{\Sigma}}\dif\mu_{\Sigma}}=0,\quad i\neq k.
\end{align}

If $k>1$, then the functions $u_{k-1}(t)$ satisfies
\begin{align*}
    -\Delta u_{k-1}(t)=&2\left(\rho_{k-1}\dfrac{h_{k-1}e^{u_{k-1}(t)}}{\int_{\Sigma}h_{k-1}e^{u_{k-1}(t)}\dif\mu_{\Sigma}}-Q_{k-1}-\dfrac{\partial e^{u_{k-1}(t)}}{\partial t}\right)-\left(\rho_{k-2}\dfrac{h_{k-2}e^{u_{k-2}(t)}}{\int_{\Sigma}h_{k-2}e^{u_{k-2}(t)}\dif\mu_{\Sigma}}-Q_{k-2}-\dfrac{\partial e^{u_{k-2}(t)}}{\partial t}\right)\\
    &-\left(\rho_{k}\dfrac{h_{k}e^{u_{k}(t)}}{\int_{\Sigma}h_{k}e^{u_{k}(t)}\dif\mu_{\Sigma}}-Q_{k}-\dfrac{\partial e^{u_{k}(t)}}{\partial t}\right).
 \end{align*}
 Since $h_k(p_0)>0$, we may assume in a neighborhood $B^{\Sigma}_{\delta}(p_0)$ of $p_0$
 \begin{align*}
     -\Delta u_{k-1}(t)\leq f_{k-1}(t)\coloneqq&2\left(\rho_{k-1}\dfrac{h_{k-1}e^{u_{k-1}(t)}}{\int_{\Sigma}h_{k-1}e^{u_{k-1}(t)}\dif\mu_{\Sigma}}-Q_{k-1}-\dfrac{\partial e^{u_{k-1}(t)}}{\partial t}\right)-\left(\rho_{k-2}\dfrac{h_{k-2}e^{u_{k-2}(t)}}{\int_{\Sigma}h_{k-2}e^{u_{k-2}(t)}\dif\mu_{\Sigma}}-Q_{k-2}-\dfrac{\partial e^{u_{k-2}(t)}}{\partial t}\right)\\
    &-\left(-Q_{k}-\dfrac{\partial e^{u_{k}(t)}}{\partial t}\right),
 \end{align*}
 and according to \eqref{eq:no-S_i}
 \begin{align*}
     \norm{f_{k-1}(t)}_{L^1\left(B^{\Sigma}_{\delta}(p_0)\right)}<\dfrac{4\pi}{p}.
 \end{align*}
 Solve
 \begin{align*}
 \begin{cases}
     -\Delta\tilde w_{k-1}(t)=f_{k-1}(t),&\text{in}\ B^{\Sigma}_{\delta}(p_0),\\
     \tilde w_{k-1}(t)=0,&\text{on}\ \partial B^{\Sigma}_{\delta}(p_0).
 \end{cases}
 \end{align*}
 We have from \autoref{lem:BM}
 \begin{align*}
     \int_{B^{\Sigma}_{\delta}(p_0)}e^{p\abs{\tilde w_{k-1}(t)}}\dif\mu_{\Sigma}\leq C_p.
 \end{align*}
 Since the functions $\tilde h_{k-1}(t)=u_{k-1}(t)-\tilde w_{k-1}(t)$ satisfies
 \begin{align*}
     \begin{cases}
     -\Delta\tilde h_{k-1}(t)\leq0,&\text{in}\ B^{\Sigma}_{\delta}(p_0),\\
     \tilde h_{k-1}(t)=u_{k-1}(t)\leq C,&\text{on}\ \partial B^{\Sigma}_{\delta}(p_0).
 \end{cases}
 \end{align*}
 We conclude that
 \begin{align*}
     u_{k-1}(t)\leq\tilde w_{k-1}(t)+C,\quad \text{in}\ B^{\Sigma}_{\delta}(p_0).
 \end{align*}
 Consequently,
 \begin{align*}
     \int_{B^{\Sigma}_{\delta}(p_0)}e^{pu_{k-1}(t)}\dif\mu_{\Sigma}\leq C_p.
 \end{align*}
 Hence, if $k>1$, then 
 \begin{align*}
     \int_{\Sigma}e^{pu_{k-1}(t)}\dif\mu_{\Sigma}\leq C_p.
 \end{align*}
 Similarly, if $k<N$, then we have
  \begin{align*}
     \int_{\Sigma}e^{pu_{k+1}(t)}\dif\mu_{\Sigma}\leq C_p.
 \end{align*}

 Secondly, we show that
 \begin{align*}
     u_i(t)\leq C,\quad i\neq k.
 \end{align*}
 Notice that
 \begin{align*}
     -\Delta U_i=\rho_i\dfrac{h_ie^{u_i}}{\partial t}-Q_i-\dfrac{\partial e^{u_i}}{\partial t}.
 \end{align*}
 According to \eqref{eq:conservation}, \eqref{eq:t-lp} and \eqref{eq:uniformly-l2},  we have for every $p>1$
 \begin{align*}
     \norm{\dfrac{\partial e^{u_i}}{\partial t}}_{L^p\left(\Sigma\right)}\leq \norm{e^{(2p-1)u_i/(2p)}}_{L^{2p}\left(\Sigma\right)}\norm{e^{u_i/(2p)}\dfrac{\partial u_i}{\partial t}}_{L^{2p}\left(\Sigma\right)}\to 0,\quad i\neq k.
 \end{align*}
 A standard elliptic estimate implies
 \begin{align*}
     \norm{U_i-\bar U_i}_{W^{2}_p\left(\Sigma\right)}\leq C_p,\quad i\neq k.
 \end{align*}
 Consequently, for every $\alpha\in(0,1)$
 \begin{align*}
     \norm{u_i(t)}_{C^{1+\alpha}\left(\Sigma\right)}\leq C,\quad i<k-1\ \text{or}\ i>k+1,\\
     \norm{2u_{k-1}(t)+u_{k}(t)-\bar u_{k}(t)}_{C^{1+\alpha}\left(\Sigma\right)}\leq C,\quad \norm{2u_{k+1}(t)+u_{k}(t)-\bar u_{k}(t)}_{C^{1+\alpha}\left(\Sigma\right)}\leq C.
 \end{align*}
In particular,
\begin{align*}
    \abs{u_i(t)}\leq C,\quad i<k-1\ or\ i>k+1,\\
    2u_{k-1}(t)+u_{k}(t)\leq C,\quad 2u_{k+1}(t)+u_{k}(t)\leq C
\end{align*}

 Choose $p_t\in\Sigma$ such that
 \begin{align*}
     u_{k-1}(t,p_t)=\max_{\Sigma}u_{k-1}(t).
 \end{align*}
 We have for $k>1$
 \begin{align*}
     2u_{k-1}(t,p_t)+u_{k}(t,p_t)\leq C.
 \end{align*}
 If $u_{k-1}(t)$ is not bounded from above, then we have for some sequence $t_n\to\infty$
 \begin{align*}
     \lim_{n\to\infty}u_{k-1}(t_n,p_{t_n})=\infty
 \end{align*}
 which implies
 \begin{align*}
     \lim_{n\to\infty}u_{k}(t_n,p_{t_n})=-\infty.
 \end{align*}
 Consequently, we may assume for some positive number $\delta$ and $t_0$
 \begin{align*}
     \set{p_t:t\geq t_0}\subset\Sigma\setminus B^{\Sigma}_{\delta}(p_0).
 \end{align*}
 This is impossible according to \autoref{lem:basic}. In other words, we obtain
 \begin{align*}
     u_{k-1}(t)\leq C.
 \end{align*}
 Similar argument yields for $k<N$
 \begin{align*}
     u_{k+1}(t)\leq C.
 \end{align*}

\end{proof}

\section{The lower bound of the functional}

In this section, we shall deriving an explicit lower bound of the functional $J(u(t))$ when $\set{u(t)}$ blows up along the Toda flow \eqref{eq:toda}. 

Let $G(\cdot,p_0)=(G_1(\cdot,p_0),\dots,G_N(\cdot,p_0))\in W^1_1\left(\Sigma,\mathbb{R}^N\right)$ be a weak solution to the following system
\begin{align*}
    \begin{cases}
    -\sum_{j=1}^Na^{ij}\Delta G_j(\cdot,p_0)=\rho_i\frac{h_ie^{G_i(\cdot,p_0)}}{\int_{\Sigma}h_ie^{G_i(\cdot,p_0)}\dif\mu_{\Sigma}}-Q_i,\quad\bar G_i(\cdot, p_0)=0,\quad i\neq k,\\
        -\sum_{j=1}^Na^{kj}\Delta G_j(\cdot,p_0)=4\pi\delta_{p_0}-Q_k,\quad \bar G_k(\cdot,p_0)=0.
    \end{cases}
\end{align*}
Moreover, we assume
\begin{align*}
    \int_{\Sigma}h_ie^{G_i(\cdot,p_0)}\dif\mu_{\Sigma}>0,\quad \int_{\Sigma}e^{G_i(\cdot,p_0)}\dif\mu_{\Sigma}<\infty,\quad i\neq k.
\end{align*}
Denote by
\begin{align*}
    H_i(\cdot,p_0)=\sum_{j=1}^Na^{ij}G_j(\cdot,p_0),\quad i\in I.
\end{align*}
We have
\begin{align*}
\begin{cases}
    -\Delta H_i(\cdot,p_0)=\rho_i\frac{h_ie^{G_i(\cdot,p_0)}}{\int_{\Sigma}h_ie^{G_i(\cdot,p_0)}\dif\mu_{\Sigma}}-Q_i,\quad \bar H_i(\cdot,p_0)=0,\quad i\neq k,\\
    -\Delta H_k(\cdot,p_0)=4\pi\delta_{p_0}-Q_k,\quad\bar H_k(\cdot,p_0)=0.
\end{cases}
\end{align*}

It is well know that we have the following expression of $H_k(\cdot,p_0)$ in a normal coordinates $\set{x}$ which is centered at $p_0$
\begin{align*}
    H_k(x,p_0)=-2\ln\abs{x}+R_k+\hin{\beta_k}{x}+O\left(\abs{x}^2\right).
\end{align*}
Consequently, $e^{-H_k(\cdot,p_0)}$ is smooth, and we conclude that $H_i(\cdot,p_0)$ are smooth if $i\neq k$. In fact, if we denote for  $i\in I$
\begin{align*}
    \underline{H_i}=\begin{cases}
        H_i(\cdot,p_0),&i\neq k,\\
        0,&i=k,
    \end{cases}\quad \underline{Q_i}=\begin{cases}
        Q_i,&i\neq k,\\
        0,&i=k,
    \end{cases}\quad \underline{\rho_i}=\begin{cases}
        \rho_i,&i\neq k,\\
        0,&i=k,
    \end{cases}
\end{align*}
and
\begin{align*}
    \underline{h_i}=\begin{cases}
        h_i,& i<k-1\ or\ i>k+1,\\
        h_ie^{-H_k(\cdot,p_0)},&i=k-1\ or\ i=k+1,\\
        1,&i=k,
    \end{cases}\quad \underline{G_i}=\sum_{j=1}^Na_{ij}\underline{H_j}(\cdot,p_0)=G_i(\cdot,p_0)+\begin{cases}
        0,&i<k-1\ or\ i>k+1,\\
        H_k(\cdot,p_0),& i=k-1 \ or\ i=k+1,\\
        -2H_k(\cdot,p_0),& i=k,
    \end{cases}
\end{align*}
then we have
\begin{align*}
    -\Delta\underline{H_i}=\underline{\rho_i}\dfrac{\underline{h_i}e^{\underline{G_i}}}{\int_{\Sigma}\underline{h_i}e^{\underline{G_i}}\dif\mu_{\Sigma}}-\underline{Q_i},\quad\overline{\underline{G_i}}=0,\quad i\in I.
\end{align*}
Thus, the Brezis-Merle argument together with the standard  elliptic estimates yields that as $\underline{H_i}$ are smooth functions. 
In other words,
\begin{align*}
    G_i(\cdot,p_0)\in C^{\infty}\left(\Sigma\right),\quad i<k-1\ or\ i>k+1,\\
    G_{i}(\cdot,p_0)+H_k(\cdot,p_0)\in C^{\infty}\left(\Sigma\right),\quad i=k-1\ or\ i=k+1,\\
    G_{k}(\cdot,p_0)-2H_k(\cdot,p_0)\in C^{\infty}\left(\Sigma\right).
\end{align*}

To state our main result in this section, we consider the following functional
\begin{align}\label{eq:lower-func-1}
    J_k(u)\coloneqq J(u)-\left(\dfrac14\int_{\Sigma}\abs{\nabla u_k}^2\dif\mu_{\Sigma}+\int_{\Sigma}Q_ku_k\dif\mu_{\Sigma}-4\pi\ln\int_{\Sigma}h_ke^{u_k}\dif\mu_{\Sigma}\right),\quad u\in H^1\left(\Sigma,\mathbb{R}^N\right).
\end{align}
Then the  primary results of this section are presented as follows.
\begin{theorem}\label{thm:lower}
   If the Toda flow \eqref{eq:toda} is not convergent at infinity, the following inequality holds:
\begin{align*}
J(u(t))\geq&J_k(G)+\dfrac14\int_{\Sigma}\abs{\nabla\left(H_{k-1}+H_{k+1}\right)}^2\dif\mu_{\Sigma}+\int_{\Sigma}Q_k H_k\dif\mu_{\Sigma}\\
    &-4\pi\left(\lim_{p\to p_0}\left(H_k(p,p_0)+2\mathrm{dist}(p,p_0)\right)+\ln h_k(p_0)\right)-4\pi-4\pi\ln\pi.
\end{align*}
\end{theorem}
\begin{proof}
We have already prove from \eqref{eq:conservation}, \eqref{eq:uniformly_regularity} and \eqref{eq:t-lp} that
\begin{align*}
    \int_{\Sigma}e^{u_i(t)}\dif\mu_{\Sigma}=e^{u_{i,0}}\dif\mu_{\Sigma},\quad \int_{\Sigma}h_ie^{u_i(t)}\dif\mu_{\Sigma}\geq C^{-1},\quad \lim_{t\to\infty}\int_{\Sigma}e^{u_i(t)}\abs{\dfrac{\partial u_i}{\partial t}}^2\dif\mu_{\Sigma}=0,\quad i\in I
\end{align*}
and \autoref{lem:upper} that
\begin{align*}
   \abs{\bar u_i(t)}\leq C,\quad  u_i(t)\leq C,\quad\forall i\neq k.
\end{align*}

Applying H\"older's inequality, 
\begin{align*}
    \norm{\dfrac{\partial e^{u_i}}{\partial t}}_{L^1\left(\Sigma\right)}\leq\norm{e^{u_i/2}}_{L^2\left(\Sigma\right)}\norm{e^{u_i/2}\dfrac{\partial u_i}{\partial t}}_{L^2\left(\Sigma\right)},
\end{align*}
we conclude that
\begin{align*}
    \lim_{t\to\infty}\norm{\dfrac{\partial e^{u_i(t)}}{\partial t}}_{L^1\left(\Sigma\right)}=0.
\end{align*}
Since the functions $u_i(t)$ satisfy
\begin{align*}
    -\Delta u_i(t)=\sum_{j=1}^Na_{ij}\left(\rho_j\dfrac{h_je^{u_j(t)}}{\int_{\Sigma}h_je^{u_j(t)}\dif\mu_{\Sigma}}-Q_j-\dfrac{\partial e^{u_j(t)}}{\partial t}\right),\quad i\in I,
\end{align*}
we have
\begin{align*}
    \norm{\Delta u(t)}_{L^1\left(\Sigma\right)}\leq C,
\end{align*}
which implies from the standard potential estimate that for each $1<p<2$
\begin{align*}
    \norm{u(t)-\bar u(t)}_{W^{1}_p\left(\Sigma\right)}\le C_p.
\end{align*}
Notice that the functions $U_i(t)$ satisfy
\begin{align*}
    -\Delta U_i(t)=\rho_i\dfrac{h_ie^{u_i(t)}}{\int_{\Sigma}h_ie^{u_i(t)}\dif\mu_{\Sigma}}-Q_i-\dfrac{\partial e^{u_i(t)}}{\partial t}.
\end{align*}
We have
\begin{align*}
    \norm{\Delta U_i(t)}_{L^2\left(\Sigma\right)}\leq C,\quad i\neq k,
\end{align*}
which gives from the standard elliptic estimate
\begin{align*}
    \norm{U_i(t)-\bar U_i(t)}_{H^2\left(\Sigma\right)}\leq C,\quad i\neq k.
\end{align*}

Since $\set{u(t)}$ blows up, according to \autoref{lem:blowup}, we may choose a sequence of positive numbers
 $\set{t_n}$ such that $t_{n+1}\geq t_n+1$ and
 \begin{align*}
     \lim_{n\to\infty}\bar u_k(t_n)=-\infty.
 \end{align*}
 Without loss of generality, we assume the following statements holds:
\begin{enumerate}[(1)]
    \item $\lim_{n\to\infty}\int_{\Sigma}h_ie^{u_i(t_n)}\dif\mu_{\Sigma}=a_i>0$.
    \item $\lim_{n\to\infty}e^{\bar u_i(t_n)}=b_i$  where $b_i>0$ if $i\neq k$.
    \item $e^{u_i(t_n)}\dif\mu_{\Sigma}$ converges to a nonnegative Radon measure $\mu_i$. 
    \item $u(t_n)-\bar u(t_n)$ converges to $G=(G_1,\dots,G_N)$ weakly in $W^{1,p}\left(\Sigma,\mathbb{R}^N\right)$ for every $1<p<2$ and strongly in $L^q\left(\Sigma,\mathbb{R}^N\right)$ for every $1<q<\infty$. 
    \item $U_i(t_n)-\bar U_i(t_n)$ converges to $H_i$ weakly in $H^2\left(\Sigma\right)$ and strongly in $W^{1,q}\left(\Sigma\right)$ for every $1<q<\infty$ if $i\neq k$. 
\end{enumerate}
Here for convenience, we denote by
\begin{align*}
    G_i=G_i(\cdot,p_0),\quad H_i=H_i(\cdot,p_0),\quad i\in I.
\end{align*}
According to \autoref{lem:singular-set}, we know that
\begin{align*}
    \mu_i=b_ie^{G_i}\dif\mu_{\Sigma},\quad i\neq k,\\
    \mu_k=\int_{\Sigma}e^{u_{k,0}}\dif\mu_{\Sigma}\delta_{p_0}.
\end{align*}

One can check that
\begin{align*}
    \lim_{n\to\infty}\left(\int_{\Sigma}Q_iu_i(t_n)\dif\mu_{\Sigma}-\rho_i\ln\int_{\Sigma}h_ie^{u_i(t_n)}\dif\mu_{\Sigma}\right)=\int_{\Sigma}Q_iG_i\dif\mu_{\Sigma}-\rho_i\ln\int_{\Sigma}h_ie^{G_i}\dif\mu_{\Sigma},\quad i\neq k.
\end{align*}
Consider the following functional
\begin{align*}
    I_k(u)\coloneqq\dfrac12\sum_{i,j=1}^Na^{ij}\int_{\Sigma}\hin{\nabla u_i}{\nabla u_j}\dif\mu_{\Sigma}-\dfrac14\int_{\Sigma}\abs{\nabla u_k}^2\dif\mu_{\Sigma}.
\end{align*}
Notice that
\begin{align*}
    \int_{\Sigma}Q_iu_i\dif\mu_{\Sigma}-\rho_i\ln\int_{\Sigma}h_ie^{u_i}\dif\mu_{\Sigma}=\int_{\Sigma}\left(Q_i-\bar Q_i\right)\left(u_i-\bar u_i\right)\dif\mu_{\Sigma}-\rho_i\ln\int_{\Sigma}h_ie^{u_i-\bar u_i}\dif\mu_{\Sigma}.
\end{align*}
By checking the proof of \autoref{thm:partial-H^1} step by step, we conclude that
\begin{align*}
   \lim_{n\to\infty}\left(\int_{\Sigma}\sum_{i,j=1}^Na^{ij}\hin{\nabla u_i(t_n)}{\nabla u_j(t_n)}\dif\mu_{\Sigma}-\dfrac12\int_{\Sigma}\abs{\nabla u_k(t_n)}^2\dif\mu_{\Sigma}\right)=\int_{\Sigma}\left(\sum_{i,j=1}^Na^{ij}\hin{\nabla G_i}{\nabla G_j}\dif\mu_{\Sigma}-\dfrac12\abs{\nabla G_k}^2\right)\dif\mu_{\Sigma},
\end{align*}
which implies
\begin{align}\label{eq:lower-func-2}
    \lim_{n\to\infty}J_k(u(t_n))=J_k(G).
\end{align}

Choose $p_n\in\Sigma$ such that $=u_{k}(t_n,p_n)=M_k(t_n)=\max_{\Sigma}u_{k}(t_n)$. We may assume 
\begin{align*}
    \lim_{n\to\infty}M_k(t_n)=\infty,\quad\lim_{n\to\infty}p_n=p_0
\end{align*}
Now choose a conformal coordinate $\set{x}$ centered at $x_0$. Without loss of generality,  we identify $B^{\Sigma}_{2\delta}(p_0)=B_{2\delta}\subset\mathbb{R}^2$. 
We scale $u_i(t_n)$ by
\begin{align*}
    \tilde u_i(t_n)(x)=u_i(t_n,x_n+r_nx)-M_k(t_n),\quad \abs{x}<r_n^{-1}\left(\delta-\abs{x_n}\right),
\end{align*}
where $r_n=e^{-M_k(t_n)/2}$. 
We may assume $\tilde u_k(t_n)$ converges to $w_k$ in $H^2_{loc}(\mathbb{R}^2)$ and strongly in $W^{1}_{q,loc}(\mathbb{R}^2)$ for every $1<q<\infty$. Moreover
\begin{align*}
    -\Delta_{\mathbb{R}^2} w_k=\dfrac{8\pi}{\int_{\Sigma}e^{u_{k,0}}\dif\mu_{\Sigma}}e^{w_k},\quad \int_{\mathbb{R}^2}e^{w_k}\dif\mu_{\mathbb{R}^2}<\infty.
\end{align*}
By the classification result for Liouville equation of Chen-Li \cite{CheLi91classification}, we know that
\begin{align*}
    w_k(x)=-2\ln\left(1+L_k\abs{x}^2\right),\quad L_k=\dfrac{\pi}{\int_{\Sigma}e^{u_{k,0}}\dif\mu_{\Sigma}}.
\end{align*}
A simple calculation yields
\begin{align*}
    \dfrac14\int_{B^{\Sigma}_{r_nR}(p_n)}\abs{\nabla u_k}^2\dif\mu_{\Sigma}=&\dfrac14\int_{B_{R}}\abs{\nabla w_k}^2\dif\mu_{\mathbb{R}^2}+o_n(1)\\
    =&4\pi\ln\left(L_kR^2\right)-4\pi+o_R(1)+o_n(1).
\end{align*}
Here and in the following, we use $o_R(1), o_n(1), o_{\delta}(1)$ to denote those functions which converges to zero as $R\to\infty, n\to\infty, \delta\to0$ respectively.

Since 
\begin{align*}
    -\Delta H_k=4\pi\delta_{p_0}-Q_{k},\quad G_k=2H_{k}-H_{k-1}-H_{k+1},
\end{align*}
we get
\begin{align*}
    &\dfrac14\int_{\Sigma\setminus B^{\Sigma}_{\delta}(p_0)}\abs{\nabla u_k}^2\dif\mu_{\Sigma}\\
    =&\dfrac14\int_{\Sigma\setminus B^{\Sigma}_{\delta}(p_0)}\abs{\nabla G_k}^2\dif\mu_{\Sigma}+o_n(1)\\
    =&\int_{\Sigma\setminus B^{\Sigma}_{\delta}(p_0)}\abs{\nabla H_k}^2\dif\mu_{\Sigma}-\int_{\Sigma\setminus B^{\Sigma}_{\delta}(p_0)}\hin{\nabla H_k}{\nabla\left(H_{k-1}+H_{k+1}\right)}\dif\mu_{\Sigma}+\dfrac14\int_{\Sigma\setminus B^{\Sigma}_{\delta}(p_0)}\abs{\nabla\left(H_{k-1}+H_{k+1}\right)}^2\dif\mu_{\Sigma}+o_n(1)\\
    =&-\int_{\Sigma\setminus B^{\Sigma}_{\delta}(p_0)}H_k\Delta H_k\dif\mu_{\Sigma}-\int_{\partial B^{\Sigma}_{\delta}(p_0)}H_k\dfrac{\partial H_k}{\partial\nu}\dif\mu_{\partial B^{\Sigma}_{\delta}(p_0)}\\
    &-\int_{\Sigma}\hin{\nabla H_k}{\nabla\left(H_{k-1}+H_{k+1}\right)}\dif\mu_{\Sigma}+\dfrac14\int_{\Sigma}\abs{\nabla\left(H_{k-1}+H_{k+1}\right)}^2\dif\mu_{\Sigma}+o_n(1)+o_{\delta}(1)\\
    =&-\int_{\Sigma\setminus B^{\Sigma}_{\delta}(p_0)}Q_k H_k\dif\mu_{\Sigma}-\int_{\partial B^{\Sigma}_{\delta}(p_0)}H_k\dfrac{\partial H_k}{\partial\nu}\dif\mu_{\partial B^{\Sigma}_{\delta}(p_0)}\\
    &+\int_{\Sigma}H_k\Delta\left(H_{k-1}+H_{k+1}\right)\dif\mu_{\Sigma}+\dfrac14\int_{\Sigma}\abs{\nabla\left(H_{k-1}+H_{k+1}\right)}^2\dif\mu_{\Sigma}+o_n(1)+o_{\delta}(1)\\
    =&-\int_{\Sigma}Q_k H_k\dif\mu_{\Sigma}-\int_{\partial B^{\Sigma}_{\delta}(p_0)}H_k\dfrac{\partial H_k}{\partial\nu}\dif\mu_{\partial B^{\Sigma}_{\delta}(p_0)}\\
    &+\int_{\Sigma}H_k\Delta\left(H_{k-1}+H_{k+1}\right)\dif\mu_{\Sigma}+\dfrac14\int_{\Sigma}\abs{\nabla\left(H_{k-1}+H_{k+1}\right)}^2\dif\mu_{\Sigma}+o_n(1)+o_{\delta}(1).
\end{align*}
Recall that we have the expression of $H_k$
 \begin{align*}
    H_k(x)=-2\ln\abs{x}+R_k+\hin{\beta_k}{x}+O\left(\abs{x}^2\right),\quad x\to o.
\end{align*}
We have
\begin{align*}
    x\cdot\nabla H_k(x)=&-2+\hin{\beta_k}{x}+O\left(\abs{x}^2\right),\quad x\to o,
\end{align*}
which implies that along the boundary $\partial B^{\Sigma}_{\delta}(p_0)$
\begin{align*}
    H_k\dfrac{\partial H_k}{\partial\nu}=&\dfrac{1}{\abs{x}}\left(4\ln\abs{x}-2R_k\right)+\dfrac{1}{\abs{x}}\left(R_k-2-2\ln\abs{x}\right)\hin{\beta_k}{x}+O\left(\abs{x}\ln\abs{x}\right),\quad x\to o.
\end{align*}
Thus
\begin{align*}
    -\int_{\partial B^{\Sigma}_{\delta}(p_0)}H_k\dfrac{\partial H_k}{\partial\nu}\dif\mu_{\partial B^{\Sigma}_{\delta}(p_0)}=&-8\pi\ln\delta+4\pi R_k(p_0)+o_{\delta}(1).
\end{align*}
We conclude 
\begin{align}\label{eq:lower-func-3}
\begin{split}
     \dfrac14\int_{\Sigma\setminus B^{\Sigma}_{\delta}(p_0)}\abs{\nabla u_k}^2\dif\mu_{\Sigma}=&-\int_{\Sigma}Q_k H_k\dif\mu_{\Sigma}-8\pi\ln\delta+4\pi R_k(p_0)\\
    &+\int_{\Sigma}H_k\Delta\left(H_{k-1}+H_{k+1}\right)\dif\mu_{\Sigma}+\dfrac14\int_{\Sigma}\abs{\nabla\left(H_{k-1}+H_{k+1}\right)}^2\dif\mu_{\Sigma}+o_n(1)+o_{\delta}(1).
\end{split}
\end{align}

Define 
\begin{align*}
    u^*_k(t_n,r)=\dfrac{1}{2\pi}\int_0^{2\pi}u_k\left(t_n,x_n+re^{i\theta}\right)\dif\theta.
\end{align*}
We have for every $r_nR\leq r<s\leq\delta_n\coloneqq\delta-\abs{x_n}$ (see for example \cite[equation (3.4)]{LiWan10weak}),
\begin{align*}
    \int_{B_{s}\setminus B_r}\abs{\nabla u^*_k(t_n)}\dif\mu_{\mathbb{R}^2}\leq\int_{B_{s}\setminus B_r}\abs{\nabla u_k(t_n)}\dif\mu_{\mathbb{R}^2}.
\end{align*}
A direct calculation yields 
\begin{align*}
   u^*_k(t_n,r_nR)+2\ln r_n=&-2\ln\left(L_kR^2\right)+o_{R}(1)+o_n(1),\\
   u^*_k(t_n,\delta_n)-\bar u_k(t_n)=&-4\ln\delta_n+2R_k-H_{k-1}(p_0)-H_{k+1}(p_0)+o_{\delta}(1)+o_n(1),
\end{align*}
By Dirichlet's principle, one can verify directly
\begin{align*}
    \dfrac14\int_{B^{\Sigma}_{\delta_n}(p_n)\setminus B^{\Sigma}_{r_nR}(p_n)}\abs{\nabla u_k(t_n)}^2\dif\mu_{\Sigma}\geq&\dfrac14\int_{B_{\delta_n}\setminus B_{r_nR}}\abs{\nabla u^*_k(t_n)}^2\dif\mu_{\mu_{\mathbb{R}^2}}\\
    \geq&\dfrac{\pi\left(u^*_k(t_n,\delta_n)-u^*_k(t_n,r_nR)\right)^2}{2\left(\ln\delta_n-\ln(r_nR)\right)}.
\end{align*}
Since
\begin{align*}
  \tau_n\coloneqq& u^*_k(t_n,\delta_n)-u^*_k(t_n,r_nR)-2\ln r_n-\bar u_k(t_n)\\
  =&4\ln\dfrac{R}{\delta_n}+2R_k-H_{k-1}(p_0)-H_{k+1}(p_0)+2\ln L_k+o_{\delta}(1)+o_n(1)+o_{R}(1),
\end{align*}
we have
\begin{align}\label{eq:lower-func-4}
\begin{split}
    \dfrac14\int_{B^{\Sigma}_{\delta_n}(p_n)\setminus B^{\Sigma}_{r_nR}(p_n)}\abs{\nabla u_k(t_n)}^2\dif\mu_{\Sigma}\geq&\dfrac{\pi(\tau_n+\bar u_k(t_n)+2\ln r_n)^2}{-2\ln r_n}\left(1-\dfrac{\ln(R/\delta_n)}{-\ln r_n}\right)^{-1}\\
    \geq&\dfrac{\pi}{2}\left(\dfrac{\tau_n}{\ln r_n}+2+\dfrac{\bar u_k(t_n)}{\ln r_n}\right)^2\ln\dfrac{R}{\delta_n}+\dfrac{\pi\left(\tau_n+\bar u_k(t_n)-2\ln r_n\right)^2}{-2\ln r_n}-4\pi\left(\tau_n+\bar u_k(t_n)\right)\\
    =&\dfrac{\pi}{2}\left(\left(\dfrac{\tau_n}{\ln r_n}+2+\dfrac{\bar u_k(t_n)}{\ln r_n}\right)^2-32\right)\ln\dfrac{R}{\delta_n}+\dfrac{\pi\left(\tau_n+\bar u_k(t_n)-2\ln r_n\right)^2}{-2\ln r_n}-4\pi\bar u_k(t_n)\\
    &-8\pi R_k(p_0)+4\pi\left(H_{k-1}(p_0)+H_{k+1}(p_0)\right)-8\pi\ln L_k+o_{\delta}(1)+o_n(1)+o_{R}(1).
    \end{split}
\end{align}

Combining the estimates \eqref{eq:lower-func-1}, \eqref{eq:lower-func-2}, \eqref{eq:lower-func-3} and \eqref{eq:lower-func-4}, we obtain
\begin{align*}
    J(u(t_n))=&J_k(G)+\dfrac14\int_{\Sigma}\abs{\nabla u_k}^2\dif\mu_{\Sigma}+\int_{\Sigma}Q_kG_k\dif\mu_{\Sigma}+4\pi \bar u_k(t_n)-4\pi\ln\left(h_{k}(p_0)\dfrac{\pi}{L_k}\right)+o_n(1)\\
    =&J_k(G)+\int_{\Sigma}Q_kG_k\dif\mu_{\Sigma}+4\pi \bar u_k(t_n)-4\pi\ln\left(h_{k}(p_0)\dfrac{\pi}{L_k}\right)+o_n(1)\\
    &+\dfrac14\int_{\Sigma\setminus B^{\Sigma}_{\delta}(p_0)}\abs{\nabla u_k}^2\dif\mu_{\Sigma}+\dfrac14\int_{B^{\Sigma}_{\delta}(p_0)\setminus B^{\Sigma}_{r_nR}(p_n)}\abs{\nabla u_k(t_n)}^2\dif\mu_{\Sigma}+\dfrac14\int_{B^{\Sigma}_{r_nR}(p_n)}\abs{\nabla u_k}^2\dif\mu_{\Sigma}\\
    \geq&J_k(G)+\int_{\Sigma}Q_kG_k\dif\mu_{\Sigma}+4\pi \bar u_k(t_n)-4\pi\ln\left(h_{k}(p_0)\dfrac{\pi}{L_k}\right)+o_n(1)\\
    &+\dfrac14\int_{\Sigma\setminus B^{\Sigma}_{\delta}(p_0)}\abs{\nabla u_k}^2\dif\mu_{\Sigma}+\dfrac14\int_{B^{\Sigma}_{\delta_n}(p_n)\setminus B^{\Sigma}_{r_nR}(p_n)}\abs{\nabla u_k(t_n)}^2\dif\mu_{\Sigma}+\dfrac14\int_{B^{\Sigma}_{r_nR}(p_n)}\abs{\nabla u_k}^2\dif\mu_{\Sigma}\\
    \geq&J_k(G)+\int_{\Sigma}Q_kG_k\dif\mu_{\Sigma}+4\pi \bar u_k(t_n)-4\pi\ln\left(h_{k}(p_0)\dfrac{\pi}{L_k}\right)+o_n(1)\\
    &-\int_{\Sigma}Q_k H_k\dif\mu_{\Sigma}-8\pi\ln\delta+4\pi R_k(p_0)\\
    &+\int_{\Sigma}H_k\Delta\left(H_{k-1}+H_{k+1}\right)\dif\mu_{\Sigma}+\dfrac14\int_{\Sigma}\abs{\nabla\left(H_{k-1}+H_{k+1}\right)}^2\dif\mu_{\Sigma}+o_n(1)+o_{\delta}(1)\\
    &+\dfrac{\pi}{2}\left(\left(\dfrac{\tau_n}{\ln r_n}+2+\dfrac{\bar u_k(t_n)}{\ln r_n}\right)^2-32\right)\ln\dfrac{R}{\delta_n}+\dfrac{\pi\left(\tau_n+\bar u_k(t_n)-2\ln r_n\right)^2}{-2\ln r_n}-4\pi\bar u_k(t_n)\\
    &-8\pi R_k(p_0)+4\pi\left(H_{k-1}(p_0)+H_{k+1}(p_0)\right)-8\pi\ln L_k+o_{\delta}(1)+o_n(1)+o_{R}(1)\\
    &+4\pi\ln\left(L_kR^2\right)-4\pi+o_R(1)+o_n(1)\\
    =&J_k(G)+\int_{\Sigma}Q_kH_k\dif\mu_{\Sigma}+\dfrac14\int_{\Sigma}\abs{\nabla\left(H_{k-1}+H_{k+1}\right)}^2\dif\mu_{\Sigma}\\
    &+4\pi\left(H_{k-1}(p_0)+H_{k+1}(p_0)\right)-\int_{\Sigma}\left(H_{k-1}+H_{k+1}\right)Q_k\dif\mu_{\Sigma}+\int_{\Sigma}H_k\Delta\left(H_{k-1}+H_{k+1}\right)\dif\mu_{\Sigma}\\
    &+\dfrac{\pi}{2}\left(\left(\dfrac{\tau_n}{\ln r_n}+2+\dfrac{\bar u_k(t_n)}{\ln r_n}\right)^2-16\right)\ln\dfrac{R}{\delta_n}+\dfrac{\pi\left(\tau_n+\bar u_k(t_n)-2\ln r_n\right)^2}{-2\ln r_n}\\
    &-4\pi\left(R_k+\ln h_k(p_0)\right)-4\pi-4\pi\ln\pi +o_{\delta}(1)+o_n(1)+o_{R}(1).
\end{align*}

Since $J(u(t_n))$ is bounded from above by the monotonicity formula \eqref{eq:monotonicity}, we conclude that
\begin{align*}
   \bar u_k(t_n)-2\ln r_n=O\left(\sqrt{-\ln r_n}\right),\quad\text{as}\ n\to\infty,
\end{align*}
which implies
\begin{align*}
    \lim_{n\to\infty}\dfrac{\bar u_k(t_n)}{\ln r_n}=2.
\end{align*}
Thus
\begin{align*}
    J(u(t_n))\geq&J_k(G)+\int_{\Sigma}Q_k\left(H_k-H_{k-1}-H_{k+1}\right)\dif\mu_{\Sigma}+\dfrac14\int_{\Sigma}\abs{\nabla\left(H_{k-1}+H_{k+1}\right)}^2\dif\mu_{\Sigma}+\int_{\Sigma}H_k\Delta\left(H_{k-1}+H_{k+1}\right)\dif\mu_{\Sigma}\\
    &-4\pi\left(R_k-H_{k-1}(p_0)-H_{k+1}(p_0)+\ln h_k(p_0)\right) -4\pi-4\pi\ln\pi+o_{\delta}(1)+o_n(1)+o_{R}(1),
\end{align*}
which implies
\begin{align*}
    \liminf_{n\to\infty}J(u(t_n))\geq&J_k(G)+\int_{\Sigma}Q_k\left(G_k-H_k\right)\dif\mu_{\Sigma}+\dfrac14\int_{\Sigma}\abs{\nabla\left(H_{k-1}+H_{k+1}\right)}^2\dif\mu_{\Sigma}+\int_{\Sigma}H_k\Delta\left(H_{k-1}+H_{k+1}\right)\dif\mu_{\Sigma}\\
    &-4\pi\left(R_k-H_{k-1}(p_0)-H_{k+1}(p_0)+\ln h_k(p_0)\right)-4\pi-4\pi\ln\pi.
\end{align*}

Notice that
\begin{align*}
    -\Delta H_k=4\pi\delta_{p_0}-Q_k,\quad\bar H_k=0.
\end{align*}
We obtain for every smooth function $F$
\begin{align*}
    4\pi F(p_0)=\int_{\Sigma}FQ_k\dif\mu_{\Sigma}-\int_{\Sigma}H_k\Delta F\dif\mu_{\Sigma}
\end{align*}
which implies
\begin{align*}
  \int_{\Sigma} H_k\Delta\left(H_{k-1}+H_{k+1}\right)\dif\mu_{\Sigma}=&\int_{\Sigma} \left(H_{k-1}+H_{k+1}\right)Q_k\dif\mu_{\Sigma}-4\pi\left(H_{k-1}(p_0)+H_{k+1}(p_0)\right).
\end{align*}
Hence
\begin{align*}
J(u(t))\geq&J_k(G)+\dfrac14\int_{\Sigma}\abs{\nabla\left(H_{k-1}+H_{k+1}\right)}^2\dif\mu_{\Sigma}+\int_{\Sigma}Q_k H_k\dif\mu_{\Sigma}\\
    &-4\pi\left(R_k+\ln h_k(p_0)\right)-4\pi-4\pi\ln\pi.
\end{align*}
We complete the proof.
\end{proof}

\begin{rem}
  Consider the following transformation
    \begin{align*}
        g\mapsto\tilde g=e^{2\psi}g,\quad h_i\mapsto\tilde h_i=h_ie^{-\psi-\phi_i},\quad Q_i\mapsto\rho_i,\quad u_i\to\mapsto \tilde u_i=u_i+\phi_i,\quad i\in I,
    \end{align*}
    where $\phi_i$ solves the following elliptic system
    \begin{align*}
    -\sum_{j=1}^Na^{ij}\Delta_{g}\phi_j=Q_i-\rho_i\dfrac{e^{2\psi}}{\int_{\Sigma}e^{2\psi}\dif\mu_{g}},\quad\int_{\Sigma}\phi_ie^{2\psi}\dif\mu_{g}=0,\quad i\in I,
\end{align*}
Then
\begin{align*}
   -\sum_{j=1}^N\Delta_{\tilde g}\tilde u_i=\rho_i\left(\dfrac{\tilde h_ie^{\tilde u_i}}{\int_{\Sigma}\tilde h_ie^{\tilde u_i}\dif\mu_{\tilde g}}-\dfrac{1}{\abs{\Sigma}_{\tilde g}}\right)-e^{-2\psi-\phi_i}\dfrac{\partial e^{\tilde u_i}}{\partial t},\quad i\in I.
\end{align*}
The corresponding functions
\begin{align*}
    \tilde G_i=G_i-\dfrac{\int_{\Sigma}G_ie^{2\psi}\dif\mu_{g}}{\int_{\Sigma}e^{2\psi}\dif\mu_{g}}+\phi_i,\quad\tilde H_i=H_i-\dfrac{\int_{\Sigma}H_ie^{2\psi}\dif\mu_{g}}{\int_{\Sigma}e^{2\psi}\dif\mu_{g}}+\sum_{j=1}^Na^{ij}\phi_j,\quad i\in I.
\end{align*}
The distance function satisfies
\begin{align*}
    \ln\tilde r(x)=\ln\tilde r(x)+\psi(x)+O\left(\abs{x}^2\right),\quad x\to o
    .
\end{align*}

If we write
\begin{align*}
J(u,g,h,Q)\coloneqq&\dfrac12\sum_{i,j=1}^Na^{ij}\int_{\Sigma}\hin{\nabla_g u_i}{\nabla_g u_j}_{g}\dif\mu_{g}+\sum_{i=1}^N\left(\int_{\Sigma}Q_iu_i\dif\mu_{g}-\rho_i\ln\int_{\Sigma}h_ie^{u_i}\dif\mu_{g}\right),
\end{align*}
then a direct computation yields
\begin{align*}
J(u,g,h,Q)=J(\tilde u,\tilde g,\tilde h,\rho)-\dfrac12\sum_{i,j=1}^Na^{ij}\int_{\Sigma}\hin{\nabla_{g}\phi_i}{\nabla_{g}\phi_j}_{g}\dif\mu_{g}.
\end{align*}

Denote
\begin{align*}
    F(p_0,g,h,Q)\coloneqq&I_{k,g}(G)+\sum_{i\neq k}\left(\int_{\Sigma}Q_iG_i\dif\mu_{g}-\rho_i\int_{\Sigma}h_ie^{G_i}\dif\mu_{g}\right)\\
    &+\dfrac14\int_{\Sigma}\abs{\nabla_g\left(H_{k-1}+H_{k+1}\right)}_g^2\dif\mu_{g}+\int_{\Sigma}Q_kH_k\dif\mu_{g}\\
    &-4\pi\left(\lim_{p\to p_0}\left(H_k(p,p_0)+2\mathrm{dist}(p,p_0)\right)+\ln h_k(p_0)\right)-4\pi-4\pi\ln\pi.
\end{align*}
Then a straightforward verification yields
\begin{align*}
   F(p_0,g,h,Q)= F(p_0,\tilde g,\tilde h,\rho)-\dfrac12\sum_{i,j=1}^Na^{ij}\int_{\Sigma}\hin{\nabla_{g}\phi_i}{\nabla_{g}\phi_j}_{g}\dif\mu_{g}.
\end{align*}

 Therefore, to obtain the desired lower bound of the functional $J$ along the flow, without loss of generality, we may assume the area of $\Sigma$ is $1$ and identify a neighborhood of $p_0$ as the Euclidean domain. Moreover, we may assume
    \begin{align*}
        Q_i=\rho_i,\quad \int_{\Sigma}e^{u_{i,0}}\dif\mu_{\Sigma}=1,\quad i\in I.
    \end{align*}

\end{rem}

\section{Test functions}

Consider the following function
\begin{align*}
    p_0\mapsto F(p_0)\coloneqq& J_k\left(G(\cdot,p_0)\right)+\dfrac14\int_{\Sigma}\abs{\nabla\left(H_{k-1}(\cdot,p_0)+H_{k+1}(\cdot,p_0)\right)}^2\dif\mu_{\Sigma}+\int_{\Sigma}Q_kH_k(\cdot,p_0)\dif\mu_{\Sigma}\\
&-4\pi\left(\lim_{p\to p_0}\left(H_k(\cdot,p_0)+2\ln\mathrm{dist}(p,p_0)\right)+\ln h_k(p_0)\right)-4\pi-4\pi\ln\pi.
\end{align*}
For convenience, denote
\begin{align*}
    G_i=G_i(\cdot,p_0),\quad H_i=H_i(\cdot,p_0),\quad i\in I.
\end{align*}

Let $\Gamma(\cdot,p)$ be the Green function satisfies
\begin{align*}
    -\Delta\Gamma(\cdot,p)=4\pi\left(\delta_p-\dfrac{1}{\abs{\Sigma}}\right),\quad \bar\Gamma(\cdot,p)=0.
\end{align*}
Denote by $B(p)$ the regularity part of $G(\cdot,p)$, i.e.,
\begin{align*}
    B(p)=\lim_{q\to p}\left(\Gamma(q,p)+2\ln\mathrm{dist}(q,p)\right).
\end{align*}
Choose a normal coordinates $\set{x}$ which is centered at $p_0$. We have the following expression
\begin{align*}
    \Gamma(x,p)=-2\ln\abs{x}+a(p)+\hin{\alpha(p)}{x}+O\left(\abs{x}^2\right),\quad as\ x\to o.
\end{align*}
Since the Green function is symmetric, we conclude that
\begin{align*}
   B(p)=a(p),\quad \nabla B(p)=2\alpha(p).
\end{align*}

Notice that
\begin{align*}
    -\Delta(\Gamma(\cdot,p_0)-H_k(\cdot,p_0))=Q_k-\dfrac{4\pi}{\abs{\Sigma}},\quad\overline{\Gamma(\cdot,p_0)-H_k(\cdot,p_0)}=0.
\end{align*}
We know that $\Theta_k\coloneqq\Gamma(\cdot,p_0)-H_k(\cdot,p_0)$ is a smooth function which is independent of $p_0$. Recall
 \begin{align*}
     H_k(x,p)=-2\ln\abs{x}+R_k(p)+\hin{\beta_k(p)}{x}+O\left(\abs{x}^2\right).
 \end{align*}
 We have
\begin{align*}
     R_k=B-\Theta,\quad \beta_k=\alpha-\nabla\Theta.
\end{align*}
Thus
\begin{align*}
    \nabla R_k=\nabla B-\nabla\Theta=2\alpha-\nabla\Theta=2\beta_k+\nabla\Theta.
\end{align*}

Choose a smooth curve $\gamma:(-\epsilon,\epsilon)\To\Sigma$ such that 
\begin{align*}
    \gamma(0)=p_0,\quad\gamma'(0)=X.
\end{align*}
We consider the following function
\begin{align*}
    t\mapsto f(t)\coloneqq F(\gamma(t)).
\end{align*}
We will compute $f'(0)$. 

Use the same notation as in the beginning of the above section, a direct computation yields
\begin{align*}
    F=&J(\underline{G},g,\underline{h},\underline{Q})+\int_{\Sigma}\left(Q_k-Q_{k-1}-Q_{k+1}\right)H_k\dif\mu_{\Sigma}-4\pi\left(R_k+\ln h_k(p_0)\right)-4\pi-4\pi\ln\pi.
\end{align*}
Thus
\begin{align*}
    f'(0)=&\rho_{k-1}\dfrac{\int_{\Sigma}h_{k-1}e^{G_{k-1}}H_k'\dif\mu_{\Sigma}}{\int_{\Sigma}h_{k-1}e^{G_{k-1}}\dif\mu_{\Sigma}}+\rho_{k+1}\dfrac{\int_{\Sigma}h_{k+1}e^{G_{k+1}}H_k'\dif\mu_{\Sigma}}{\int_{\Sigma}h_{k+1}e^{G_{k+1}}\dif\mu_{\Sigma}}
    \\&+\int_{\Sigma}\left(Q_k-Q_{k-1}-Q_{k+1}\right)H'_k\dif\mu_{\Sigma}-4\pi\left(\nabla_XR_k(p_0)+\nabla_X\ln h_k(p_0)\right).
\end{align*}
Here
\begin{align*}
    H_k'=\dfrac{\partial}{\partial t}H_k
\end{align*}
satisfies for every smooth function $\zeta$
\begin{align*}
    4\pi\nabla_X\zeta(p_0)=-\int_{\Sigma}H_k'\Delta\zeta\dif\mu_{\Sigma}.
\end{align*}
Hence
\begin{align*}
    f'(0)=&-\int_{\Sigma}\Delta\left(H_{k-1}+H_{k+1}\right)H_k'\dif\mu_{\Sigma}+\int_{\Sigma}Q_kH'_k\dif\mu_{\Sigma}-4\pi\left(\nabla_XR_k(p_0)+\nabla_X\ln h_k(p_0)\right).
\end{align*}

Therefore,
\begin{align*}
    f'(0)=&-\int_{\Sigma}\Delta\left(H_{k-1}+H_{k+1}\right)H_k'\dif\mu_{\Sigma}+\int_{\Sigma}\Delta\Theta_kH'_k\dif\mu_{\Sigma}-4\pi\left(\nabla_XR_k(p_0)+\nabla_X\ln h_k(p_0)\right)\\
    =&4\pi\nabla_X\left(H_{k-1}+H_{k+1}\right)(p_0)-4\pi\nabla_X\Theta_k(p_0)-4\pi\left(\nabla_XR_k(p_0)+\nabla_X\ln h_k(p_0)\right)\\
    =&-4\pi\left(2\hin{\alpha(p_0)}{X}-\nabla_X\left(H_{k-1}+H_{k+1}\right)(p_0)+\nabla_X\ln h_k(p_0)\right).
\end{align*}
In other words,
\begin{align}\label{eq:gradient F}
    \nabla F=-4\pi\nabla\left(B-H_{k-1}-H_{k+1}+\ln h_k\right)
\end{align}

If we set
\begin{align*}
    \tilde g=e^{2\psi}g,\quad\tilde h_i=h_ie^{-2\psi-\phi_i},\quad\tilde u_i=u_i+\phi_i
\end{align*}
where
\begin{align*}
    -\sum_{j=1}^Na^{ij}\Delta_{g}\phi_j=Q_i-\rho_i\dfrac{e^{2\psi}}{\int_{\Sigma}e^{2\psi}\dif\mu_{g}},\quad\int_{\Sigma}\phi_ie^{2\psi}\dif\mu_{g}=0,\quad i\in I,
\end{align*}
Then a direct computation yields
\begin{align*}
    \Delta_{\tilde g}\ln\tilde h_i+\sum_{j=1}^Na_{ij}\dfrac{\rho_j}{\abs{\Sigma}_{\tilde g}}-2\kappa_{\tilde g}=e^{-2\psi}\left(\Delta_{g}\ln h_k+\sum_{j=1}^Na_{ij}Q_j-2\kappa_g\right).
\end{align*}

\begin{theorem}\label{thm:test}
    Fixed a point $p_0\in\Sigma$ such that $h_k(p_0)>0$. 
 There exists a sequence of smooth functions $u^{(\epsilon)}$ as $\epsilon\searrow 0$ such that
    \begin{align*}
        J(u^{(\epsilon)})=&F(p_0)-\pi\left(\dfrac{1}{16\pi^2}\abs{\nabla F(p_0)}^2+\Delta \ln h_k(p_0)+\sum_{j=1}^Na_{ij}Q_j(p_0)-2\kappa(p_0) \right)\epsilon^2\ln\epsilon^{-2}+O\left(\epsilon^2\ln(-\ln\epsilon)\right).
    \end{align*}
    Here $\kappa$ is the Gaussian curvature of $\Sigma$. 
\end{theorem}

\begin{proof}
Without loss of generality, we assume the area of $\Sigma$ is $1$ and identify  $B^{\Sigma}_{\delta}(p_0)=B_{\delta}\subset\mathbb{R}$ as the Euclidean ball in $\mathbb{R}^2$ which is centered at the origin point $o\in\mathbb{R}^2$.
We also assume
\begin{align*}
    Q_i=\rho_i,\quad\int_{\Sigma}h_ie^{G_i}\dif\mu_{\Sigma}=1,\quad i\in I.
\end{align*}
Here the functions $G_i$ satisfy the following system
\begin{align*}
\begin{cases}
     -\Delta H_i=\rho_i(h_ie^{G_i}-1),\quad\bar H_i=0,\quad i\neq k,\\
    -\Delta H_k=4\pi\left(\delta_{p_0}-1\right),\quad\bar H_k=0,
\end{cases}
\end{align*}
where
\begin{align*}
    H_i=\sum_{j=1}^Na^{ij}G_j,\quad i\in I.
\end{align*}
We have proved that $H_i\in C^{\infty}\left(\Sigma\right)$ for $i\neq k$ and $H_k+2\ln\mathrm{dist}(\cdot,p_0)\in C^{\infty}\left(\Sigma\right)$. In a neighborhood of $p_0$, we have
\begin{align*}
    H_k(x)=-2\ln\abs{x}+\pi\abs{x}^2+P(x)
\end{align*}
where $P$ is a harmonic function defined in $B_{\delta}$. Thus
\begin{align*}
    H_k(x)=-2\ln\abs{x}+a+\hin{\alpha}{x}+S(x),\quad S(x)=O\left(\abs{x}^2\right),\quad as\ x\to o.
\end{align*}
The subsequent result will be frequently utilized: given any smooth function $F$ and for sufficiently small $\rho > 0$, it holds that
\begin{align*}
    \int_{\partial B_{\rho}} F \dif \sigma = 2\pi F(p_0)\rho + \frac{\pi}{2}\Delta F(p_0)\rho^3 + O\left(\rho^5\right),\quad as\ \rho\to 0.
\end{align*}

Notice that
\begin{align*}
    J(u)=&\dfrac12\sum_{i,j=1}^Na^{ij}\int_{\Sigma}\hin{\nabla u_i}{\nabla u_j}\dif\mu_{\Sigma}+\sum_{i=1}^N\rho_i\left(\bar u_i-\ln\int_{\Sigma}h_ie^{u_i}\dif\mu_{\Sigma}\right)\\
    =&I_k(u)+\dfrac14\int_{\Sigma}\abs{\nabla u_k}^2\dif\mu_{\Sigma}+\sum_{i=1}^N\rho_i\left(\bar u_i-\ln\int_{\Sigma}h_ie^{u_i}\dif\mu_{\Sigma}\right)\\
    =&I_k(u)+\int_{\Sigma}\abs{\nabla U_k}^2\dif\mu_{\Sigma}+\dfrac14\int_{\Sigma}\abs{\nabla\left(U_{k-1}+U_{k+1}\right)}^2\dif\mu_{\Sigma}-\int_{\Sigma}\hin{\nabla U_k}{\nabla\left(U_{k-1}+U_{k+1}\right)}\dif\mu_{\Sigma}\\
    &+\sum_{i=1}^N\rho_i\left(\bar u_i-\ln\int_{\Sigma}h_ie^{u_i}\dif\mu_{\Sigma}\right).
\end{align*}
Here
\begin{align*}
    U_i=\sum_{j=1}^Na^{ij}u_j,\quad i\in I.
\end{align*}

For every positive number $\epsilon$, we choose $U_i^{(\epsilon)}=H_i$ if $i\neq k$ and
\begin{align*}
    U_k^{(\epsilon)}(x)=\begin{cases}
       -\ln\left(\abs{x}^2+\epsilon^2\right)+\hin{\alpha}{x},&x\in B_{L_\epsilon\epsilon},\\
        H_k(x)-\eta_{\epsilon}(x) S(x)+\beta_{\epsilon}, &x\in B_{2L_{\epsilon}\epsilon}\setminus B_{L_{\epsilon}\epsilon},\\
        H_k(x)+\beta_{\epsilon},&else.
    \end{cases}
\end{align*}
Here $\eta_{\epsilon}(x)=\eta\left(\frac{\abs{x}}{L_\epsilon\epsilon}\right)$ and $\eta\in C_0^{\infty}([0,2])$ satisfies 
\begin{align*}
  0\leq\eta \leq 1,\quad \eta\vert_{[0,1]}=1,\quad \eta'\leq0.
\end{align*}
Thus, $\eta_\epsilon\in C_0^{\infty}(B_{2L_\epsilon\epsilon})$ and satisfies
\begin{align*}
    0\leq\eta_\epsilon\leq 1,\quad \eta_{\epsilon}\vert_{B_{L_\epsilon\epsilon}}=1, \quad\abs{\nabla\eta_\epsilon}\leq\dfrac{C}{L_\epsilon\epsilon},\quad\abs{\nabla^2\eta_\epsilon}\leq\dfrac{C}{L_\epsilon
    ^2\epsilon^2}.
\end{align*}
The constant $\beta_{\epsilon}$ is given by
\begin{align*}
    \beta_{\epsilon}=-a-\ln\dfrac{L_{\epsilon}^2+1}{L_{\epsilon}^2}.
\end{align*}
That is, $U_k^{(\epsilon)}$ is smooth. 
The positive constant $L_{\epsilon}$ will be specified subsequently and must satisfy the conditions $L_{\epsilon}\to\infty$ and $L_{\epsilon}\epsilon\to 0$ as $\epsilon\to 0$. 
Finally, we define
\begin{align*}
    u_i^{(\epsilon)}=\sum_{j=1}^Na_{ij}U_j^{(\epsilon)},\quad i\in I.
\end{align*}
We have
\begin{align}\label{eq:test-0}
   \begin{split}
J(u^{(\epsilon)})=&J_k(G)+\dfrac14\int_{\Sigma}\abs{\nabla\left(H_{k-1}+H_{k+1}\right)}^2\dif\mu_{\Sigma}\\
    &-\rho_{k-1}\ln\int_{\Sigma}h_{k-1}e^{2H_{k-1}-H_{k-2}}e^{-U^{(\epsilon)}_k}\dif\mu_{\Sigma}-\rho_{k+1}\ln\int_{\Sigma}h_{k+1}e^{2H_{k+1}-H_{k+2}}e^{-U^{(\epsilon)}_k}\dif\mu_{\Sigma}\\
   &+\int_{\Sigma}\abs{\nabla U^{(\epsilon)}_k}^2\dif\mu_{\Sigma}+\int_{\Sigma}\hin{U^{(\epsilon)}_k}{\Delta\left(H_{k-1}+H_{k+1}\right)}\dif\mu_{\Sigma}+\left(8\pi-\rho_{k-1}-\rho_{k+1}\right)\bar U^{(\epsilon)}_k\\
    &-4\pi\ln\int_{\Sigma}h_{k}e^{-H_{k-1}-H_{k+1}}e^{2U^{(\epsilon)}_{k}}\dif\mu_{\Sigma}.
    \end{split}
\end{align}

One can check that
\begin{align*}
    \bar U_{k}^{(\epsilon)}=&\beta_\epsilon-\int_{B_{2L_\epsilon\epsilon}\setminus B_{L_\epsilon\epsilon}}\eta_\epsilon S\dif x-\int_{B_{L_\epsilon\epsilon}}\left(H_k+\ln\left(\abs{x}^2+\epsilon^2\right)-\hin{\alpha}{x}+\beta_\epsilon\right)\dif x\\
    =&\beta_\epsilon-\int_{B_{L_\epsilon\epsilon}}\left(\ln\dfrac{\abs{x}^2+\epsilon^2}{\abs{x}^2}-\ln\dfrac{L_\epsilon^2+1}{L_\epsilon^2}\right)\dif x+O\left(L_\epsilon^4\epsilon^4\right).
\end{align*}
A simple calculation yields
\begin{align*}
    \int_{B_{L_\epsilon\epsilon}}\left(\ln\dfrac{\abs{x}^2+\epsilon^2}{\abs{x}^2}-\ln\dfrac{L_\epsilon^2+1}{L_\epsilon^2}\right)\dif x=&\pi\epsilon^2\ln\left(L_\epsilon^2+1\right).
\end{align*}
We get
\begin{align}\label{eq:test-1}
    \bar U_{k}^{(\epsilon)}=-a-\ln\dfrac{L_{\epsilon}^2+1}{L_{\epsilon}^2}-\pi\epsilon^2\ln\left(L_\epsilon^2+1\right)+O\left(L_\epsilon^4\epsilon^4\right).
\end{align}
Similarly, we have for every smooth function $F$,
\begin{align*}
    \int_{\Sigma} U^{(\epsilon)}_k\Delta F\dif\mu_{\Sigma}=&\int_{\Sigma}H_k\Delta F\dif\mu_{\Sigma}-\int_{ B_{2L_{\epsilon}\epsilon}\setminus B_{L_{\epsilon}\epsilon}}\eta_{\epsilon}S \Delta F\dif x-\int_{ B_{L_{\epsilon}\epsilon}}\left(H_k+\ln\left(\abs{x}^2+\epsilon^2\right)-\hin{\alpha}{x}+\beta_\epsilon\right)\Delta F\dif x\\
    =&4\pi\bar F-4\pi F(p_0)-\int_{ B_{L_{\epsilon}\epsilon}}\left(\ln\dfrac{\abs{x}^2+\epsilon^2}{\abs{x}^2}-\ln\dfrac{L_\epsilon^2+1}{L_\epsilon^2}\right)\Delta F\dif x+O\left(L_\epsilon^4\epsilon^4\right)\\
    =&4\pi\bar F-4\pi F(p_0)-\pi\epsilon^2\ln\left(L_\epsilon^2+1\right)\Delta F(p_0)+O\left(L_\epsilon^2\epsilon^4\right)+O\left(L_\epsilon^4\epsilon^4\right),
\end{align*}
which implies
\begin{align}\label{eq:test-2}
  \int_{\Sigma}U^{(\epsilon)}_k\Delta \left(H_{k-1}+H_{k+1}\right)\dif\mu_{\Sigma}=&-4\pi \left(H_{k-1}+H_{k+1}\right)(p_0)-\pi\epsilon^2\ln\left(L_\epsilon^2+1\right)\Delta \left(H_{k-1}+H_{k+1}\right)(p_0)+O\left(L_\epsilon^4\epsilon^4\right).
\end{align}

We compute
\begin{align*}
    \int_{\Sigma}Fe^{-U^{(\epsilon)}_k}\dif\mu_{\Sigma}=&\int_{\Sigma}Fe^{-H_k-\beta_\epsilon}\dif\mu_{\Sigma}+\int_{B_{2L_\epsilon\epsilon}\setminus B_{L_\epsilon\epsilon}}Fe^{-H_k-\beta_\epsilon}(e^{\eta_\epsilon S}-1)\dif\mu_{\Sigma}\\
    &+\int_{B_{L_\epsilon\epsilon}}F\left(\left(\abs{x}^2+\epsilon^2\right)-\dfrac{L_\epsilon^2+1}{L_\epsilon^2}\abs{x}^2e^{-S}\right)e^{-\hin{\alpha}{x}}\dif\mu_{\Sigma}\\
    =&\int_{\Sigma}Fe^{-H_k-\beta_\epsilon}\dif\mu_{\Sigma}+O\left(L_\epsilon^4\epsilon^4\right).
\end{align*}
Thus, if 
\begin{align*}
    \int_{\Sigma}Fe^{-H_k}\dif\mu_{\Sigma}>0,
\end{align*}
then
\begin{align*}
    \ln\int_{\Sigma}Fe^{-U^{(\epsilon)}_k}\dif\mu_{\Sigma}=&\ln\int_{\Sigma}Fe^{-H_k}\dif\mu_{\Sigma}-\beta_\epsilon+O\left(L_\epsilon^4\epsilon^4\right).
\end{align*}
Therefore,
\begin{align}\label{eq:test-3}
\begin{split}
   &-\rho_{k-1}\ln\int_{\Sigma}h_{k-1}e^{2H_{k-1}-H_{k-2}}e^{-U^{(\epsilon)}_k}\dif\mu_{\Sigma}-\rho_{k+1}\ln\int_{\Sigma}h_{k+1}e^{2H_{k+1}-H_{k+2}}e^{-U^{(\epsilon)}_k}\dif\mu_{\Sigma}\\
   =&-\left(\rho_{k-1}+\rho_{k+1}\right)\left(a+\ln\dfrac{L_{\epsilon}^2+1}{L_{\epsilon}^2}\right)+O\left(L_\epsilon^4\epsilon^4\right).
   \end{split}
\end{align}

We compute
\begin{align*}
    \int_{\Sigma}\abs{\nabla U^{(\epsilon)}_k}^2\dif\mu_{\Sigma}=&-\int_{\Sigma}U^{(\epsilon)}_k\Delta U^{(\epsilon)}_k\dif\mu_{\Sigma}\\
    =&-\int_{\Sigma}H_k\Delta U^{(\epsilon)}_k\dif\mu_{\Sigma}+\int_{ B_{2L_{\epsilon}\epsilon}\setminus B_{L_{\epsilon}\epsilon}}\eta_{\epsilon}S \Delta U^{(\epsilon)}_k\dif x+\int_{ B_{L_{\epsilon}\epsilon}}\left(H_k+\ln\left(\abs{x}^2+\epsilon^2\right)-\hin{\alpha}{x}+\beta_\epsilon\right)\Delta U^{(\epsilon)}_k\dif x\\
    =&-4\pi\bar U^{(\epsilon)}_k-8\pi\ln\epsilon+\int_{ B_{2L_{\epsilon}\epsilon}\setminus B_{L_{\epsilon}\epsilon}}\eta_{\epsilon}S\left(4\pi-\Delta(\eta_\epsilon S)\right)\dif x\\
    &-\int_{ B_{L_{\epsilon}\epsilon}}\left(\ln\dfrac{\abs{x}^2+\epsilon^2}{\abs{x}^2}-\ln\dfrac{L_\epsilon^2+1}{L_\epsilon^2}+S\right)\dfrac{4\epsilon^2}{\left(\abs{x}^2+\epsilon^2\right)^2}\dif x\\
    =&-4\pi\bar U^{(\epsilon)}_k-8\pi\ln\epsilon-\int_{ B_{L_{\epsilon}\epsilon}}\left(\ln\dfrac{\abs{x}^2+\epsilon^2}{\abs{x}^2}-\ln\dfrac{L_\epsilon^2+1}{L_\epsilon^2}+\pi\abs{x}^2\right)\dfrac{4\epsilon^2}{\left(\abs{x}^2+\epsilon^2\right)^2}\dif x+O\left(L_\epsilon^4\epsilon^4\right)\\
    =&-4\pi\bar U^{(\epsilon)}_k-8\pi\ln\epsilon-4\pi\dfrac{L_\epsilon^2}{L_\epsilon^2+1}-4\pi^2\epsilon^2\left(\ln\left(L_\epsilon^2+1\right)-\dfrac{L_\epsilon^2}{L_\epsilon^2+1}\right)+O\left(L_\epsilon^4\epsilon^4\right).
\end{align*}
We obtain
\begin{align*}
    \int_{\Sigma}\abs{\nabla U^{(\epsilon)}_k}^2\dif\mu_{\Sigma}=&4\pi a+4\pi\ln\dfrac{L_{\epsilon}^2+1}{L_{\epsilon}^2}+4\pi^2\epsilon^2\ln\left(L_\epsilon^2+1\right)\\
    &-8\pi\ln\epsilon-4\pi\dfrac{L_\epsilon^2}{L_\epsilon^2+1}-4\pi^2\epsilon^2\left(\ln\left(L_\epsilon^2+1\right)-\dfrac{L_\epsilon^2}{L_\epsilon^2+1}\right)+O\left(L_\epsilon^4\epsilon^4\right),
\end{align*}
which gives
\begin{align}\label{eq:test-4}
    \int_{\Sigma}\abs{\nabla U^{(\epsilon)}_k}^2\dif\mu_{\Sigma}=&4\pi a-8\pi\ln\epsilon+4\pi\ln\dfrac{L_{\epsilon}^2+1}{L_{\epsilon}^2}-4\pi\dfrac{L_\epsilon^2}{L_\epsilon^2+1}+4\pi^2\epsilon^2\dfrac{L_\epsilon^2}{L_\epsilon^2+1}+O\left(L_\epsilon^4\epsilon^4\right).
\end{align}

 We compute
\begin{align*}
    \int_{\Sigma}Fe^{2U^{(\epsilon)}_k}\dif\mu_{\Sigma}
    =&\int_{\Sigma\setminus B_{\delta}}Fe^{2U^{(\epsilon)}_k}\dif\mu_{\Sigma}+\int_{B_\delta\setminus B_{2L_\epsilon\epsilon}}Fe^{2U^{(\epsilon)}_k}\dif\mu_{\Sigma}+\int_{B_{2L_\epsilon\epsilon}\setminus B_{L_\epsilon\epsilon}}Fe^{2U^{(\epsilon)}_k}\dif\mu_{\Sigma}+\int_{B_{L_\epsilon\epsilon}}Fe^{2U^{(\epsilon)}_k}\dif\mu_{\Sigma}\\
   = &\int_{\Sigma\setminus  B_{\delta}}Fe^{2H_k+2\beta_{\epsilon}}\dif\mu_{\Sigma}+\int_{B_{\delta}\setminus B_{2L_\epsilon\epsilon}}Fe^{2H_k+2\beta_{\epsilon}}\dif x\\
   &+\int_{B_{2L_\epsilon\epsilon}\setminus B_{L_\epsilon\epsilon}}Fe^{2H_k+2\beta_{\epsilon}-2\eta_{\epsilon}S}\dif x+\int_{B_{L_\epsilon\epsilon}}\dfrac{1}{\left(\abs{x}^2+\epsilon^2\right)^2}Fe^{2\hin{\alpha}{x}}\dif x\\
   =&\dfrac{L_\epsilon^4}{\left(L_\epsilon^2+1\right)^2}\int_{B_{\delta}\setminus B_{L_\epsilon\epsilon}}\abs{x}^{-4}Fe^{2\hin{\alpha}{x}+2S}\dif x+\int_{B_{L_\epsilon\epsilon}}\dfrac{1}{\left(\abs{x}^2+\epsilon^2\right)^2}Fe^{2\hin{\alpha}{x}}\dif x+O(1)\\
   =&\dfrac{L_\epsilon^4}{\left(L_\epsilon^2+1\right)^2}\int_{L_\epsilon\epsilon}^{\delta}\abs{r}^{-3}\left(2\pi F(p_0)+\dfrac{\pi}{2}F(p_0)\left(\Delta\ln F(p_0)+8\pi+\abs{\nabla\ln F(p_0)+2\alpha}^2\right)r^2\right)\dif r\\
   &+\int_0^{L\epsilon}\dfrac{r}{\left(r^2+\epsilon^2\right)^2}\left(2\pi F(p_0)+\dfrac{\pi}{2}F(p_0)\left(\Delta\ln F(p_0)+\abs{\nabla\ln F(p_0)+2\alpha}^2\right)r^2\right)\dif r+O(1)\\
   =&\dfrac{L_\epsilon^4}{\left(L_\epsilon^2+1\right)^2}\pi F(p_0)\left(\dfrac{1}{L_\epsilon^2\epsilon^2}-2\left(2\pi +\abs{\alpha+\dfrac12\nabla\ln F(p_0)}^2+\dfrac{1}{4}\Delta\ln  F(p_0)\right)\ln\left(L_\epsilon\epsilon\right)\right)\\
   &+\pi F(p_0)\left(\dfrac{L_\epsilon^2}{\left(L_\epsilon^2+1\right)\epsilon^2}+\left(\abs{\alpha+\dfrac12\nabla\ln F(p_0)}^2+\dfrac{1}{4}\Delta \ln F(p_0)\right)\left(\ln\left(L_\epsilon^2+1\right)-\dfrac{L_\epsilon^2}{L_\epsilon^2+1}\right)\right)+O(1).
\end{align*}
Thus
\begin{align*}
    \int_{\Sigma}Fe^{2U^{(\epsilon)}_k}\dif\mu_{\Sigma}
    =&\dfrac{\pi F(p_0)}{\epsilon^2}\left(1-\dfrac{1}{\left(L_\epsilon^2+1\right)^2}-\dfrac{4\pi L_\epsilon^4}{\left(L_\epsilon^2+1\right)^2}\epsilon^2\ln\left(L_\epsilon\epsilon\right)\right)\\
    &+\pi F(p_0)\left(\abs{\alpha+\dfrac12\nabla\ln F(p_0)}^2+\dfrac{1}{4}\Delta \ln F(p_0)\right)\left(\ln\left(L_\epsilon^2+1\right)-\dfrac{L_\epsilon^4}{\left(L_\epsilon^2+1\right)^2}\ln\left(L_\epsilon^2\epsilon^2\right)\right)+O(1).
\end{align*}
We obtain
\begin{align}\label{eq:test-5}
    \ln\int_{\Sigma}Fe^{2U^{(\epsilon)}_k}\dif\mu_{\Sigma}=&\ln\dfrac{\pi F(p_0)}{\epsilon^2}+\left(\abs{\alpha+\dfrac12\nabla\ln F(p_0)}^2+\dfrac{1}{4}\Delta \ln F(p_0)\right)\epsilon^2\ln\left(L_\epsilon^2+1\right)+O\left(L_\epsilon^{-4}\right)+O\left(\epsilon^2\ln\left(L_\epsilon\epsilon\right)\right).
\end{align}

Combining the above computations \eqref{eq:test-0}, \eqref{eq:test-1},  \eqref{eq:test-2}, \eqref{eq:test-3}, \eqref{eq:test-4} and  \eqref{eq:test-5}, we obtain
\begin{align*}
    J(u^{(\epsilon)})=&J_k(G)+\dfrac14\int_{\Sigma}\abs{\nabla\left(H_{k-1}+H_{k+1}\right)}^2\dif\mu_{\Sigma}\\
    &-\left(\rho_{k-1}+\rho_{k+1}\right)\left(a+\ln\dfrac{L_{\epsilon}^2+1}{L_{\epsilon}^2}\right)+O\left(L_\epsilon^4\epsilon^4\right)\\
   &+4\pi a-8\pi\ln\epsilon+4\pi\ln\dfrac{L_{\epsilon}^2+1}{L_{\epsilon}^2}-4\pi\dfrac{L_\epsilon^2}{L_\epsilon^2+1}+4\pi^2\epsilon^2\dfrac{L_\epsilon^2}{L_\epsilon^2+1}+O\left(L_\epsilon^4\epsilon^4\right)\\
   &-4\pi \left(H_{k-1}+H_{k+1}\right)(p_0)-\pi\epsilon^2\ln\left(L_\epsilon^2+1\right)\Delta \left(H_{k-1}+H_{k+1}\right)(p_0)+O\left(L_\epsilon^4\epsilon^4\right)\\
   &+\left(8\pi-\rho_{k-1}-\rho_{k+1}\right)\left(-a-\ln\dfrac{L_{\epsilon}^2+1}{L_{\epsilon}^2}-\pi\epsilon^2\ln\left(L_\epsilon^2+1\right)+O\left(L_\epsilon^4\epsilon^4\right)\right)\\
    &-4\pi\left(\ln\dfrac{\pi F(p_0)}{\epsilon^2}+\left(\abs{\alpha+\dfrac12\nabla\ln F(p_0)}^2+\dfrac{1}{4}\Delta \ln F(p_0)\right)\epsilon^2\ln\left(L_\epsilon^2+1\right)+O\left(L_\epsilon^{-4}\right)+O\left(\epsilon^2\ln\left(L_\epsilon\epsilon\right)\right)\right).
\end{align*}
Here $F=h_ke^{-H_{k-1}-H_{k+1}}$. Therefore,
\begin{align*}
    J(u^{(\epsilon)})=&J_k(G)+\dfrac14\int_{\Sigma}\abs{\nabla\left(H_{k-1}+H_{k+1}\right)}^2\dif\mu_{\Sigma}-4\pi-4\pi\ln\pi-4\pi\left(a+\ln h_k(p_0)\right)\\
    &-\pi\left(\abs{2\alpha-\nabla\left(H_{k-1}+H_{k+1}\right)(p_0)+\nabla\ln h_k(p_0)}^2+\Delta \ln h_k(p_0)+\left(8\pi-\rho_{k-1}-\rho_{k+1}\right)\right)\epsilon^2\ln\left(L_\epsilon^2+1\right)\\
    &+O\left(L_\epsilon^{-4}\right)+O\left(\epsilon^2\ln\left(L_\epsilon\epsilon\right)\right).
\end{align*}
Choose $L_\epsilon$ so that $L_{\epsilon}\epsilon=\frac{1}{\ln\epsilon^{-1}}$. We obtain
\begin{align*}
    J(u^{(\epsilon)})=&J_k(G)+\dfrac14\int_{\Sigma}\abs{\nabla\left(H_{k-1}+H_{k+1}\right)}^2\dif\mu_{\Sigma}-4\pi-4\pi\ln\pi-4\pi\left(a+\ln h_k(p_0)\right)\\
   &-\pi\left(\abs{2\alpha-\nabla\left(H_{k-1}+H_{k+1}\right)(p_0)+\nabla\ln h_k(p_0)}^2+\Delta \ln h_k(p_0)+\left(8\pi-\rho_{k-1}-\rho_{k+1}\right)\right)\epsilon^2\ln\epsilon^{-2}\\
   &+O\left(\epsilon^2\ln(-\ln\epsilon)\right).
\end{align*}
At last, by \eqref{eq:gradient F}, discussions above \autoref{thm:test} and switching back to the metric $g$ on $\Sigma$, we complete the proof of  \autoref{thm:test}.
\end{proof}

Now we can prove the main \autoref{main:thm2} as following.
\begin{proof}
\autoref{main:thm2} is an immediate consequence of \autoref{thm:lower} and \autoref{thm:test}.
\end{proof}

\appendix

\section{A Classification of Toda system}
Let $A=(a_{ij})_{N\times N}$ be the Cartan matrix for $\mathrm{SU}(N+1)$ and $\rho=(\rho_1,\dotsc,\rho_N)\in\mathbb{R}^{N}$. Consider the following system
\begin{align}\label{aeq:toda}
\Delta_{\mathbb{C}} u_i+\sum_{j=1}^Na_{ij}\rho_je^{u_j}=0,\quad \text{in}\ \mathbb{C},\quad\int_{\mathbb{C}}e^{u_i}\dif\mu_{\mathbb{C}}<+\infty,\quad\forall 1\leq i\leq N,
\end{align}
where $\dif\mu_{\mathbb{C}}=\frac{1}{2\sqrt{-1}}\dif z\wedge\dif\bar z$.

\begin{theorem}\label{athm:classification} Every solution of the Toda system \eqref{aeq:toda} can be expressed in terms of a polynomial map $f:\mathbb{C}\To\mathbb{C}^{N+1}, z\mapsto f(z)=\sum_{j=1}^Nv_jz^j/j!$ with $\det(v_0,\dotsc,v_N)=1$: 
\begin{align*}
u_i+\ln\rho_i=\ln\dfrac{4\abs{\Lambda_{i+1}(f)}^2\abs{\Lambda_{i-1}(f)}^2}{\abs{\Lambda_i(f)}^4},\quad\forall 1\leq i\leq N,
\end{align*}
where $\Lambda_{i+1}(f)=f\wedge f^{(1)}\wedge\dotsm\wedge f^{(i)}$ and $f^{(i)}=\frac{\partial^if}{\partial z^i}$. In particular
\begin{align*}
\sum_{j=1}^Na_{ij}\rho_j\int_{\mathbb{C}}e^{u_j}\dif\mu_{\mathbb{C}}=8\pi,\quad\forall 1\leq i\leq N,
\end{align*}
and consequently,
\begin{align*}
    \rho_i\int_{\mathbb{C}}e^{u_i}\dif\mu_{\mathbb{C}}\in 4\pi\mathbb{N}^*,\quad\forall 1\leq i\leq N.
\end{align*}
\end{theorem}
\begin{proof}
For every $\varepsilon>0$, there exists $R>0$ such that
\begin{align*}
\max_{i\in I}\int_{\mathbb{C}\setminus\mathrm{B}_{R}}e^{u_i}\dif\mu_{\mathbb{C}}<\varepsilon.
\end{align*}
By Brezis-Merle's argument \cite{BreMer91uniform}, we conclude that
\begin{align*}
\sup_{\mathbb{C}}\max_{i\in I}u_i\leq C.
\end{align*}
Applying the Green representation formula for $u_i$, we obtain
\begin{align*}
u_i(z)=\dfrac{1}{2\pi}\int_{\mathbb{C}}\ln\dfrac{1+\abs{w}}{\abs{z-w}}\sum_{j=1}^Na_{ij}\rho_je^{u_j(w)}\dif\mu_{\mathbb{C}}(w)+c_i,\quad\forall 1\leq i\leq N.
\end{align*}
The potential analysis (cf. \cite{CheLi95waht,JosWan02classification}) implies
\begin{align*}
u_{i}(z)=-\gamma_i\ln\abs{z}+a_i+O(1/\abs{z}),\quad\text{as}\ z\to\infty,\quad\forall 1\leq i\leq N
\end{align*}
where 
\begin{align*}
\gamma_i=\dfrac{1}{2\pi}\sum_{j=1}^Na_{ij}\rho_j\int_{\mathbb{C}}e^{u_j}\dif\mu_{\mathbb{C}}>2,\quad\forall 1\leq i\leq N.
\end{align*}
Therefore,
\begin{align*}
\dfrac{1}{2\pi}\rho_i\int_{\mathbb{C}}e^{u_i}\dif\mu_{\mathbb{C}}=\sum_{j=1}^Na^{ij}\gamma_j>i(N+1-i),\quad\forall 1\leq i\leq N.
\end{align*}
In particular, $\min_{1\leq i\leq N}\rho_i>0$. Then applying a classification result for Toda system of Jost and Wang \cite{JosWan02classification}, we complete the proof.
\end{proof}



\end{document}